\documentclass{article}

\usepackage{amsmath,amsthm,amsfonts,amssymb,verbatim,dsfont}
\usepackage{hyperref}
\usepackage[letterpaper]{geometry}

\newtheorem{theorem}{Theorem}[section]
\newtheorem{lemma}[theorem]{Lemma}
\newtheorem{corollary}[theorem]{Corollary}
\newtheorem{prop}[theorem]{Proposition}

\theoremstyle{definition}
\newtheorem{definition}[theorem]{Definition}

\theoremstyle{remark}
\newtheorem{remark}[theorem]{Remark}
\newcommand{\abs}[1]{\left\lvert #1 \right\rvert}
\newcommand{\norm}[1]{\left\lVert #1 \right\rVert}
\newcommand{\laplacian}{\Delta}
\newcommand{\grad}{\nabla}
\newcommand{\cross}{\times}
\newcommand{\real}{\mathds{R}}

\newcommand{\Qb}{\widetilde{Q}_b}

\newcommand{\epsRe}{{\epsilon_\text{re}}}
\newcommand{\epsIm}{{\epsilon_\text{im}}}
\newcommand{\zb}{\tilde{\zeta}}

\newcommand{\dy}{\,dy}

\newcommand{\epsNorm}{\int{\abs{\grad_y\epsilon}^2\mu\dy} + \int_{\abs{y}\leq\frac{10}{b}}{\abs{\epsilon}^2e^{-\abs{y}}\dy}}
\newcommand{\epsTildeNorm}{\int{\abs{\grad_y\widetilde{\epsilon}}^2\mu(y)\dy} + \int_{\abs{y}\leq\frac{10}{b}}{\abs{\widetilde{\epsilon}}^2e^{-\abs{y}}\dy}}

\title{Standing Ring Blowup Solutions for Cubic NLS}
\author{Ian Zwiers\footnote{Department of Mathematics, University of Toronto. 40 St. George Street, Toronto, Ontario, Canada M5S 2E4} }
\date{}

\begin{document}
\maketitle

\begin{abstract}
We prove there exist solutions to the focusing cubic nonlinear Schr\"odinger equation on $\real^3$ that blowup on a circle, in the sense of $L^2$ concentration on a ring, bounded $H^1$ norm outside any surrounding toroid, and growth of the global $H^1$ norm with the log-log rate. 

Analogous behaviour occurs in any dimension $N\geq 3$. That is, there exists data on $\real^N$ for which the corresponding evolution by cubic NLS explodes on a set of co-dimension two.
To simplify the exposition, the proof is presented in dimension three, with remarks to indicate the adaptations in higher dimension. 
\end{abstract}

\section{Introduction}

Consider the cubic focusing nonlinear Schr\"odinger equation in dimension three,
\begin{equation}\label{Eqn-NLS}\left\{\begin{aligned}
&iu_t+\laplacian u + u\abs{u}^2 = 0\\
&u(0,x) = u_0 : \real^{3} \to {\mathds C}.
\end{aligned}\right.\end{equation}
This is a canonical model equation arising in physics and engineering, \cite{SulemSulem}. This equation, and other closely related equations, have been the subject of many recent mathematical studies.

Equation (\ref{Eqn-NLS}) is locally wellposed for data $u_0\in H^s(\real^3)$, for any $s \geq \frac{1}{2}$, \cite{Cazenave03}.  Higher regularity persists under local-in-time dynamics, and the maximal time $T_{max}>0$ for which $u\in C\left([0,T_{max}),H^{s}\right)$ is the same for all $s>\frac{1}{2}$, 
 and we have the classic blowup alternative: either $T_{max} = +\infty$ or $\norm{u(t)}_{H^{s}}\to\infty$ as $t\to T_{max}$. 
Evolution by equation (\ref{Eqn-NLS}) preserves:
\begin{align}
\label{ConserveMass}
& 
\int_{\real^{3}}{\abs{u(t,x)}^2\,dx}
=\int{\abs{u_0}^2\,dx} = M[u_0],
&& \text{(mass)}
\\
\label{ConserveEnergy}
&
\int{\abs{\grad_xu(t,x)}^2\,dx} - \frac{1}{2}\int{\abs{u(t,x)}^4\,dx}
= E[u(t,x)] = E[u_0],
&& \text{(energy)}
\\
\label{ConserveMoment}
&
Im\left(\int{\overline{u}(t,x)\grad u(t,x)\,dx}\right) 
= Im\left(\int{\overline{u_0}\grad u_0\,dx}\right).
&& \text{(momentum)}
\end{align}
There are corresponding symmetries. If $u(t,x)$ satisfies (\ref{Eqn-NLS}), then so do the following:
\[\begin{aligned}
u(t,x+x_0) 
	&&& \forall\, x_0 \in \real^{3}
	&& \text{(spatial translation invariance)}\\
u(t+t_0,x) 
	&&& \forall\, t_0\in \real
	&& \text{(time translation invariance)}\\
u(t,x)e^{i\gamma_0} 
	&&& \forall\, \gamma_0\in\real 
	&& \text{(phase invariance)}\\
u(t,x-\beta_0 t)e^{i\frac{\beta_0}{2}\cdot\left(x-\frac{\beta_0}{2}t\right)} 
	&&& \forall\, \beta_0 \in \real^{3} 
	&& \text{(Galilean invariance)}\\
\lambda_0 u(\lambda_0^2t,\lambda_0 x)
	&&& \forall\, \lambda_0 > 0
	&& \text{(scaling invariance)}
\end{aligned}\]
Scaling invariance leaves the $\dot{H}^\frac{1}{2}(\real^3)$ norm of data unchanged and for this reason equation (\ref{Eqn-NLS}) is deemed {\it $H^{\frac{1}{2}}$-critical}. 
Equation (\ref{Eqn-NLS}) has standing wave solutions. The ansatz, $u(t,x) = e^{i t}W(x)$, leads to the elliptic PDE,
\begin{equation}\label{DefnEqn-W}\left\{\begin{aligned}
&\laplacian W - W + W\abs{W}^2 = 0,\\
&\begin{aligned}W(\abs{x}) > 0 && \text{for } x \in \real^3\end{aligned}.
\end{aligned}\right.
\end{equation}
The unique positive radial solution\footnotemark to equation (\ref{DefnEqn-W}) is the {\it ground-state} solution of equation (\ref{Eqn-NLS}).  We reserve the notation $Q$ for the ground-state solution of the {\bf two-dimensional} problem,
\begin{equation}\label{DefnEqn-Q}
\left\{\begin{aligned}
	&\laplacian_{\real^2} Q -Q + Q\abs{Q}^2 = 0,\\
	&\begin{aligned}Q(\abs{y}) > 0 && \text{for } y \in \real^2\end{aligned}.
\end{aligned}\right.
\end{equation}
\footnotetext{
The classic proof is in \cite{Coffman-Uniqueness3DCubicGroundState-72}. For a more general proof, including other dimensions, see \cite{Weinstein-NLSSharpInterpolation-82}, \cite{BerestyckiLions-83} and \cite{Kwong-Uniqueness-89}. 
For a concise overview of these results, see \cite[Appendix B]{Tao06}. 
}
Recently it was shown, \cite{DHR-Scattering3DCubic-08}, that solutions to equation (\ref{Eqn-NLS}) exist for all time, and scatter, if,
\begin{equation}\label{Eqn-DHR-Result}\begin{aligned}
M[u_0]\,E[u_0] < M[W]\,E[W] \text{ and } \norm{u_0}_{L^2}\norm{\grad u_0}_{L^2} < \norm{W}_{L^2}\norm{\grad W}_{L^2}.
\end{aligned}\end{equation}
Negative energy data $u_0\in H^1$ lead to blow up in finite time if either radially symmetric or with finite variance, $u_0 \in \Sigma = H^1\cap\{f : \abs{x}f(x)\in L^2\}$, \cite{OgawaTsutsumi-NegEnerBlowupForEnergySubcrit-91}. 
By adjusting the quadratic phase of negative-energy data, one can produce examples of blowup solutions with arbitrary energy, \cite[Remark 6.5.9]{Cazenave03}. Further sufficient conditions for blowup based on the virial identity are known, \cite{HolmerPlatteRoudenko-09}.
As a companion to equation (\ref{Eqn-DHR-Result}), \cite{HolmerRoudenko-SharpCondition3DCubicNLS-08} and \cite{HolmerRoudenko-DivInfVar3DCubic-10} show that if,
\[
\begin{aligned}
M[u_0]\,E[u_0] < M[W]\,E[W] \text{ and } \norm{u_0}_{L^2}\norm{\grad u_0}_{L^2} > \norm{W}_{L^2}\norm{\grad W}_{L^2},
\end{aligned}
\]
then the solution either breaks down in finite time or is unbounded in $H^1$ as $t\to\infty$.
More generally, since equation (\ref{Eqn-NLS}) is $H^1$-subcritical, local wellposedness and the scaling symmetry prove that all solutions in $H^1$ that blow up in finite time must obey the {\it scaling lower bound}, 
\[
\norm{u(t)}_{H^1} \geq \frac{C}{\left(T_{max}-t\right)^\frac{1}{4}}.
\]
Alternatively, the scaling lower bound can be established through energy conservation, \cite{CW-CauchyProblemHs-90}.
Numerics suggest self-similar solutions that blowup at this rate may exist \cite{SulemSulem}. Asymptotics also suggest there may be radially symmetric solutions that focus on a sphere, as that sphere collapses into a point \cite{FGW-SingularRingNumerics-2007, HolmerRoudenko-OnBlowup3DCubic-07}. The growth of $H^1$ in that case appears to be $\left(T_{max}-t\right)^{-\frac{1}{3}}$.
Recently it was shown, \cite{MR-critNormRadialL2Super-07}, that radially symmetric solutions in $H^1$ that blow up in finite time must also blowup in the critical norm, according to,
\[
\norm{u(t)}_{\dot{H}^\frac{1}{2}} \geq \abs{\log(T_{max}-t)}^C.
\]
Since equation (\ref{Eqn-NLS}) does not satisfy the pseudo-conformal symmetry, there is no explicit closed-form blowup solution based on $W$. Indeed, the present work constructs solutions of equation (\ref{Eqn-NLS}) with precise blowup rate.
Lastly, for data of finite variance $u_0\in\Sigma$, there is an integral upper bound, \cite{Merle-LimitSolutionsAtBlowup-89}, on the blowup rate, 
\[\begin{aligned}
\int_0^{T_{max}}{\norm{u(t)}_{\dot{H}^1}^\mu\,dt} < +\infty
&& \text{ for }
&& 0 \leq \mu < 1.
\end{aligned}\]

\begin{remark}[Higher Dimensions]\label{Remark-HigherDimEqn}
We will refer to the cubic nonlinear Schr\"odinger equation in dimension $N$,
\begin{equation}\label{Eqn-NLS-Ndim}
\left\{\begin{aligned}
&iu_t+\laplacian u + u\abs{u}^2 = 0\\
&u(0,x) = u_0 : \real^{N} \to {\mathds C}.
\end{aligned}\right.\end{equation}
This equation is $\dot{H}^{\frac{N}{2}-1}$-critical, and is locally wellposed for data $u_0\in H^s$, for any $s > \frac{N}{2}-1$, with the classic blowup alternative, \cite{Cazenave03}. For data $u_0 \in H^{s'}$, for some $s' > \frac{N}{2}$, higher regularity persists and, as with equation (\ref{Eqn-NLS}), the maximal time $T_{max} > 0$ for which $u\in C\left([0,T_{max}),H^s\right)$ is the same for all $s > \frac{N}{2}-1$.
\end{remark}

\begin{remark}[Notation]\label{Remark-Notation}
We use
$f\lesssim g$, $f \gtrsim g$ and $f\approx g$ to denote that there exist constants $C_1, C_2>0$ such that $f \leq C_1 g$, $f \geq C_2 g$ and $C_2 g \leq f \leq C_1 g$, respectively. 
Notation $f \sim g$ is used in more casual discussion to symbolize $f$ and $g$ are of the same order. We will use $\delta(\alpha)$ to denote any function of $\alpha$ with the property $\delta(\alpha) \rightarrow 0$ as $\alpha \rightarrow 0$. The exact form of $\delta$ will depend on the context. Frequently, we use the operator,
\[\begin{aligned}
\Lambda = 1 + y\cdot\grad_y, && \text{ where } y \text{ is a two-dimensional variable.}
\end{aligned}\]
Note that for $f, g\in L^2(\real^2)$ we have, $\left(\Lambda f, g\right) = -\left(f,\Lambda g\right)$.
\end{remark}

\subsection{Statement of Result}

For all $N\geq 3$ we introduce cylindrical coordinates $x = (r,z,\theta) \in [0,\infty)\times\real\times S^{N-2}$ for $x\in\real^N$. We refer to functions that are symmetric with respect to $\theta$ as cylindrically symmetric, and we let $H^s_{cyl}(\real^N)$ denote the cylindrically symmetric subset of $H^s$.

\begin{theorem}[Main Result]\label{Thm-MainResult}
For all $N\geq 3$, there exists a  set of cylindrically symmetric data $u_0 \in {\mathcal P}$,  open in $H^{N}_{cyl}(\real^{N})$ for which the corresponding solution $u(t)$ of (\ref{Eqn-NLS}) has maximum (forward) lifetime $0 < T_{max} < +\infty$ and exhibits the following properties:
\begin{itemize}
\item \underline{\rm Concentration}:

There exist parameters $\lambda(t)>0$, $r(t) > 0$, $z(t) \in \real^{N-2}$ and $\gamma(t) \in \real$, with convergence,
\begin{equation}\label{Thm-MainResult-RZ}\begin{aligned}
\left(r(t),z(t)\right) \longrightarrow \left(r_{max},z_{max}\right)
&& \text{ as } t \rightarrow T_{max}
&& \text{ with } r_{max} \sim 1,
\end{aligned}\end{equation}
such that there is the following strong convergence in $L^2(\real^{N})$,
\begin{equation}\label{Thm-MainResult-L2}\begin{aligned}
u(t,r,z,\theta) - \frac{1}{\lambda(t)}Q\left(\frac{(r,z)-(r(t),z(t))}{\lambda(t)}\right)e^{-i\gamma(t)}
\longrightarrow
u^*(r,z,\theta),
&& \text{ as } t \rightarrow T_{max}.
\end{aligned}\end{equation}

\item \underline{\rm Persistent regularity away from singular ring}:

For any $R > 0$,
\begin{equation}\label{Thm-MainResult-SingOnRing}
u^* \in H^\frac{N-1}{2}\left(\abs{(r,z)-(r_{max},z_{max})}>R\right).
\end{equation}

\item \underline{\rm Log-log blowup rate}:

The solution leaves $H^1$ at the log-log rate,
\begin{equation}\label{Thm-MainResult-LogLog}\begin{aligned}
\frac{
	\left(\frac{\log\abs{\log T_{max}-t}}{T_{max}-t}\right)^\frac{1}{2} 
	}{
\norm{u(t)}_{H^1(\real^{3})} 
}
\longrightarrow \frac{\sqrt{2\pi}}{\norm{Q}_{L^2(\real^2)}}
&& \text{ as } t \rightarrow T_{max}.
\end{aligned}\end{equation} 
Moreover, the higher-order norm behaves appropriately,
\begin{equation}\label{Thm-MainResult-LogLogHigher}\begin{aligned}
\frac{
	\norm{u(t)}_{H^{N}} 
	}{
\norm{u(t)}_{H^1}^{N}\,\log\norm{u(t)}_{H^1}
}
\longrightarrow 0
&& \text{ as } t \rightarrow T_{max}.
\end{aligned}\end{equation}

\end{itemize}
\end{theorem}
\begin{remark}[Proof for $N>3$]
To simplify the exposition, we only present proof in the case $N=3$.
The adaptations for higher dimensions are indicated by Remarks \ref{Remark-HigherDimEqn}, \ref{Remark-HigherDim-Hypo}, \ref{Remark-HigherDim-StrategyUseHCrit} and \ref{Remark-HigherDim-NthEnergy}.
\end{remark}

\begin{remark}[Nature of $u^*$]
\label{Remark-ProfileH1Away}
For the $L^2$-critical problem it is known that the residual profile $u^*$ is not in $H^1$, \cite{MR-ProfilesQuantization-05}.  Indeed, equation (\ref{Thm-MainResult-SingOnRing}) fails for $R=0$. See Remark \ref{Remark-ProfileNonH1} for further comment.
\end{remark}

\subsection{Brief Heuristic}
In cylindrical coordinates we write the Laplacian,
\begin{equation}\label{Eqn-Laplacian}
\laplacian_x = \partial_r^2 + \partial_z^2 + \frac{\partial_r}{r}.
\end{equation}
Suppose that a solution to equation (\ref{Eqn-NLS}) is cylindrically symmetric and concentrated near the ring $(r,z) \sim (r_0,z_0)$.  Then for an appropriately small $\lambda_0 > 0$ we may write,
\begin{equation}\label{Eqn-BriefHeuristic1}
u(t,x) = \frac{1}{\lambda_0}v\left(\frac{t}{\lambda_0^2},\frac{(r,z)-(r_0,z_0)}{\lambda_0}\right),
\end{equation}
where the function $v$ is supported on the half-plane $(r,z)\in \left\lbrack -\frac{r_0}{\lambda_0},\infty\right) \times \real$. 
Neglect that our parameters may vary in time. After changing coordinates, $v$ satisfies,
\begin{equation}\label{Eqn-BriefHeuristic2}
\begin{aligned}
i\partial_{s}v + \laplacian_y v
	+ \frac{\lambda_0}{r}\partial_{y_1} v
	+ v\abs{v}^2 =0
&&\text{ where }s=\frac{t}{\lambda_0^2}, 
&& y = \frac{(r,z)-(r_0,z_0)}{\lambda_0}.
\end{aligned}
\end{equation}
For a solution $u(t,x)$ tightly concentrated near $(r_0,z_0)$, we might choose $\lambda_0 \ll 1$ as the width of the window of concentration.  Then, $\frac{\lambda_0}{r}\partial_{y_1}v$ can be taken as a lower order correction, and the evolution of $v$ is essentially that of two-dimensional cubic NLS.  If $v(s,y)$ falls within the robust log-log blowup dynamic, we would expect the concentration near $(r_0,z_0)$ to increase, and for the lower order correction in equation (\ref{Eqn-BriefHeuristic2}) to become less relevant.

We may identify our main challenge: to ensure persistence of sufficient decay in the original variables near $r=0$ such that conditions there mimic those at infinity during a log-log blowup of two-dimensional cubic NLS.

\subsection{Similar Results}
\begin{theorem}[Standing Ring Blowups for Quintic NLS in 2D \cite{R06}]\label{Thm-Raph}
There exists a set of radially symmetric data $u_0$, open in $H^{1}_{rad}(\real^2)$, for which the corresponding solution to, 
\begin{equation}\label{Eqn-QuinticNLS2D}
\left\{\begin{aligned}
&iu_t+\laplacian u + u\abs{u}^4 = 0,\\
&u_0 \in H^1(\real^2),\end{aligned}\right.
\end{equation}
has maximum lifetime $T_{max}<+\infty$ and exhibits log-log blowup on a ring of fixed radius; that is, there holds appropriate equivalents of (\ref{Thm-MainResult-RZ}), (\ref{Thm-MainResult-L2}), (\ref{Thm-MainResult-SingOnRing}) and (\ref{Thm-MainResult-LogLog}).
\end{theorem}

The following result and Theorem \ref{Thm-MainResult} were developed simultaneously.\footnotemark
\footnotetext{After Theorem \ref{Thm-Raph}, the idea to consider other $H^\frac{1}{2}$-critical problems was first suggested to the author's thesis advisor by Justin Holmer and Svetlana Roudenko in private conversation.  }
\begin{theorem}[Standing Ring Blowups for Cubic NLS in 3D \cite{HolmerRoudenko-3DCubicCircle-10}]\label{Thm-HolmerRoudenko}
There exists a set of cylindrically symmetric data $u_0$, open in $H^1_{cyl}(\real^3)$, for which the corresponding solution to (\ref{Eqn-NLS}) has maximum lifetime $T_{max} < +\infty$ and exhibits log-log blowup on a ring of fixed radius. In particular, (\ref{Thm-MainResult-RZ}) and (\ref{Thm-MainResult-L2}) hold, as does (\ref{Thm-MainResult-SingOnRing}) at the level of $H^\frac{1}{2}$ regularity.

\end{theorem}
The methods of \cite{R06} do not extend to prove either Theorem \ref{Thm-MainResult} or Theorem \ref{Thm-HolmerRoudenko}.  At issue is the initial localized gain of regularity, \cite[equation (4.137)]{R06}. Calculating $\frac{d}{dt}\norm{D^\nu\left(\chi u\right)}_{L^2}^2$ results in a nonlocal term, due to the lack of Leibniz rule - see (\ref{Strategy2-KatoStrich}). In this place \cite{R06} relies upon the Strauss radial embedding, and \cite{HolmerRoudenko-3DCubicCircle-10} use an elegant microlocal estimate to smooth the nonlocal part. We will avoid such problems through the use of higher regularity. 
Note that Theorem \ref{Thm-HolmerRoudenko} describes a larger class of data, with lower regularity, than does Theorem \ref{Thm-MainResult} in the case $N=3$.

\begin{theorem}[Standing Ring Blowups for Quintic NLS \cite{RaphaelSzeftel-StandingRingNDimQuintic-08}]\label{Thm-RaphSzef}
For all $N\geq 3$, there exists a set of radially-symmetric data $u_0$, open in $H^{N}_{rad}(\real^{N})$, for which the corresponding solution to:
\begin{equation}\label{Eqn-QuinticNLS}
\left\{\begin{aligned}
&iu_t + \laplacian u + u\abs{u}^4 = 0,\\
&u_0 \in H^{N}(\real^{N}),\end{aligned}\right.
\end{equation}
has maximum lifetime $T_{max}<+\infty$ and exhibits log-log blowup on a fixed ring $r\sim 1$; that is, there holds appropriate equivalents of (\ref{Thm-MainResult-L2}), (\ref{Thm-MainResult-RZ}), (\ref{Thm-MainResult-LogLog}) and (\ref{Thm-MainResult-SingOnRing}).
\end{theorem}
It is essential to the proof of both Theorem \ref{Thm-MainResult} and Theorem \ref{Thm-RaphSzef} that the behaviour of a higher-order norm can be controlled in terms of the (understood) behaviour of the $H^1$ norm. Again, the fundamental obstruction is $H^1 \not\hookrightarrow L^\infty$, which manifests in two ways:
\begin{enumerate}
\item Inability to bootstrap a global $H^3$ control.  Without radial symmetry, one cannot make the estimate, $\frac{d}{dt}\norm{u}_{H^3}^2 \lesssim \norm{u}_{H^3}^{2-\delta}\norm{u}_{H^1}^{3\delta+2}$ - see the third term of \cite[equation (44)]{RaphaelSzeftel-StandingRingNDimQuintic-08}. 

\item Inability to achieve an initial localized gain of regularity.
As with the arguments of \cite{R06},  an analogue of \cite[equation (63)]{RaphaelSzeftel-StandingRingNDimQuintic-08} cannot be established. To bypass this issue, we require tighter control of higher regularity than is achieved in \cite{RaphaelSzeftel-StandingRingNDimQuintic-08}, and a Brezis-Gallou\"et type argument, Lemma \ref{Lemma-NdimEndpointSobolev}. 
\end{enumerate}
We apply three new strategies:
\begin{enumerate}
\item Consider the singular and residual portions of the solution separately during the most difficult calculations. 

This is recognition that the higher order norms of the central profile of the solution scale exactly with the global $H^1$ norm.  While the worst hypothesized $H^3$ behaviour may only be attributed to the residual portion of the solution, for the same portion we possess superior $H^1$ control. We arrange to only evaluate the $H^3$ norm of the residual when there is also a factor of the $H^1$ norm of the same.  
By making a delicate hypothesis of the global $H^3$ behaviour, only slightly worse than scaling, we can arrange for product of the $H^1$ and $H^3$ norms of the residual to be slightly better than scaling.

\item Integrate the modulation parameters with more accuracy to allow a more refined hypothesis of the global $H^3$ behaviour.  See Lemma \ref{Lemma-lambdaIntegral}.

\item Apply a modified Brezis-Gallou\"et argument to the residual portion of the solution to estimate $\norm{u}_{L^\infty}$ much closer to scaling than any two-dimensional Sobolev embedding. This allows us to complete a Gronwall argument and prove a localized gain of regularity outside the support of some bootstrap hypotheses.
\end{enumerate}

\begin{remark}[Theorem \ref{Thm-MainResult} for other nonlinearities]
Other $H^\frac{1}{2}$-critical equations, such as $\textrm{NLS}^{1+\frac{4}{3}}(\real^4)$ and $\textrm{NLS}^{2}(\real^5)$ suffer from non-smooth nonlinearities.  The author is unaware of local wellposedness in any space more regular than $H^1$ in these cases.  Non-smooth nonlinearities also prohibit iteration of interior regularity arguments, so that in these cases Lemma \ref{Lemma-UnprovenProperty} below is only true for $s=3$.   
\end{remark}

\subsection{Acknowledgements}
\label{Subsec-Ack}
The author would like to thank Pierre Rapha\"el and J\'er\'emie Szeftel for correspondence and for sharing the draft of their work, \cite{RaphaelSzeftel-StandingRingNDimQuintic-08}, and Justin Holmer and Svetlana Roudenko for conversations. The author is also indebted to his advisor, Jim Colliander, for his attention and guidance. This work is part of the author's Ph.D. thesis.

\section{Setting of the Bootstrap}
\label{Section-SetBootstrap}
In this chapter we identify data concentrated near the ring $(r,z)\sim (1,0)$, according to properties we will later show persist. Our subsequent arguments are based on the two-dimensional $L^2$-critical log-log blowup dynamic, which has been comprehensively investigated by Merle \& Rapha\"el, \cite{MR-BlowupDynamic-05, MR-SharpUpperL2Critical-03, MR-UniversalityBlowupL2Critical-04, R-StabilityOfLogLog-05, MR-ProfilesQuantization-05, MR-SharpLowerL2Critical-06}.  This work stems from those detailed studies.

\begin{definition}[Fundamental Properties of Almost Self-similar Profiles]\label{Defn-FundSelfSimProps}
For all $b> 0$ sufficiently small, there exists a solution $\widetilde{Q}_b \in H^1(\real^2)$ of,
\[
\laplacian \widetilde{Q}_b -\widetilde{Q}_b + ib\Lambda\widetilde{Q}_b + \widetilde{Q}_b\abs{\widetilde{Q}_b}^2 = -\Psi_b,
\]
that is supported on the ball of radius $\frac{2}{\abs{b}}$ and converges to $Q$ in $C^3(\real^2)$ as $b\rightarrow 0$. Profiles $\widetilde{Q}_b$ have mass the order of $b^2$ larger than $Q$, and energy of the order $e^{-\frac{C}{b}}$. The truncation error $\Psi_b$ acts as the source of the linear radiation,
\[
\laplacian \zeta_b - \zeta_b + ib\Lambda\zeta_b = \Psi_b.
\]
Radiation $\zeta_b$ is not in $L^2$, with the precise decay rate $\Gamma_b = \lim_{\abs{y} \to +\infty}\abs{y}\abs{\zeta_b}^2$. It is known $\Gamma_b \sim e^{-\frac{\pi}{b}}$, and it is this decay property linked to the central profile $\widetilde{Q}_b$ that is 
responsible for the log-log rate of the two-dimensional $L^2$-critical problem. For our analysis, we will truncate $\zeta_b$ near $\abs{y} \sim e^{+\frac{a}{b}}$ for a small fixed parameter $a$. See Section \ref{Subsec-SelfSim} for details.
\end{definition}
\begin{lemma}[Smoothness of $\widetilde{Q}_b$]
\label{Lemma-UnprovenProperty}
The almost self-similar profiles $\widetilde{Q}_b$ are smooth. For any $s\geq 3$,
\begin{equation}\label{Eqn-QbDerivProperty}\begin{aligned}
\limsup_{b\to 0} \norm{\widetilde{Q}_b}_{C^{s}(\real^2)} < +\infty
&& \text{ and }
&& \limsup_{b\to 0}\norm{\widetilde{Q}_b}_{H^{s}(\real^2)} < +\infty.
\end{aligned}\end{equation}
\end{lemma}


\subsection{Geometric Decomposition}
\label{Subsec-GeoDecomp}

In place of $(r,z,\theta) \in \real^{3}$ we change coordinates to the rescaled half-plane,
\begin{equation}\label{DefnEqn-y} 
y = \left(\frac{(r,z) - (r_0,z_0)}{\lambda_0}\right) \in \left\lbrack-\frac{r_0}{\lambda_0},+\infty\right)\times\real. 
\end{equation}
Fixed parameters $r_0$,$z_0$,$\lambda_0$ will later be replaced by $r(t)$, $z(t)$, and $\lambda(t)$. It will be clear from the context. Note the measure due to cylindrical symmetry, $dx = \lambda_0 \mu_{\lambda_0,r_0}(y)\,dy$ is given by,
\begin{equation}\label{DefnEqn-mu}
\mu_{\lambda_0,r_0}(y) = 2\pi\left(\lambda_0y_1+r_0\right){\mathds 1}_{y_1 \geq -\frac{r_0}{\lambda_0}}.
\end{equation}
We will shortly hypothesize parameters of the decomposition in such a way that the support of both $\widetilde{Q}_b$ and $\widetilde{\zeta}_b$ are well away from the boundary of domain (\ref{DefnEqn-y}). For convenience we will omit the constant factor $2\pi$ and approximate $\mu(y)\sim 1$ on this region - see (\ref{Hypo1-consequenceForMu}). Integrals in $y$ can then be seen as taken over all of $\real^2$, and regular integration by parts applies. 
Any integral that cannot be localized in this way will be treated separately, and very carefully.

To begin, we modulate suitable cylindrically symmetric data as if two-dimensional, \cite[Lemma 2]{R06}.
\begin{lemma}[Existence of Geometric Decomposition at a Fixed Time]\label{Lemma-GeoDecompFixed}
Suppose that $v\in H^1_{cyl}(\real^{3})$ may be written in the form, 
\begin{equation}\label{Lemma-GeoDecompFixed-Eqn}
v(r,z,\theta) = 
	\frac{1}{\lambda_v}\left(\widetilde{Q}_{b_v}+\epsilon_v\right)
	\left(\frac{(r,z)-(r_v,z_v)}{\lambda}\right)e^{-i\gamma_v}
\end{equation}
for some parameters $\lambda_v,b_v,r_v > 0$ and $\gamma_v,z_v \in \real$ such that,
\begin{equation}\label{Lemma-GeoDecompFixed-eps}
\int{\abs{\grad_y\epsilon_v}^2\mu_{\lambda_v,r_v}(y)\,dy}
	+\int_{\abs{y}\leq\frac{10}{b_v}}{\abs{\epsilon_v}^2e^{-\abs{y}}\,dy}
	< \Gamma_{b_v}^\frac{1}{2},
\end{equation}
\begin{equation}\label{Lemma-GeoDecompFixed-params}\begin{aligned}
\abs{(r_v,z_v) - (1,0)} < \frac{1}{3}
&& \text{ and }
&& 10\lambda_v < b_v < \alpha^*. 
\end{aligned}\end{equation}
Then there are nearby parameters: $\lambda_0,b_0,r_0 > 0$ and $\gamma_0,z_0 \in \real$ with,
\begin{equation}\label{Lemma-GeoDecompFixed-nearby}
\abs{b_0-b_v}
+\abs{\frac{\lambda_0}{\lambda_v} - 1} 
+\frac{\abs{(r_0,z_0)-(r_v,z_v)}}{\lambda_v} 
\leq \Gamma_{b_0}^\frac{1}{5},
\end{equation}
such that the corresponding $\epsilon_0$,
\begin{equation}\label{Lemma-GeoDecompFixed-defnEps}
\epsilon_0(y) = \lambda_0\,v\left(\lambda_0 y + (r_0,z_0)\right)\,e^{i\gamma_0} - \widetilde{Q}_{b_0},
\end{equation}
satisfies the two-dimensional orthogonality conditions\footnotemark:
\begin{equation}\label{Lemma-GeoDecompFixed-orthog}
Re\left(\epsilon_0,\abs{y}^2\widetilde{Q}_{b_0}\right) = 
Re\left(\epsilon_0,y\widetilde{Q}_{b_0}\right) = 
Im\left(\epsilon_0,\Lambda^2\widetilde{Q}_{b_0}\right) =
Im\left(\epsilon_0,\Lambda\widetilde{Q}_{b_0}\right) = 0.
\end{equation}
\end{lemma}
\footnotetext{
The decomposition of \cite{MR-SharpUpperL2Critical-03} used slightly different orthogonality conditions.  Equation (\ref{Lemma-GeoDecompFixed-orthog}) is the decomposition introduced \cite[Lemma 6]{MR-UniversalityBlowupL2Critical-04}, which leads to a better estimate on the phase parameter than was achieved in \cite{MR-SharpUpperL2Critical-03}.
}

We envisage a singular ring contained within a toroid, the complement of which is contiguous and includes both the origin and infinity.  Denote two smooth cutoff functions,
\begin{equation}\label{DefnEqn-Chi}\begin{aligned}
\chi(r,z,\theta) = &\left\{\begin{aligned}
		1 && \text{ for } \abs{(r,z) - (1,0)} \geq \frac{2}{3}\\
		0 && \text{ for } \abs{(r,z) - (1,0)} \leq \frac{1}{3}
	\end{aligned}\right.
&& \text{and},\\
\chi_0(r,z,\theta) = &\left\{\begin{aligned}
		1 && \text{ for } \abs{(r,z) - (1,0)} \geq \frac{1}{7}\\
		0 && \text{ for } \abs{(r,z) - (1,0)} \leq \frac{1}{8}.
	\end{aligned}\right.
&&
\end{aligned}\end{equation}
In Chapter \ref{Section-BootAtInfty} we will define a further series of cutoff functions $\psi$ and $\varphi$; these further definitions will be supported on bounded regions where $\chi_0 \equiv 1$. 
We now describe the initial data for our bootstrap procedure. 
\begin{definition}[Description of Initial Data ${\mathcal P}$]\label{Defn-DataP}

For $\alpha^* > 0$, a constant to be determined, let the set ${\mathcal P}(\alpha^*)$ comprise cylindrically symmetric $u_0 \in H^{3}_{cyl}(\real^{3})$ that may be written of the form,
\begin{equation}\label{Defn-DataP1-GeoDecomp}\begin{aligned}
u_0(r,z) &= \frac{1}{\lambda_0}\left(\widetilde{Q}_{b_0}+\epsilon_0\right)
	\left(\frac{(r,z) - (r_0,z_0)}{\lambda_0}\right)e^{-i\gamma_0}\\
	&=\frac{1}{\lambda_0}\left(\widetilde{Q}_{b_0}\right)
	\left(\frac{(r,z) - (r_0,z_0)}{\lambda_0}\right)e^{-i\gamma_0} 
	+ \tilde{u}_0(r,z),
\end{aligned}\end{equation}
in a way that satisfies the following two sets of conditions:

	\begin{description}
	\item[{Singularity of a log-log nature:}]
	
	\item[C1.1] {\it 'Radial' profile focused near a singular ring,}
	\begin{equation}\label{DataP1-rz}
	\abs{(r_0,z_0) - (1,0)} < \alpha^*.
	\end{equation}

	\item[C1.2] {\it 'Radial' profile is close to $Q$ near the singular ring,}
	
	Profile $\widetilde{Q}_b$ have nearly the mass of $Q$, and account for nearly all mass globally,
	\begin{equation}\label{DataP1-mass}\begin{aligned}
	0 < b_0 + \norm{\tilde{u}_0}_{L^2(\real^{3})} < \alpha^*,
	\end{aligned}\end{equation} 
	and $\epsilon_0(y)$ both satisfies the orthogonality conditions, 
	\begin{equation}\label{DataP1-orthog}\begin{aligned}
	Re\left(\epsilon_0,\abs{y}^2\widetilde{Q}_{b_0}\right)
	= Re\left(\epsilon_0,y\widetilde{Q}_{b_0}\right)
	= Im\left(\epsilon_0,\Lambda^2\widetilde{Q}_{b_0}\right)
	= Im\left(\epsilon_0,\Lambda\widetilde{Q}_{b_0}\right) = 0,
	\end{aligned}\end{equation}
	and the smallness condition,
	\begin{equation}\label{DataP1-epsilon}
	\int{ \abs{\grad_y\epsilon_0(y)}^2 \mu_{\lambda_0,r_0}(y)\,dy}
		+\int_{\abs{y}\leq\frac{10}{b_0}}{\abs{\epsilon_0(y)}^2e^{-\abs{y}}\,dy}
		<\Gamma_{b_0}^\frac{6}{7}.
	\end{equation}
	
	\item[C1.3] {\it Conformal and scaling parameters are consistent with log-log blowup speed,}
	\begin{equation}\label{DataP1-loglog}
	e^{-e^\frac{2\pi}{b_0}} < \lambda_0 < e^{-e^{\frac{\pi}{2}\frac{1}{b_0}}}.
	\end{equation}
	
	\item[C1.4] {\it Energy and localized momentum are normalized,}
	\begin{equation}\label{DataP1-enerMoment}
	\lambda_0^2\abs{E_0} 
		+\lambda_0\abs{Im\left( \int{\grad_x\psi^{(x)}\cdot\grad_xu_0\overline{u}_0}\right)}
		< \Gamma_{b_0}^{10},
	\end{equation}
	where $\psi^{(x)}$ is a smooth cylindrically symmetric `cutoff' function with,
	\begin{equation}\label{DefnEqn-psi-x}
	\psi^{(x)}(r,z,\theta) = \left\{\begin{aligned}
		r + z && \text{ for } \abs{(r,z)-(1,0)}\leq\frac{1}{2},\\
		0 && \text{ for } \abs{(r,z)-(1,0)}\geq\frac{3}{4}.
		\end{aligned}\right.
	\end{equation}

	\item[{Regularity away from the singularity:}]
	\item[C2.1] {\it Scaling-consistent $\dot{H}^{3}$ norm,}
	\begin{equation}\label{DataP2-Hnk}
	\norm{u_0}_{H^{3}(\real^{3})} < \frac{C_{\widetilde{Q}}}{\lambda_0^{2+k}},
	\end{equation}
	where $C_{\widetilde{Q}}$ is a universal constant due to Lemma \ref{Lemma-UnprovenProperty},
	\item[C2.2] {\it Strong hierarchy of regularity away from the singular ring,}
	\begin{equation}\label{DataP2-Hlower}
	\norm{\chi_0u_0}_{H^{3-\kappa}(\real^{3})} < \frac{1}{\lambda_0^{3-2\kappa}},
	\end{equation}
	for each half integer $\frac{1}{2} \leq \kappa \leq \frac{3}{2}$, and,

	\item[C2.3] {\it Vanishing lower-order norms away from the singular ring,}
	\begin{equation}\label{DataP2-Hcrit}
	\norm{\chi_0u_0}_{H^1(\real^{3})} < \left(\alpha^*\right)^\frac{1}{2}.
	\end{equation}

	\end{description}
\end{definition}
Lemma \ref{Lemma-GeoDecompFixed} guarantees that ${\mathcal P}(\alpha^*)$ is open in $H^1_{cyl}\cap H^3_{cyl}$.  See Appendix \ref{Appendix}
for proof that ${\mathcal P}(\alpha^*)$ is nonempty.
For the remainder of this paper, fix an arbitrary $u_0 \in {\mathcal P}(\alpha^*)$. Let $u(t)$ denote the evolution by equation (\ref{Eqn-NLS}), with maximum (forward) lifetime, $0 < T_{max} \leq +\infty$.

Continuous evolution in $H^3(\real^3)$ implies the same in $H^1(\real^3)$, and so by Lemma \ref{Lemma-GeoDecompFixed} there is some $0<T_{geo} \leq T_{max}$, (which may be assumed maximal,) for which the geometric decomposition of Lemma \ref{Lemma-GeoDecompFixed} may be applied on $[0,T_{geo})$. 
There exist unique continuous functions $\lambda(t),b(t),r(t) : [0,T_{geo}) \to (0,\infty)$, and $\gamma(t),z(t) : [0,T_{geo}) \to \real$, with the expected initial values, where,
\begin{equation}\label{Eqn-GeoDecomp}\begin{aligned}
u(t,r,z,\theta) = 
	&\frac{1}{\lambda(t)}\left(\widetilde{Q}_{b(t)}+\epsilon(t)\right)
	\left(\frac{(r,z)-(r(t),z(t))}{\lambda}\right)e^{-i\gamma(t)}\\
	=&\frac{1}{\lambda(t)}\left(\widetilde{Q}_{b(t)}\right)
	\left(\frac{(r,z)-(r(t),z(t))}{\lambda}\right)e^{-i\gamma(t)}
	+\tilde{u}(t,r,z,\theta),
\end{aligned}\end{equation}
such that $\epsilon(t,y)$ satisfies the two-dimensional orthogonality conditions:
\begin{align}
&&Re\left(\epsilon(t),\abs{y}^2\widetilde{Q}_{b(t)}\right) 
	&= 0,\label{Eqn-GeoDecomp-OrthogReY2}\\
&&Re\left(\epsilon(t),y\widetilde{Q}_{b(t)}\right) 
	&= 0,\label{Eqn-GeoDecomp-OrthogReY}\\
&&Im\left(\epsilon(t),\Lambda^2\widetilde{Q}_{b(t)}\right) 
	&= 0,\label{Eqn-GeoDecomp-OrthogImL2}\\
&\text{and}&Im\left(\epsilon(t),\Lambda\widetilde{Q}_{b(t)}\right) 
	&= 0.\label{Eqn-GeoDecomp-OrthogImL}
\end{align}
We may now define the rescaled time,
\begin{equation}\label{DefnEqn-s}\begin{aligned}
s(t) = \int_0^t{\frac{1}{\lambda^2(\tau)}d\,\tau} + s_0
&& \text{ where }
&& s_0 = e^\frac{3\pi}{4b_0}.
\end{aligned}\end{equation}
Also denote $s_1 = s(T_{hyp})$, for $T_{hyp}$ due to the forthcoming Definition \ref{Defn-Hypo}. 
The choice of $s_0$ will prove convenient in Section \ref{Subsec-LocalVirial}.

\begin{remark}[Fixed Parameters] 
To aid the reader, we provide a brief summary of the various parameters that will be introduced, in the order one might ultimately determine them:
\begin{itemize}
\item $\eta$ and $a$: Parameters that determine the cutoff shape of $\widetilde{Q}_b$ and $\widetilde{\zeta}_b$, see equations (\ref{DefnEqn-Rb}) and (\ref{DefnEqn-A}) respectively. The value of $a>0$ is assumed sufficiently small for the proof of Lemma \ref{Lemma-Lowerbound}, relative to some universal constant. The earlier proof of Lemma \ref{Lemma-LyapounovFunc} is conditioned on the eventual choice of $\eta < \frac{a}{C_0}$, for another universal constant $C_0 > 0$, see equation (\ref{Proof-LyapounovFunc-eqn2}). These choices affect the class of initial data ${\mathcal P}$ both by setting the profiles $\widetilde{Q}_b$ and by forcing an upper bound on the value of $\alpha^*$.

\item $\sigma_1$, $\sigma_2$ and $\sigma_3$: 
Parameters in the statements of Lemma \ref{Lemma-lambdaIntegralCrude}, Lemma \ref{Lemma-lambdaIntegral}, and Corollary \ref{Corollary-lambdaIntegral2version2}. Their value is chosen (repeatedly) according to circumstance.

\item $\sigma_4$: Defined for Lemma \ref{Lemma-NSobolev}. Value is uniform over all $m>0$ sufficiently small.

\item $\delta_5$: A fixed arbitrary universal constant $0<\delta_5\ll 1$, used in the proof of Lemma \ref{Lemma-NSobolev}. 

\item $m'$: Existence of $m'<m$ with particular properties in a key assertion of Proposition \ref{Prop-Improv}, below.  Some particular value $m'\in(m-\frac{\sigma_4}{2},m)$ is chosen for the proof of Lemma \ref{Lemma-ControlledHnk}.

\item $\sigma_5$: Parameter in the statement of Lemma \ref{Lemma-NdimEndpointSobolev}. Value is fixed for the proof of Lemma \ref{Lemma-Annular12}.

\item $\sigma_6$: Defined for Lemma \ref{Lemma-AnnularN2}.  Value is uniform over all $m>0$ sufficiently small.

\item $m$:  Fixed constant $m>0$ features in the bootstrap hypotheses of Definition \ref{Defn-Hypo}, below. For the purpose of various proofs in Chapter \ref{Section-BootAtInfty}, $m$ will be assumed sufficiently small.
The exact value of $m$ may be determined apriori, and will affect the class of initial data ${\mathcal P}$ by forcing an upper bound on the value of $\alpha^*$.  

\item $\alpha^*$: Fixed constant $\alpha^*>0$ is determined last.  For the purpose of various proofs throughout this paper, $\alpha^*$ will be assumed sufficiently small.  
\end{itemize}
\end{remark}

The following bootstrap hypotheses are possible due to our choice of data in ${\mathcal P}$.
\begin{definition}[Time $T_{hyp} > 0$ \& Bootstrap Hypotheses]\label{Defn-Hypo}
Let $0 < T_{hyp} \leq T_{max}$ be the maximum time such that for all $t\in[0,T_{hyp})$ the following two sets of conditions hold:
\begin{description}
	\item[Singularity remains of a log-log nature:]
	
	\item[H1.1] {\it Profile remains focused near a singular ring,}
	\begin{equation}\label{Hypo1-rz}
	\abs{(r(t),z(t))-(1,0)} < \left(\alpha^*\right)^\frac{1}{2}.
	\end{equation}
	
	\item[H1.2] {\it Profile remains close to Q near the singular ring,}
	\begin{equation}\label{Hypo1-b}\begin{aligned}
	0 < b(t) + \norm{\tilde{u}(t)}_{L^2(\real^{3})} < \left(\alpha^*\right)^\frac{1}{10}.
	\end{aligned}\end{equation}
	
	\begin{equation}\label{Hypo1-epsilon}\begin{aligned}
	\int{\abs{\grad_y\epsilon(t)}^2\mu_{\lambda(t),r(t)}(y)\,dy} + \int_{\abs{y}\leq\frac{10}{b(t)}}{\abs{\epsilon(t)}^2e^{-\abs{y}}\,dy} \leq \Gamma_{b(t)}^\frac{3}{4}.
	\end{aligned}\end{equation}
	
	\item[H1.3] {\it Conformal and scaling parameters remain consistent with log-log blowup speed,}
	\begin{equation}\label{Hypo1-loglog}\begin{aligned}
	\frac{\pi}{10}\frac{1}{\log s} < b(s) < \frac{10\pi}{\log s}, &&
	e^{-e^\frac{10\pi}{b(s)}} < \lambda(s) < e^{-e^{\frac{\pi}{10}\frac{1}{b(s)}}}.
	\end{aligned}\end{equation}

	\item[H1.4] {\it Energy and localized momentum remain normalized,}
	\begin{equation}\label{Hypo1-enerMoment}
	\lambda^2(t)\abs{E_0} + \lambda(t)\abs{Im\left(\int{\grad\psi^{(x)}\cdot\grad u(t)\overline{u}(t)}\right)}
	< \Gamma_{b(t)}^2.
	\end{equation}

	\item[H1.5] {\it Norm growths are almost monotonic,}
	\begin{equation}\label{Hypo1-almostMonotony}\begin{aligned}
	\forall s_a \leq s_b \in[s_0,s_1], && \lambda(s_b) \leq 3\lambda(s_a).
	\end{aligned}\end{equation}

	\item[Regularity away from the singularity persists:]
	\item[H2.1] {\it Growth of $\dot{H}^3$ is near scaling,}
	\begin{equation}\label{Hypo2-Hnk}
	\norm{u(t)}_{H^{3}(\real^{3})} < \frac{e^{+\frac{m}{b(t)}}}{\lambda^{3}(t)},
	\end{equation}

	\item[H2.2] {\it Strong hierarchy of regularity away from the singular ring persists,}
	\begin{equation}\label{Hypo2-Hlower}
	\norm{\chi u(t)}_{H^{3-\kappa}} < \frac{e^{+(1+\kappa)\frac{m}{b(t)}}}{\lambda^{3-2\kappa}(t)},
	\end{equation}
	for each half integer $\frac{1}{2} \leq \kappa < \frac{3}{2}$,
	\begin{equation}\label{Hypo2-HlowerN2}
	\norm{\chi u(t)}_{H^\frac{3}{2}} < e^{+\frac{2m + \pi}{b(t)}},
	\end{equation}
	and,
	\item[H2.3] {\it Lower-order norms away from the singular ring remain bounded,}
	\begin{equation}\label{Hypo2-Hcrit}
	\norm{\chi u(t)}_{H^1} < \left(\alpha^*\right)^\frac{1}{10}.
	\end{equation}

	\end{description}
\end{definition}
\noindent An important consequence of {\bf H1.2}, {\bf H1.3},
and the forthcoming estimate on $\Gamma_b$, (\ref{Eqn-GammaBEstimate}), is that,
\begin{equation}\label{Hypo1-consequenceForLambdaGamma}
\lambda(t) < e^{-e^\frac{\pi}{10b(t)}} < \Gamma_{b(t)}^{10}.
\end{equation}
Therefore as a consequence of {\bf H1.1}
and the forthcoming definition of $A$, (\ref{DefnEqn-A}),
\begin{equation}\label{Hypo1-consequenceForMu}\begin{aligned}
\frac{2}{3} \leq \mu(y) \leq \frac{3}{2}
&&\forall \abs{y}\leq 5 A(t).
\end{aligned}\end{equation}
The region $\abs{y}\leq 5A(t)$ is exceptionally wide, encompassing the support of both the central profile $\widetilde{Q}_b$ and the associated radiation $\widetilde{\zeta}_b$. 

\begin{remark}[Geometric decomposition is well defined]
Hypotheses {\bf H1} easily satisfy the conditions of Lemma \ref{Lemma-GeoDecompFixed}, ensuring that $T_{hyp}\leq T_{geo}$ and the unique geometric decomposition (\ref{Eqn-GeoDecomp}) is available.
\end{remark}

\begin{remark}[Higher dimensions]\label{Remark-HigherDim-Hypo}
For the higher dimensional case, extend the hypotheses {\bf H2} to include:
\begin{itemize}
\item Controlled growth of $\norm{u(t)}_{H^{N}}$, in place of {\bf H2.1}
\item Equation (\ref{Hypo2-Hlower}) for each half integer $\frac{1}{2} \leq \kappa < \frac{N}{2}$, and phrase (\ref{Hypo2-HlowerN2}) with respect to $H^\frac{N}{2}$.
\item Bounded behaviour in $H^{\frac{N-1}{2}}$ away from the singular ring, in place of {\bf H2.3}.
\end{itemize}
The obvious changes to the class of initial data should also be made.
\end{remark}

\begin{prop}[Bootstrap Conclusion]\label{Prop-Improv}
For $\alpha^*>0$ sufficiently small, hypotheses (\ref{Hypo1-rz}) through (\ref{Hypo2-Hcrit}) are not sharp. There exists $m'<m$ such that for all $t\in[0,T_{hyp})$:
\begin{description}
\item[I1.1]
	\begin{equation}\label{Improv1-rz}
	\abs{(r(t),z(t))-(1,0)} < \left(\alpha^*\right)^\frac{2}{3},
	\end{equation}

	\item[I1.2]
	\begin{equation}\label{Improv1-b}\begin{aligned}
	0 < b(t) + \norm{\tilde{u}(t)}_{L^2(\real^{3})} < \left(\alpha^*\right)^\frac{1}{5},
	\end{aligned}\end{equation}

	\begin{equation}\label{Improv1-epsilon}\begin{aligned}
	\int{\abs{\grad_y\epsilon(t)}^2\mu_{\lambda(t),r(t)}(y)\,dy} + \int_{\abs{y}\leq\frac{10}{b(t)}}{\abs{\epsilon(t)}^2e^{-\abs{y}}\,dy} \leq \Gamma_{b(t)}^\frac{4}{5},
	\end{aligned}\end{equation}

	\item[I1.3]
	\begin{equation}\label{Improv1-loglog}\begin{aligned}
	\frac{\pi}{5}\frac{1}{\log s} < b(s) < \frac{5\pi}{\log s}, &&
	e^{-e^\frac{5\pi}{b(t)}} < \lambda(t) < e^{-e^{\frac{\pi}{5}\frac{1}{b(t)}}},
	\end{aligned}\end{equation}

	\item[I1.4]
	\begin{equation}\label{Improv1-enerMoment}
	\lambda^2(t)\abs{E_0} + \lambda(t)\abs{Im\left(\int{\grad\psi^{(x)}\cdot\grad u(t)\overline{u}(t)}\right)}
	< \Gamma_{b(t)}^4,
	\end{equation}

	\item[I1.5]
	\begin{equation}\label{Improv1-almostMonotony}\begin{aligned}
	\forall s_a \leq s_b \in[s_0,s_1], && \lambda(s_b) \leq 2\lambda(s_a).
	\end{aligned}\end{equation}


	\item[I2.1]
	\begin{equation}\label{Improv2-Hnk}
	\norm{u(t)}_{H^{3}(\real^{3})} < \frac{e^{+\frac{m'}{b(t)}}}{\lambda^{3}(t)},
	\end{equation}

	\item[I2.2]
	\begin{equation}\label{Improv2-Hlower}
	\norm{\chi u(t)}_{H^{3-\kappa}} < \frac{e^{+(1+\kappa)\frac{m'}{b(t)}}}{\lambda^{3-2\kappa}(t)},
	\end{equation}
	for each half integer $\frac{1}{2} \leq \kappa < \frac{3}{2}$, 
	\begin{equation}\label{Improv2-HlowerN2}
	\norm{\chi u(t)}_{H^\frac{3}{2}} < e^{+\frac{2m' + \pi}{b(t)}},
	\end{equation}
and,
	\item[I2.3]
	\begin{equation}\label{Improv2-Hcrit}
	\norm{\chi u(t)}_{H^1} < \left(\alpha^*\right)^\frac{1}{5}.
	\end{equation}

\end{description}
As a consequence, $T_{hyp} = T_{max}$.
\end{prop}

\subsection{Strategy of Proof: the log-log argument}
\label{Subsec-StrategyPart1}

We will establish statements {\bf I1} in Chapter \ref{Section-BootLoglog} using the arguments of \cite{MR-SharpUpperL2Critical-03} and \cite{MR-SharpLowerL2Critical-06}. Here we identify the main challenge in maintaining the log-log dynamics.
As with all modulation arguments, we seek to reduce the question of blowup to a finite-dimensional ODE dynamic for the parameters. This is only possible due to the algebraic structure associated with $Q$.  
Recall the operator $\Lambda = 1 +y\cdot\grad_y$, which one might recognize from either the argument $E(Q) = 0$:
\begin{equation}\label{Strategy1-energy}
\left(0,\Lambda(Q)\right) = \left(\laplacian Q-Q+Q\abs{Q}^2,\Lambda(Q)\right) = -2E(Q), 
\end{equation}
or from the Pohozaev identity:
\begin{equation}\label{Strategy1-pohozaev}
\left(0,\Lambda(v)\right) 
	= Re\left(iv_s+\laplacian_y v + v\abs{v}^2,\Lambda(v)\right)
	= \frac{1}{2}\frac{d}{ds}Im\int{v\,y\cdot\grad\overline{v}\,dy} - 2E(v),
\end{equation}
which is also a consequence of formally calculating the virial identity, $\frac{d^2}{d^2s}\int{\abs{y}^2\abs{v}^2\,dy}$.
Substitution of (\ref{Eqn-GeoDecomp}) into (\ref{Eqn-NLS}) will produce an equation for $\epsilon$. Ignoring the distinction between $Q$ and $\Qb$, the terms linear in $\epsilon$ are, $i\partial_s\epsilon + L(\epsilon)$, where $L$ is the linearized propagator near $Q$. As a matrix on real and imaginary parts,
\begin{equation}\label{DefnEqn-L}\begin{aligned}
L(\epsilon) = 
	\left.
	\begin{bmatrix}0&L_-\\-L_+&0\end{bmatrix}
	\begin{bmatrix}\epsRe\\i\, \epsIm\end{bmatrix}
	\right.
&& \text{ with: }
&& \begin{aligned}
		L_+ &= -\laplacian + 1 - 3Q^2\\
		L_- &= -\laplacian + 1 - Q^2
	\end{aligned}
\end{aligned}
\end{equation}
Weinstein noted \cite{Weinstein85} the following:
\begin{equation}\label{Strategy1-Algebra}\begin{aligned}
L_-\left(\abs{y}^2Q\right) = -2\Lambda Q, &&
L_-\left(yQ\right) = -2\grad Q, && \text{ and } &&
L_+\left(\Lambda Q\right) = -2Q.
\end{aligned}\end{equation}
These algebraic properties are the inspiration of the orthogonality conditions so that, by taking appropriate inner products of the $\epsilon$-equation, linear terms cancel. For example, the imaginary part of the inner product with $\abs{y}^2Q$ has no linear terms due to conditions (\ref{Eqn-GeoDecomp-OrthogReY2}) and (\ref{Eqn-GeoDecomp-OrthogImL}). The imaginary part of the inner product with $yQ$ is controlled by momentum.

The most fruitful calculation is when we take the real part of an inner product of the $\epsilon$-equation with $\Lambda Q$. This is of course a localized version of equation (\ref{Strategy1-pohozaev}). We substitute conservation of energy to eliminate the linear term, $2\,Re\left(\epsilon,Q\right)$, which is due to the third identity of equation (\ref{Strategy1-Algebra}). The remaining terms quadratic in $\epsilon$ form the following, 
\begin{equation}\label{DefnEqn-H}
\begin{aligned}
H(\epsilon,\epsilon) = 
	\left.
	\begin{bmatrix}{\mathcal L}_\text{re}&0\\0&{\mathcal L}_\text{im}\end{bmatrix}
	\begin{bmatrix}\epsRe\\i\, \epsIm\end{bmatrix}
	\cdot
	\begin{bmatrix}\epsRe\\-i\, \epsIm\end{bmatrix}
	\right.
&& \text{ with: }
&& \begin{aligned}
		{\mathcal L}_\text{re} &= -\laplacian + 3Qy\cdot\grad Q\\
		{\mathcal L}_\text{im} &= -\laplacian + Qy\cdot\grad Q
	\end{aligned}
\end{aligned}
\end{equation}
Operator $H(\epsilon,\epsilon)$ is the derivative with respect to scaling of the conserved energy of the linear flow. It has coercivity properties that mirror the stability of $Q$:
\begin{prop}[Spectral Property]
\label{Prop-SpectralProperty}
There exists a universal constant ${\delta}_0 > 0$ such that $\forall v \in H^1$:
\begin{equation}\label{Eqn-2DSpectral}\begin{aligned}
H(v,v) \geq &{\delta}_0
	\left(\int_{y\in\real^2}{\abs{\grad_yv}^2} + 
		\int_{y\in\real^2}{\abs{v^2}e^{-\abs{y}}}\right)\\
	&-\frac{1}{{\delta}_0}\left(
		\begin{aligned}&Re\left(v,Q\right) + Re\left(v,\Lambda Q\right) +
		Re\left(v,yQ\right)\\ &+ Im\left(v,\Lambda Q\right) + 
		Im\left(v,\Lambda^2Q\right) + Im\left(v,\grad Q\right)\end{aligned}
	\right)^2.
\end{aligned}\end{equation}
\end{prop} 
The two-dimensional Spectral Property as stated here has a numerical proof \cite{FMR-ProofOfSpectralProperty-06}\footnotemark.
\footnotetext{
The numerical proof is given for the $L^2$-critical problem in dimensions $N=2,3,4$ and nonlinearity $u\abs{u}^\frac{4}{N}$.  Proof for dimension $N=1$ is explicit, \cite[Proposition 2]{MR-BlowupDynamic-05}.
}
Assuming we can ensure $H$ is coercive, the goal is to prove the {\it local virial identity},
\begin{equation}\label{Strategy1-localVirial}
b_s \geq \delta_1\left(\epsNorm\right) - \Gamma_b^{1-C\eta}.
\end{equation}
To prove (\ref{Strategy1-localVirial}) using the Spectral Property requires we control the contribution from all other terms of the conservation of energy.  In particular, we must establish the {\bf non-local} control,
\begin{equation}\label{Eqn-HRTemp1}
\int_{\real^{2}}{\abs{\epsilon(y)}^4\mu(y)} \ll \int_{\real^2}{\abs{\grad_y\epsilon}^2\mu(y)}.
\end{equation}
This is our main challenge.

The local virial identity (\ref{Strategy1-localVirial}) is a satisfactory control for $\epsilon$ at times where $b_s < 0$. However, our argument is based on approximating the central profile of the solution, therefore we cannot expect monotinicity in our modulation parameters. Including the radiation $\zb$ to better approximate the central profile, repeating the local virial calculation, and taking into account the mass flux leaving the support of the radiation, Merle \& Rapha\"el discovered a Lyapounov functional, \cite{MR-SharpLowerL2Critical-06}.  
It is remarkable that we may approximate the Lyapounov functional very precisely in terms of a positive multiple of a norm of $\epsilon$.  The functional is then used to bridge the control of $\epsilon$ between times where $b_s<0$. The approximation here is achieved through the conservation of energy, and involves equation (\ref{Eqn-HRTemp1}) a second time.

Regarding (\ref{Eqn-HRTemp1}), change variables,
\begin{equation}\label{Eqn-HRTemp2}
\int_{\real^{2}}{\abs{\epsilon(y)}^4\mu(y)} 
	= \lambda^2\int_{\real^{3}}{\abs{\widetilde{u}}^4}
	= \lambda^2\int_{\real^{3}}{\abs{\chi\widetilde{u}}^4} 
		+\lambda^2\int_{\real^{3}}{(1-\chi^4)\abs{\widetilde{u}}^4}.
\end{equation}
Since the support of $\chi$ includes the origin, we must apply three-dimensional Sobolev to that term,
\[
\norm{\chi \widetilde{u}}_{L^4(\real^3)}^4 \lesssim \norm{\chi u}_{H^\frac{1}{2}(\real^3)}^2\norm{\chi u}_{H^1(\real^3)}^2.
\]
Changing variables again, we observe that to achieve (\ref{Eqn-HRTemp1}) requires at least that, $\norm{\chi u}_{H^\frac{1}{2}(\real^3)} \ll 1$.
\begin{remark}[Higher dimensions]\label{Remark-HigherDim-StrategyUseHCrit}
In general the Sobolev embedding into $L^4$ involves the critical norm $H^{\frac{N-2}{2}}$,
\[
\norm{\chi u}_{L^4(\real^N)}^4 \lesssim \norm{\chi u}_{H^\frac{N-2}{2}(\real^N)}^2\norm{\chi u}_{H^1(\real^N)}^2.
\]
\end{remark}

\subsection{Strategy of Proof: persistence of regularity}
\label{Subsec-StrategyPart2}
Once we have established the log-log nature of our blowup, we expect powers of $\frac{1}{\lambda}$ to be as integrable in time as powers of $\sqrt{\frac{\log\abs{\log(T_{max}-t)}}{T_{max}-t}}$. Indeed, as noted in \cite{RaphaelSzeftel-StandingRingNDimQuintic-08},
\begin{equation}\label{Strategy2-lambdaIntegral}
\int_0^t{\frac{d\tau}{\lambda^{2\nu+1}(\tau)}} \leq C(\epsilon)\frac{1}{\lambda^{2\nu-1+\epsilon}(t)},
\end{equation}
for any $\delta > 0$ and $\nu \geq \frac{1}{2}$. As we explain below, our argument requires more care. We prove that,
\begin{equation}\begin{aligned}\label{Strategy2-lambdaIntegralImproved}
\int_0^t{\frac{\abs{\log\lambda}^{\sigma^*}}{\lambda^{2\nu+1}(\tau)}\,d\tau}
\leq C(\sigma^*,\sigma,\nu)  \frac{\abs{\log\lambda}^{\sigma}(t)}{\lambda^{\nu-1}(t)},
\end{aligned}\end{equation}
for any $\sigma^*<\sigma$, of either sign, and $\nu \geq \frac{1}{2}$.
The arguments of Chapter \ref{Section-BootAtInfty}, to establish statements {\bf I2}, proceed in three stages. 

\noindent {\it Control of $\norm{u}_{H^3}$}

We explicitly calculate $\frac{d}{dt}\norm{\grad^3u}_{L^2}^2$ and seek to estimate the resulting error terms separately in two regions of space.

First, away from the singularity, on the truly three-dimensional region that includes the origin. Here the estimates are simpler, due to hypotheses {\bf H2.2} and {\bf H2.3}.

Second, on a toroidal region that includes the singular ring. This region requires more delicacy, and we split the solution into the rescaled almost self-similar profile, and $\widetilde{u}$.
Since $\Qb$ is smooth, the higher-order norms scale exactly with $\frac{1}{\lambda}$. In particular,
\begin{equation}\label{Strategy2-scalingQb}
\norm{\frac{1}{\lambda}\Qb(y)}_{H^3(\real^3)} \leq \frac{C(\Qb)}{\lambda^3(t)},
\end{equation}
where the constant is uniform for all $b$ sufficiently small - see Lemma \ref{Lemma-UnprovenProperty}. Note that equation (\ref{Strategy2-scalingQb}) is better than {\bf H2.1}.
For terms in $\widetilde{u}$, note that the $H^1$ norm is better than $\frac{1}{\lambda}$ due to {\bf H1.3}.  By assuming $m>0$ is sufficiently small, we use this superior $H^1$ control to offset our use of {\bf H2.1}.  We prove that,
\[
\abs{\frac{d}{dt}\norm{u}_{H^3}^2} \lesssim \frac{1}{\lambda^8} + \frac{e^{-\frac{\sigma_4}{b}}}{\lambda^2}\norm{u}_{H^3}^2.
\]
To prove {\bf I2.1}, we use equation (\ref{Strategy2-lambdaIntegralImproved}) to integrate carefully.

\noindent {\it Initial regularity improvement}

Let $\psi^A$ be a smooth cutoff function that covers the support of $\grad \chi$ - this is a toroidal shell that acts as an interface between the singular dynamics and the truly three-dimensional dynamics. We hope for any control of $\norm{\psi^A u}_{H^\nu}$ that is better than an interpolation of {\bf H2.1}.  Calculate $\frac{d}{dt}\norm{\psi^Au}_{H^\nu}^2$ directly from equation (\ref{Eqn-NLS}) and integrate in time. The result is effectively Kato's smoothing effect and a Strichartz estimate,
\begin{equation}\label{Strategy2-KatoStrich}
\norm{\psi^A u}_{L^\infty_tH^\nu}^2 
	\lesssim \norm{\psi^B u}_{L^2_tH^{\nu+\frac{1}{2}}}^2 
	+ \int_0^t\abs{\int{D^\nu\left(\psi^Au\abs{u}^2\right)
		\, D^\nu\left(\psi^A\overline{u}\right)}}.
\end{equation}
where $\psi^B$ is some other cutoff function with slightly larger support.

Due to equation (\ref{Strategy2-lambdaIntegralImproved}), we see that the term in $H^{\nu+\frac{1}{2}}$ is infact of the order $\frac{1}{\lambda^{2\left(\nu-\frac{1}{2}\right)}}$. This is exactly the sort of control we want, however the nonlinear term of (\ref{Strategy2-KatoStrich}) is uncooperative.

Since, $H^1(\real^2) \not\hookrightarrow L^\infty(\real^2)$, we cannot apply an $L^\infty$ norm without breaking scaling.  
This is unlike the arguments of, \cite{R06, RaphaelSzeftel-StandingRingNDimQuintic-08}, where Strauss' radial interpolation inequality allows exactly such an embedding.
To estimate the nonlinear term of (\ref{Strategy2-KatoStrich}), we prove a modified Brezis-Gallou\"et estimate that does not break scaling `too-badly'.  It is here that the form of hypothesis {\bf H2.1} is used delicately.

\noindent{\it Iterated Smoothing}

The next stage is to prove {\bf I2.2} and {\bf I2.3} hold on the support of $\grad \chi$. We iterate the argument of equation (\ref{Strategy2-KatoStrich}), in half-integer steps, beginning with $\nu = 3-\frac{1}{2}$, and introducing a new cutoff with smaller support each time.  Due to the initial regularity improvement, it is possible to handle the nonlinear term of (\ref{Strategy2-KatoStrich}) systematically, and at the same order as the term in $H^{\nu+\frac{1}{2}}$.  Due to integration (\ref{Strategy2-lambdaIntegralImproved}), at each stage we may smooth (almost) a half-derivative farther, relative to scaling, than was proved in the previous stage.  After three iterates, we find that $\norm{\psi^C u}_{H^\frac{3}{2}}$ is (almost) order-zero in $\frac{1}{\lambda}$. The final iterate proves $\norm{\psi^D u}_{H^1}$ is constant.

To complete the proof of {\bf I2.2} and {\bf I2.3}, we repeat the iteration scheme for $\chi u$.  The combination of hypotheses {\bf H2.2} and {\bf H2.3} with the results of the first iteration make the second iteration substantially simpler. 

\section{Proof of Log-log Singular Behaviour}
\label{Section-BootLoglog}

In this chapter we will prove that properties {\bf I1.1} through {\bf I1.5} are a consequence of hypotheses {\bf H1.1} through {\bf H1.5} and the bound,
\begin{equation}\label{Hypo2-HcritReinterpreted}
\norm{\chi u(t)}_{H^\frac{1}{2}} < \left(\alpha^*\right)^\frac{1}{10},
\end{equation}
which is a particular consequence of {\bf H2.3}.

\subsection{Almost Self-similar Profiles}
\label{Subsec-SelfSim}

Forthcoming parameter $\eta > 0$ is universal, sufficiently small, and will be determined in Section \ref{Subsec-Lyapounov}. For $b \neq 0$ let,
\begin{equation}\label{DefnEqn-Rb}\begin{aligned}
R_b = \frac{2}{b}\sqrt{1-\eta} && \text{ and } && R_b^- = R_b\sqrt{1-\eta},
\end{aligned}\end{equation}
and let $\phi_b$ denote a radially symmetric cutoff function with,
\begin{equation}\label{DefnEqn-phib}\begin{aligned}
\phi_b(y) = \left\{\begin{aligned} 
	1 && \text{ for } &\abs{y} \leq R_b^-\\
	0 && \text{ for } &\abs{y} \geq R_b
	\end{aligned}\right.
&& \text{ and }
&& \abs{\grad\phi_b}_{L^\infty} + \abs{\laplacian\phi_b}_{L^\infty} \to 0 \text{ as } \abs{b} \to 0.
\end{aligned}\end{equation}
The following result was original shown in \cite[Prop 1]{MR-SharpUpperL2Critical-03}. The refined cutoff, with parameter $\eta$, is introduced in \cite[Prop 8 and 9]{MR-UniversalityBlowupL2Critical-04}.
\begin{prop}[Localized Self-Similar Profiles]\label{Prop-SelfSimilar}

For all $\eta > 0$ sufficiently small there exists positive $b^*(\eta)$ and $\delta(\eta)$ such that for all $\abs{b} < b^*(\eta)$ there exists a unique radial solution $Q_b$ to,
\begin{equation}\label{DefnEqn-Qb}
\left\{\begin{aligned}
&\laplacian Q_b - Q_b + ib\Lambda Q_b + Q_b\abs{Q_b}^2 = 0,\\
&\begin{aligned}
	P_b = Q_b e^{i\frac{b\abs{y}^2}{4}} > 0 
	&& \text{ for } y\in[0,R_b),
	\end{aligned}\\
&\begin{aligned}
	\abs{Q_b(0)-Q(0)} < \delta(\eta), 
	&& Q_b(R_b) = 0.
	\end{aligned}
\end{aligned}\right.
\end{equation}
The truncation to $\abs{y} < \frac{2}{b}$, $\widetilde{Q}_b(y) = Q_b(y)\phi_b(y)$, satisfies,
\begin{equation}\label{DefnEqn-TildeQb}
\laplacian\widetilde{Q}_b -\widetilde{Q}_b + ib\Lambda\widetilde{Q}_b + \widetilde{Q}_b\abs{\widetilde{Q}_b}^2 = - \Psi_b,
\end{equation}
with the explicit error term,
\begin{equation}\label{DefnEqn-Psib}
-\Psi_b = Q_b\laplacian\phi_b + 2\grad\phi_b\cdot\grad Q_b + ibQ_by\cdot\grad\phi_b+\left(\phi_b^3 - \phi_b\right)Q_b\abs{Q_b}^2.
\end{equation}
Moreover, $\widetilde{Q}_b$ satisfies the following properties:
\begin{itemize}

\item {\it Uniform closeness to the ground state}:
\begin{equation}\label{Eqn-Qb_closeToQ}\begin{aligned}
\norm{e^{C\abs{y}}\left(\widetilde{Q}_b-Q\right)}_{C^{3}} \rightarrow 0 
&& \text{ as } 
&& b \rightarrow 0.
\end{aligned}\end{equation}

\item {\it Derivative with respect to $b$}:
\begin{equation}\label{Eqn-Qb_by_b}\begin{aligned}
\norm{e^{C\abs{y}}
	\left(\frac{\partial}{\partial b}\widetilde{Q}_b + i\frac{\abs{y}^2}{4}Q\right)}_{C^2} \rightarrow 0 
&& \text{ as }
&& b \rightarrow 0.
\end{aligned}\end{equation}

\item {\it Supercritical mass}:
\begin{equation}\label{Eqn-Qb_by_bSquared}\begin{aligned}
\left. \frac{d}{d(b^2)}\left(\int{\abs{\widetilde{Q}_b}^2}\right)\right|_{b^2=0} = d_0
&& \text{ with }
&& 0 < d_0 < +\infty.
\end{aligned}\end{equation}
\end{itemize}

As a consequence of (\ref{Eqn-Qb_closeToQ}), for any polynomial $P(y)$ and $k=0,1$,
\begin{equation}\label{Eqn-PsibEstimate}
\abs{P(y)\grad^k\Psi_b}_{L^\infty} \leq e^{-\frac{C(P)}{\abs{b}}}.
\end{equation}

In particular, energy and momentum are degenerate,
\begin{equation}\label{Eqn-Qb_enerAndMoment}\begin{aligned}
\abs{E\left(\widetilde{Q}_b\right)} \leq e^{-(1-C\eta)\frac{\pi}{\abs{b}}}
&& \text{ and }
&& Im\left(\int{\grad_y\widetilde{Q}_b\,\overline{\widetilde{Q}_b}}\right) = 0.
\end{aligned}\end{equation}

\end{prop}

The linearized Schr\"odinger operator near $\widetilde{Q}_b$ is, 
$M\begin{bmatrix}v\\iw\end{bmatrix} = M_+(v,w) + iM_-(v,w)$,
with,
\begin{align}\label{DefnEqn-M}
M_+(v,w) =& 
	-\laplacian_y v + v
	-\left(\frac{\widetilde{Q}_b^2}{\abs{\widetilde{Q}_b}^2}+2\right)
		\abs{\widetilde{Q}_b}^2v
	-Im(\widetilde{Q}_b^2)\,w, \\
M_-(v,w) =& 
	-\laplacian_y w + w
	-\left(2-\frac{\widetilde{Q}_b^2}{\abs{\widetilde{Q}_b}^2}\right)
		\abs{\widetilde{Q}_b}^2w
	-Im(\widetilde{Q}_b^2)\,v. 
\end{align}
As with $L$, equation (\ref{DefnEqn-L}), there is an associated bilinear operator,
\begin{equation}\label{DefnEqn-Hb}
H_b(\epsilon,\epsilon) = H(\epsilon,\epsilon) + \widetilde{H}_b(\epsilon,\epsilon),
\end{equation}
where $H(\epsilon,\epsilon)$ is the usual form (\ref{DefnEqn-H}) associated with $L$. The correction term may be written,
\begin{equation}\label{DefnEqn-HbTilde}
\widetilde{H}_b(\epsilon,\epsilon) = \int{V_{11}\epsRe^2}+\int{V_{12}\epsRe\epsIm}+\int{V_{22}\epsIm^2},
\end{equation}
for well-localized potentials built on $\widetilde{Q}_b$, $Q$ and $y\cdot\grad$ - see \cite[Appendix C]{MR-UniversalityBlowupL2Critical-04}. Due to proximity with $Q$, equation (\ref{Eqn-Qb_closeToQ}), there is universal constant $C$ with,
\begin{equation}\label{Eqn-HbTildeEst}\begin{aligned}
\norm{e^{C\abs{y}}V_{ij}}_{L^\infty} \rightarrow 0 
&& \text{ as }
&& b \rightarrow 0.
\end{aligned}\end{equation}
The following variation of $H$ is of a different nature. Let,
\begin{equation}\label{DefnEqn-HTilde}
\widetilde{H}(\epsilon,\epsilon) = 
	H(\epsilon,\epsilon) - \frac{1}{\norm{\Lambda Q}_{L^2}^2}\left(\epsRe,L_+\Lambda^2Q\right)\left(\epsRe,\Lambda Q\right),
\end{equation}
which simply alters the definition of ${\mathcal L}_+$, (\ref{DefnEqn-H}).
The following is a consequence of equation $(\ref{Strategy1-Algebra})$ and the Spectral Property,
\begin{lemma}[Alternative Spectral Property, {\cite[page 616]{MR-UniversalityBlowupL2Critical-04}}]
\label{Lemma-AlternateSpectral}
There exists a universal constant $\delta_0 > \tilde{\delta}_0 > 0$ such that $\forall \epsilon \in H^1$:
\begin{equation}\label{Eqn-2DSpectral-Alternate}\begin{aligned}
\widetilde{H}(\epsilon,\epsilon) \geq &\tilde{\delta}_0
	\left(\int_{y\in\real^2}{\abs{\grad_y\epsilon}^2} + 
		\int_{y\in\real^2}{\abs{\epsilon^2}e^{-\abs{y}}}\right)\\
	&-\frac{1}{\tilde{\delta}_0}\left(
		\begin{aligned}&Re\left(\epsilon,Q\right) + Re\left(\epsilon,\abs{y}^2 Q\right) +
		Re\left(\epsilon,yQ\right)\\ &+ Im\left(\epsilon,\Lambda Q\right) + 
		Im\left(\epsilon,\Lambda^2Q\right) + Im\left(\epsilon,\grad Q\right)\end{aligned}
	\right)^2.
\end{aligned}\end{equation}
\end{lemma}

The following result is proven \cite[Lemma 15]{MR-UniversalityBlowupL2Critical-04}. In Section \ref{Subsec-Lyapounov} we will find the study of linear radiation gives an accurate description of mass ejection from the singular regime.
\begin{lemma}[Linear Radiation]\label{Lemma-LinearRadiation}
There are universal constants $C>0$ and $\eta^*>0$ such that for all $0 < \eta < \eta^*$ there is $b^*(\eta) > 0$ such that for all $0 < b < b^*(\eta)$ the following is true. 

There exists a unique radial solution $\zeta_b$ to:
\begin{equation}\label{DefnEqn-Zb}
\left\{\begin{aligned}
\laplacian\zeta_b - \zeta_b + ib\Lambda\zeta_b = \Psi_b,\\
\int{\abs{\grad\zeta_b}^2} < + \infty,
\end{aligned}\right.
\end{equation}
where $\Psi_b$ is the truncation error given by (\ref{DefnEqn-TildeQb}). Moreover, if we denote 
\begin{equation}\label{DefnEqn-GammaB}
\Gamma_b = \lim_{\abs{y}\to+\infty}\abs{y}\abs{\zeta_b(y)}^2,
\end{equation}
then the solution satisfies:
\begin{itemize}
\item {\it Decay past the support of $\Psi_b$}:
\begin{equation}\begin{aligned}\label{Eqn-zb_H1L2Near}
\norm{\abs{y}\abs{\zeta_b}+\abs{y}^2\abs{\grad\zeta_b}}_{L^\infty(\abs{y}\geq R_b)} 
	&\leq \Gamma_b^{\frac{1}{2}-C\eta} < +\infty.
\end{aligned}\end{equation}

\item {\it Smallness in $\dot{H}^1$}:
\begin{equation}\label{Eqn-zb_smallH1}\begin{aligned}
\int{\abs{\grad_y\zeta_b}^2} \leq \Gamma_b^{1-C\eta}.
\end{aligned}\end{equation}

\item {\it Derivative with respect to $b$}:
\begin{equation}\label{Eqn-zb_by_b}
\norm{\frac{\partial\zeta_b}{\partial b}}_{C^1} \leq \Gamma_b^{\frac{1}{2}-C\eta}.
\end{equation}

\item {\it Stronger decay for larger $\abs{y}$}:
\begin{align}
\label{Eqn-zbH1Far}
&\norm{\abs{y}^2\abs{\grad\zeta_b}}_{L^\infty(\abs{y}\geq R_b^2)}
	\leq C\frac{\Gamma_b^\frac{1}{2}}{\abs{b}}, \text{  and}\\
\label{Eqn-GammaBEstimate} 
e^{-(1+C\eta)\frac{\pi}{b}} 
\leq \frac{4}{5} \Gamma_b 
\leq &\norm{\abs{y}^2\abs{\zeta_b}^2}_{L^\infty(\abs{y}\geq R_b^2)}
\leq e^{-(1-C\eta)\frac{\pi}{b}}.
\end{align}
As an estimate on $\Gamma_b$, (\ref{Eqn-GammaBEstimate}) will be indispensable.
\end{itemize}
\end{lemma}

Forthcoming parameter $a>0$ is universal, sufficiently small, will be determined in Section \ref{Subsec-Lyapounov}, and determines the choice of $\eta$. 
We denote,
\begin{equation}\begin{aligned}
\label{DefnEqn-A}
A(t) = e^{a\frac{\pi}{b(t)}}, 
&& \text{ so that, }
&& \Gamma_b^{-\frac{a}{2}} \leq A \leq \Gamma_b^{-\frac{3a}{2}},
\end{aligned}\end{equation}
and let $\phi_A$ denote a radially symmetric cutoff function with,
\begin{equation}\begin{aligned}
\label{DefnEqn-phiA}
\phi_A(y) = \left\{\begin{aligned}
	1 && \text{ for }\abs{y} \leq A\\
	0 && \text{ for }\abs{y} \geq 2A.
	\end{aligned}\right.
\end{aligned}\end{equation}
Truncated radiation $\widetilde{\zeta}_b(y)=\phi_A(y)\zeta_b$ satisfies:
\begin{equation}\label{DefnEqn-TildeZb}
\laplacian\widetilde{\zeta}_b - \widetilde{\zeta}_b +ib\Lambda\widetilde{\zeta}_b = \Psi_b + F,
\end{equation}
where the error term $F$ is explicit,
\begin{equation}\label{DefnEqn-F}
F = \zeta_b\laplacian\phi_A + 2\grad\phi_A\cdot\grad\zeta_b + ib\zeta_b y\cdot\grad\phi_A,
\end{equation}
and, in particular, by (\ref{Eqn-zbH1Far}) and (\ref{Eqn-GammaBEstimate}),
\begin{equation}\label{Eqn-FEstimate}
\abs{F}_{L^\infty}+\abs{y\cdot\grad F}_{L^\infty} \leq C\frac{\Gamma^\frac{1}{2}_b}{A}.
\end{equation}

\begin{remark}
\label{Remark-ChoiceOf-a}
For smaller values of $\eta$ the central profiles $\widetilde{Q}_b$ approximate the mass of the singular region more closely, equation (\ref{DefnEqn-Rb}), at the cost that estimates (\ref{Eqn-Qb_closeToQ}) through (\ref{Eqn-Qb_enerAndMoment}) are only known for ever smaller values of $b$.  When $\eta$ is larger, to compensate for the imperfection of our central profile we require more of the radiative tail to get an accurate picture of mass transport, requiring a larger choice of $a$.
See \cite[page 53]{MR-SharpLowerL2Critical-06} for similar remarks on the optimality in choice of $A(t)$. 
\end{remark}

\subsection{Estimates directly due to Geometric Decomposition}
\label{Subsec-LocalVirial}
The following Lemma explains our choice of norm for $\epsilon$.

\begin{lemma}[Weighted and Local $L^2$ Estimates]
\label{Lemma-L2ByGrad}
For any $\kappa > 0$ and for all $v \in H^1(\real^2)$,
\begin{equation}\label{Eqn-ExpDecayByGrad}
\int_{y\in\real^2}{\abs{v(y)}^2e^{-\kappa\abs{y}}}
	\leq C(\kappa)\left(\int{\abs{\grad v(y)}^2} + \int_{\abs{y} \leq 1}{\abs{v(y)}^2e^{-\abs{y}}}\right).
\end{equation}
\begin{equation}\label{Eqn-L2ByGrad}
\int_{\abs{y} \leq \kappa}{\abs{v(y)}^2} \leq C\,\kappa^2\log \kappa
	\left(\int{\abs{\grad v(y)}^2} 
	+ \int_{\abs{y}\leq 1}{\abs{v(y)}^2e^{-\abs{y}}}\right).
\end{equation}
\end{lemma}
Equation (\ref{Eqn-L2ByGrad}) is due to {\cite[eqn (4.11)]{MR-SharpLowerL2Critical-06}}. While, the original proof of (\ref{Eqn-ExpDecayByGrad}), {\cite[Lemma 5]{MR-UniversalityBlowupL2Critical-04}}, has a flaw, the methods of \cite{MR-SharpLowerL2Critical-06} give an alternate proof.

\begin{remark}[Non-concern for $\mu$]
In practice, we apply these Lemmas and the interaction estimates below only on regions within $\{\abs{y} \lesssim A(t)\}$. That is, equation (\ref{Hypo1-consequenceForMu}) always applies and we may choose to include measure $\mu(y)$ as appropriate.
\end{remark}

\begin{lemma}[Estimates on Interaction Terms, {\cite[Section 5.3(C)]{MR-SharpUpperL2Critical-03}} ]\label{Lemma-InteractionEst}
For all $s\in[s_0,s_1)$,
\begin{itemize}
\item {\it First-Order Terms}
\begin{equation}\label{Eqn-1stOrderEst}
\abs{\left(\epsilon(y),P(y)\frac{d^k}{dy^k}\widetilde{Q}_b(y)\right)}
	\leq C(P)\left(\epsNorm\right)^\frac{1}{2},
\end{equation}
where $P(y)$ is any polynomial and $0\leq k \leq 3$. 
\item {\it Second-Order Terms}
\begin{equation}\label{Eqn-2ndOrderEst}
\abs{\left(R(\epsilon),P(y)\frac{d^k}{dy^k}\widetilde{Q}_b(y)\right)}
	\leq C(P)\left(\epsNorm\right)
\end{equation}
where $P(y)$ is any polynomial, $0\leq k \leq 3$, and $R(\epsilon)$ is the terms of $(\epsilon+\widetilde{Q}_b)\abs{\epsilon+\widetilde{Q}_b}^2$ formally quadratic in $\epsilon$ - see equation (\ref{DefnEqn-R}).  
\item {\it Localized Higher-Order Terms}
\begin{equation}\label{Eqn-3rdOrderEst-J}
\int{\abs{J(\epsilon)-\abs{\epsilon}^4}\mu(y)\,dy}
	\leq \delta(\alpha^*)\left(\epsNorm\right),
\end{equation}
where $J(\epsilon)-\abs{\epsilon}^4 = 4\abs{\epsilon}^2Re\left(\epsilon\overline{\widetilde{Q}_b}\right)$, the term of $\abs{\epsilon+\widetilde{Q}_b}^4$ formally cubic in $\epsilon$ and localized to the support of $\widetilde{Q}_b$. Similarly,
\begin{equation}\label{Eqn-3rdOrderEst-RTilde}
\left(\widetilde{R}(\epsilon),\Lambda\widetilde{Q}_b\right)
	\leq \delta(\alpha^*)\left(\epsNorm\right),
\end{equation}
where $\widetilde{R}(\epsilon) = \epsilon\abs{\epsilon}^2$, the term of $(\epsilon+\widetilde{Q}_b)\abs{\epsilon+\widetilde{Q}_b}^2$ formally cubic in $\epsilon$. 
\end{itemize}
\end{lemma}
The following estimate is our first nontrivial departure from the $L^2$-critical argument.
\begin{lemma}[Complete Estimate on $J(\epsilon)$]
\label{Lemma-JEstimate}
For all $s\in[s_0,s_1)$,
\begin{equation}\label{Eqn-4thOrderEst}
\int{\abs{\epsilon(y)}^4\mu(y)\,dy} \leq \delta(\alpha^*)\left(\epsNorm\right).
\end{equation}
With equation (\ref{Eqn-3rdOrderEst-J}), this gives a complete estimate for $J(\epsilon)$.
\end{lemma}
\begin{proof}
Partition the support of $\epsilon$ into two and three dimensional regions,
\begin{equation}\label{Proof-4thOrder-eqn1}
\int{\abs{\epsilon(y)}^4\mu(y)\,dy}
=	\int{\left(1-\chi^4\right)\abs{\epsilon(y)}^4\mu(y)\,dy}
	+ \int{\abs{\chi\left(\lambda y+(r,z)(s)\right)\epsilon(y)}^4\mu(y)\,dy}.
\end{equation}
The first RH term is supported away from $r=0$, and due to {\bf H1.1} the support of $1-\chi^4$ is approximately, $\left\{\abs{y}<\frac{2}{3}\frac{1}{\lambda}\right\}$, so that, $\frac{1}{3} \lesssim \mu(y) \lesssim \frac{5}{3}$. We estimate this term by two-dimensional Sobolev embedding and the small mass assumption {\bf H1.2}.
Regarding the second RH term, the support of $\chi^4$ excludes the support of $\widetilde{Q}_b$ by the same reasons. Changing variables,
\begin{equation}\label{Proof-4thOrder-eqn2}
\int{\abs{\chi\left(x(y)\right)\epsilon(y)}^4\mu(y)\,dy}
=\lambda^2\int_{x\in\real^{3}}{\abs{\chi(x) u(x)}^4\,dx}.
\end{equation}
By the three-dimensional Sobolev embedding, $\dot{H}^{\frac{3}{4}} \hookrightarrow L^4(\real^{3})$, and interpolation,
\begin{equation}\begin{aligned}
\label{Proof-4thOrder-eqn3}
\lambda^2\int_{x\in\real^{3}}{\abs{\chi(x) u(x)}^4\,dx}
	&\lesssim \norm{\chi u}_{\dot{H}^\frac{1}{2}}^2
		\lambda^2\norm{\chi u}_{\dot{H}^1(\real^{3})}^2
	&\lesssim \norm{\chi u}_{H^\frac{1}{2}}\left(\int{\abs{\grad_y\epsilon}^2\mu\,dy}\right).
\end{aligned}\end{equation}
To complete the proof, we use the assumed control {\bf H2.3} for the first and only time.
\end{proof}

\begin{lemma}[Estimates due to Conservation Laws]
\label{Lemma-PrelimEstimatesConservLaws}
For all $s\in[s_0,s_1)$ the following are true:
\begin{itemize}
\item {\it Due to conservation of mass:}
\begin{equation}\label{Eqn-prelimMassEst}
b^2 + \int{\abs{\tilde{u}}^2} \lesssim \left(\alpha^*\right)^\frac{1}{2}.
\end{equation}
\item {\it Due to conservation of energy:}
\begin{multline}\label{Eqn-prelimEnerEst}
\abs{ 
2Re\left(\epsilon,\widetilde{Q}_b\right) 
-\int{\abs{\grad\epsilon}^2\mu(y)\,dy}
+3\int_{\abs{y}\leq\frac{10}{b}}{Q^2\epsRe^2}
+\int_{\abs{y}\leq\frac{10}{b}}{Q^2\epsIm^2}
}\\
\shoveright{
\leq 
\Gamma_b^{1-C\eta} + \delta(\alpha^*)\left(\epsNorm\right),
}
\end{multline}
\item {\it Due to localized momentum (\ref{Hypo1-enerMoment}):}
\begin{equation}\label{Eqn-prelimMomentEst}
\abs{Im\left(\epsilon,\grad\widetilde{Q}\right)}
\leq
\Gamma_b^2 + \delta(\alpha^*)\left(\epsNorm\right)^\frac{1}{2}.
\end{equation}
In particular, by H\"older and (\ref{Eqn-Qb_closeToQ}), (\ref{Eqn-prelimMomentEst}) also holds for $\abs{\left(\epsIm,Re(\grad\widetilde{Q}_b)\right)}$.
\end{itemize}
\end{lemma}
\begin{proof}

\begin{itemize}
\item {\it Conservation of mass}: $\int_{\real^{3}}{\abs{u(t)}^2\,dx} = \int{\abs{u_0}^2}$

From the geometric decomposition, expand and change some variables,
\begin{equation}\label{Proof-prelimMass-eqn1}
\int{\abs{\widetilde{Q}_b(y)}^2\mu(y)\,dy}
	+2Re\left(\int{\epsilon\overline{\widetilde{Q}}_b\mu(y)\,dy}\right)
	+\int{\abs{\tilde{u}(t)}^2}
	=\int{\abs{u_0}^2}.
\end{equation}
Expand measure $\mu$. 
Due to the bound on $\lambda$, equation (\ref{Hypo1-consequenceForLambdaGamma}), hypotheses  {\bf H1.1} and {\bf H1.2}, and the supercritical mass of $\widetilde{Q}_b$, 
\begin{equation}\label{Proof-prelimMass-eqn2}\begin{aligned}
\int{\abs{\widetilde{Q}_b}^2\mu(y)\,dy} - \int{Q^2} =
& \lambda\int{y_1\abs{\widetilde{Q}_b}^2\,dy}
+(r(t)-1)\int{\abs{\widetilde{Q}_b}^2}\\
& +\int{\abs{\widetilde{Q}_b}^2} - \int{Q^2}
\gtrsim  b^2 - \sqrt{\alpha^*}.
\end{aligned}\end{equation}
Due to small $b_0$ and the small mass of $\epsilon_0$, {\bf C1.2} 
$\abs{\int_{\real^{3}}{\abs{u_0}^2}-\int_{\real^2}{Q^2}} \lesssim C\alpha^*$.

Due to local support, and hypothesis {\bf H1.2},
$
\abs{\int{\epsilon\overline{\widetilde{Q}}_b\mu}} \lesssim \alpha^*.
$

\item {\it Conservation of Energy}: $\int_{\real^{3}}{\abs{\grad u(t)}^2\,dx}-\frac{1}{2}\int{\abs{u}^4} = 2E_0$

From the geometric decomposition, 
\begin{equation}\label{Proof-prelimEnergy-eqn1}
2\lambda^2E_0 = \int{\abs{\grad_y(\widetilde{Q}_b+\epsilon)}^2\mu(y)\,dy} - \frac{1}{2}\int{\abs{\widetilde{Q}_b+\epsilon}^4\mu(y)\,dy}
\end{equation}
Partially expand measure $\mu$,
\begin{equation}\label{Proof-prelimEnergy-eqn2}\begin{aligned}
\int{\abs{\grad_y(\widetilde{Q}_b+\epsilon)}^2\mu(y)\,dy}
= &
	r(t)\int{\abs{\grad_y\widetilde{Q}_b}^2} 
	+ 2r(t)Re\left(\int{\grad_y\epsilon\cdot\grad_y\overline{\widetilde{Q}_b}}\right)
	+ \int{\abs{\grad_y\epsilon}^2\mu(y)\,dy} \\
 &
	+\int{\lambda y_1
		\left(\abs{\grad_y\widetilde{Q}_b}^2
		 + 2Re(\epsilon\overline{\widetilde{Q}_b})\right)\,dy}.
\end{aligned}\end{equation}
Due to the support of $\widetilde{Q}_b$, the second line is of the order $\lambda$, and thus inconsequential. 
Through a similar approach,
\begin{equation}\label{Proof-prelimEnergy-eqn3}\begin{aligned}
-\frac{1}{2}\int{\abs{\widetilde{Q}_b+\epsilon}^4\mu(y)\,dy}
= -&r(t)\left(\begin{aligned}
		&\frac{1}{2}\int{\abs{\widetilde{Q}_b}^4} 
		+2Re\left(\int{\epsilon\overline{\widetilde{Q}_b}\abs{\widetilde{Q}_b}^2}\right)\\
		&+\int{\abs{\epsilon}^2\abs{\widetilde{Q}_b}^2}
		+Re\left(\int{\epsilon^2\overline{\widetilde{Q}_b}^2}\right)
		\end{aligned}\right)\\
	&+\lambda\, O\left(\abs{\widetilde{Q}_b}^2\right)
		-\frac{1}{2}\int{J(\epsilon)\mu(y)\,dy}.
\end{aligned}\end{equation}
Now proceed as in the $L^2$-critical argument. Integrate $\int{\grad_y\epsilon\cdot\grad_y\overline{\widetilde{Q}_b}}$ by parts and substitute the equation for $\widetilde{Q}_b$ (\ref{DefnEqn-TildeQb}) - this cancels the term of (\ref{Proof-prelimEnergy-eqn3}) linear in $\epsilon$. Recall the bound for $\Psi_b$ (\ref{Eqn-PsibEstimate}), the degenerate energy of $\widetilde{Q}_b$ (\ref{Eqn-Qb_enerAndMoment}), proximity to $Q$ (\ref{Eqn-Qb_closeToQ}), that $r(t) \sim 1$, and the non-trivial estimate on $J$, equation (\ref{Eqn-4thOrderEst}).

\item {\it Localized momentum (\ref{Hypo1-enerMoment})}:

In cylindrical coordinates, $\grad_xf\cdot\grad_xg = \partial_rf\partial_rg + \partial_zf\partial_zf$. For this proof we denote $r$ by $x_1$ and $z$ by $x_2$. Fix either $j=1$ or $j=2$.
From the geometric decomposition,
\begin{multline}\label{Proof-prelimMoment-eqn1}
\lambda(t)Im\left(\int_{\real^{3}}{\partial_{x_j}\psi^{(x)}\partial_{x_j} u\,\overline{u}\,dx}\right)
=
Im\left(\int{
	\partial_{x_j}\psi^{(x)}
	\partial_{y_j}\left(\widetilde{Q}_b+\epsilon\right)\,
	\overline{\left(\widetilde{Q}_b+\epsilon\right)}\mu(y)\,dy}\right).
\end{multline}
Directly from definition (\ref{DefnEqn-psi-x}), $\partial_{x_j}\psi^{(x)} = 1$ on the support of $\widetilde{Q}_b$. 
For the interaction term in $\partial_{y_j}\epsilon\overline{\widetilde{Q}_b}$ we expand the measure $\mu$ (\ref{DefnEqn-mu}) and integrate by parts the term in $r(t)$. Applying the denegerate momentum of $\widetilde{Q}_b$ (\ref{Eqn-Qb_enerAndMoment}) we have,
\begin{equation}\label{Proof-prelimMoment-eqn2}\begin{aligned}
2r(t)Im\left(\epsilon,\partial_{y_j}\widetilde{Q}_b\right) = 
& Im\left(\int{\lambda y_1
		\left(\partial_{y_j}\epsilon\overline{\widetilde{Q}_b} + \partial_{y_j}\widetilde{Q}_b\overline{\epsilon} + \partial_{y_j}\widetilde{Q}_b\overline{\widetilde{Q}_b}\right)\,dy}\right)\\
&+Im\left(\int{\partial_{x_j}\psi^{(x)}
	\partial_{y_j}\epsilon\overline{\epsilon}\mu(y)\,dy}\right)\\
&-\lambda(t)Im\left(\int_{\real^{3}}{\partial_{x_j}\psi^{(x)}\partial_{x_j} u\,\overline{u}\,dx}\right).
\end{aligned}\end{equation}
The first line is the order $\lambda$, and thus negligible. 
The second line we apply H\"older and the small mass assumption {\bf H1.2}. The final term is controlled by {\bf H1.4}.
\begin{remark}[Role of Momentum Conservation]
The estimate analogous to (\ref{Eqn-prelimMomentEst}) in the $L^2$-critical context is proven with the conservation of momentum in place of {\bf H1.4}, \cite[Appendix A]{MR-SharpUpperL2Critical-03}.
As might be expected, the proof of {\bf I1.4} will resemble the proof of momentum conservation. See equation (\ref{Proof-Improv1-enerMoment-MorawetzType}).
\end{remark}
\end{itemize}

\end{proof}

\begin{definition}[NLS Reformulated for $\epsilon$]
\label{Defn-epsEqn}
For $s\in[s_0,s_1)$, $y \in \left\lbrack-\frac{r(t)}{\lambda(t)},+\infty\right)\times\real$, and a suitable boundary condition at $y_1 = -\frac{r(t)}{\lambda(t)}$, $\epsilon$ satisfies:
\begin{equation}\begin{aligned}
\label{DefnEqn-epsEqn}
	ib_s\frac{\partial\widetilde{Q}_b}{\partial b} 
	+i\epsilon_s 
	-M(\epsilon) + \frac{\lambda}{r(y_1)}\partial_{y_1}\epsilon
	+ib\Lambda\epsilon
= &
	i\left(\frac{\lambda_s}{\lambda}+b\right)\Lambda\widetilde{Q}_b 
	+\tilde{\gamma}_s\widetilde{Q}_b
	+i\frac{(r_s,z_s)}{\lambda}\cdot\grad_y\widetilde{Q}_b\\
  &
	+i\left(\frac{\lambda_s}{\lambda}+b\right)\Lambda\epsilon 
	+\tilde{\gamma}_s\epsilon
	+i\frac{(r_s,z_s)}{\lambda}\cdot\grad_y\epsilon\\
  &
	+\Psi_b
	-R(\epsilon),
\end{aligned}\end{equation} 
where we introduced the new variable,
\begin{equation}\label{DefnEqn-gammaTilde}
\tilde{\gamma}(s) = -s - \gamma(s).
\end{equation}
Note the single new term due to cylindrical symmetry.
As already mentioned, the term $R(\epsilon)$ corresponds to those terms formally quadratic in $\epsilon$,
\begin{equation}\label{DefnEqn-R}
R(\epsilon) = 
	\left(\epsilon+\widetilde{Q}_b\right)\abs{\epsilon+\widetilde{Q}_b}^2
	-\widetilde{Q}_b\abs{\widetilde{Q}_b}^2
	-2\abs{\widetilde{Q}_b}^2\epsilon - \left(2\widetilde{Q}_b^2-Re(\widetilde{Q}_b^2)\right)\overline{\epsilon}.
\end{equation}
\end{definition}

\begin{lemma}[Estimates due to Orthogonality Conditions]
\label{Lemma-PrelimEstimatesOrthogConds}
For all $s\in[s_0,s_1)$, 
\begin{equation}\label{Eqn-prelimLambda+BEst}
\abs{\frac{\lambda_s}{\lambda}+b} +\abs{b_s} 
	\lesssim
	\Gamma_b^{1-C\eta} + \left(\epsNorm\right),
\end{equation}
\begin{multline}\label{Eqn-prelimGamma+REst}
\abs{\tilde{\gamma}_s - 
	\frac{1}{\abs{\Lambda Q}_{L^2}^2}\left(\epsRe,L_+(\Lambda^2Q)\right)}
+ \abs{\frac{r_s}{\lambda}}+\abs{\frac{z_s}{\lambda}}\\
{
\leq
\Gamma_b^{1-C\eta} + \delta(\alpha^*)\left(\epsNorm\right)^\frac{1}{2}.
}
\end{multline}
\end{lemma}
Estimates (\ref{Eqn-prelimLambda+BEst}) and (\ref{Eqn-prelimGamma+REst}) are a direct result of orthogonality conditions (\ref{Eqn-GeoDecomp-OrthogReY2}), (\ref{Eqn-GeoDecomp-OrthogReY}), (\ref{Eqn-GeoDecomp-OrthogImL2}) and (\ref{Eqn-GeoDecomp-OrthogImL}) by taking the respective inner products with the $\epsilon$ equation (\ref{DefnEqn-epsEqn}). Due to equation (\ref{Hypo1-consequenceForLambdaGamma}), terms resulting from $\frac{\lambda}{r(\lambda)}\partial_{y_1}\epsilon$ are inconsequential.
The estimates due to energy and momentum, (\ref{Eqn-prelimEnerEst}) and (\ref{Eqn-prelimMomentEst}), are involved in the estimates of $\abs{b_s}$ and $\abs{\frac{r_s}{\lambda}}+\abs{\frac{z_s}{\lambda}}$ respectively.
Otherwise, all calculations are localized to the support of $\widetilde{Q}_b$ and are identical to the $L^2$-critical argument.
See \cite[Appendix C]{MR-UniversalityBlowupL2Critical-04} or \cite[Appendix A]{R-StabilityOfLogLog-05} for the complete calculations.

\begin{lemma}[Local Virial Identity]\label{Lemma-LocalVirial}
For all $s\in[s_0,s_1)$,
\begin{equation}\label{Eqn-LocalVirial}
b_s \geq \delta_1\left(\epsNorm\right) - \Gamma_b^{1-C\eta},
\end{equation}
where $\delta_1 > 0$ is a universal constant.
\end{lemma}
\begin{proof}[Brief Proof]
Begin with the method used to prove preliminary estimate (\ref{Eqn-prelimLambda+BEst}). Take the real part of the inner product of $\epsilon$ equation (\ref{DefnEqn-epsEqn}) with $\Lambda\widetilde{Q}_b$.  Recognize that $\partial_sIm\left(\epsilon,\Lambda\widetilde{Q}_b\right) = 0$ due to orthogonality condition (\ref{Eqn-GeoDecomp-OrthogImL}). 
An adapted version of the algebraic property $L_+(\Lambda Q) = -2Q$ is applied, \cite[equation (101)]{MR-UniversalityBlowupL2Critical-04}.
After recognizing the equation of $\widetilde{Q}_b$, 
injecting the conservation of energy cancels the remaining terms linear in $\epsilon$.  
The resulting terms quadratic in $\epsilon$ are the bilinear operator $H_b(\epsilon,\epsilon)$, equation (\ref{DefnEqn-Hb}). The remaining terms cubic in $\epsilon$ (due to the original inner product) were estimated as part of Lemma \ref{Lemma-InteractionEst}.
See \cite[Appendix C]{MR-UniversalityBlowupL2Critical-04} 
for the complete calculation. Controlling the auxilliary terms of the conservation of energy with equation (\ref{Eqn-prelimEnerEst}) we have,
\begin{equation}\label{Proof-LocalVirial-firstMain}\begin{aligned}
-b_s\,Im\left(\frac{\partial}{\partial b}\widetilde{Q}_b,\Lambda\widetilde{Q}_b\right) 
\gtrsim
&H_b(\epsilon,\epsilon)\\
&+b_s\,Im\left(\epsilon,\Lambda\frac{\partial}{\partial b}\widetilde{Q}_b\right)
-\left(\frac{\lambda_s}{\lambda} + b\right)
		Im\left(\epsilon,\Lambda^2\widetilde{Q}_b\right)\\
&-\tilde{\gamma}_s\,Re\left(\epsilon,\Lambda\widetilde{Q}_b\right)
-\frac{(r_s,z_s)}{\lambda}\cdot Im\left(\epsilon,\grad\widetilde{Q}_b\right)\\
&-\Gamma_b^{1-C\eta} - \delta(\alpha^*)\left(\epsNorm\right).
\end{aligned}\end{equation}
Recall that $\partial_b\widetilde{Q}_b \approx -i\frac{\abs{y}^2}{4}Q$,
make the correction (\ref{Eqn-HbTildeEst}) for $\widetilde{H}_b$,  
and apply preliminary estimates (\ref{Eqn-prelimLambda+BEst}) and (\ref{Eqn-prelimGamma+REst}). With the proximity to $Q$ we may write,
\begin{multline}\label{Proof-LocalVirial-secondMain}
b_s\,\frac{1}{4}\norm{yQ}_{L^2}^2 
\gtrsim
H(\epsilon,\epsilon) - \tilde{\gamma}_s\left(\epsRe,\Lambda Q\right)\\
-\Gamma_b^{1-C\eta} - \delta(\alpha^*)\left(\epsNorm\right).
\end{multline}
Identify the alternate form of $\widetilde{H}$, equation (\ref{DefnEqn-HTilde}), apply the preliminary estimate for $\tilde{\gamma}_s$, equation (\ref{Eqn-prelimGamma+REst}), and apply the adapted version of the Spectral Property, Lemma \ref{Lemma-AlternateSpectral}.
\end{proof}

\begin{remark}[Progress in proving Proposition \ref{Prop-Improv}]
\label{Remark-waypointToImprov1}
We have already proven the first half of {\bf I1.2} as the preliminary estimate (\ref{Eqn-prelimMassEst}). The local virial identity with preliminary estimate (\ref{Eqn-prelimLambda+BEst}) produce a closed expression for $\lambda$ and $b$, which we treat with simple arguments to prove the following Lemma.  In particular,
equation (\ref{Eqn-lambdaUpperBound}) implies the
1\textsuperscript{st} lower bound of {\bf I1.3}. 
Following similar methods, we will then prove, the 2\textsuperscript{nd} upper bound of 
{\bf I1.3}, {\bf I1.4}, {\bf I1.5} and {\bf I1.1}.
\end{remark}

\begin{lemma}[Upper Bound on Blowup Rate]\label{Lemma-Upperbound}
For all $s\in[s_0,s_1)$,
\begin{equation}\label{Eqn-bLowerBound}
b(s) \geq \frac{3\pi}{4\log s}, \text{ and}
\end{equation}
\begin{equation}\label{Eqn-lambdaUpperBound}
\lambda(s) \leq \sqrt{\lambda_0}e^{-\frac{\pi}{3}\frac{s}{\log s}}.
\end{equation}
\end{lemma}
\begin{proof}
Inject hypothesis {\bf H1.2} into the local virial identity (\ref{Eqn-LocalVirial}) and carefully integrate in time. From $b>0$, the bound on $\Gamma_b$ (\ref{Eqn-GammaBEstimate}), and the clever choice of $s_0$, equation (\ref{DefnEqn-s}),
\begin{equation}\label{Proof-UpperBound-eqn1}\begin{aligned}
\partial_se^{+\frac{3\pi}{4b}} = -\frac{b_s}{b^2}\frac{3\pi}{4}e^{+\frac{3\pi}{4b}} \leq 1 
&& \Longrightarrow 
&& e^{+\frac{3\pi}{4b}} \leq s -s_0 + e^{+\frac{3\pi}{4b_0}} = s.
\end{aligned}\end{equation}
This proves (\ref{Eqn-bLowerBound}).
Next, we view preliminary estimate (\ref{Eqn-prelimLambda+BEst}) and hypothesis {\bf H1.2} as the approximate dynamics of $\lambda$,
\begin{equation}\label{Eqn-lambda-prelimDynamics}
\abs{\frac{\lambda_s}{\lambda}+b}+\abs{b_s} < \Gamma_b^\frac{1}{2}.
\end{equation}
In particular as $b>0$ is small, $-\frac{\lambda_s}{\lambda} \geq \frac{2b}{3}$, which we integrate with (\ref{Eqn-bLowerBound}),
\begin{equation}\label{Proof-UpperBound-eqn2}
-\log\lambda \geq -\log\lambda_0 + \int_{s_0}^{s}{\frac{\pi}{2\log\sigma}\,d\sigma}.
\end{equation}
Assume $s_0$ is sufficiently large through choice of data (\ref{DataP1-mass}) with $\alpha^*$ sufficiently small, then,
\begin{equation}\label{Proof-UpperBound-eqn3}
\int_{s_0}^s{\frac{\pi}{2\log\sigma}d\sigma} 
	\geq \frac{\pi}{3}\left(\frac{s}{\log s}-\frac{s_0}{\log s_0}\right).
\end{equation}
From the choice of data {\bf C1.3}, 
and (\ref{DefnEqn-s}), $-\log\lambda_0 \geq e^{\frac{\pi}{2b_0}} = s_0^\frac{3}{2}$.  Thus we have proven (\ref{Eqn-lambdaUpperBound}),
\[
-\log\lambda \geq -\frac{1}{2}\log\lambda_0 + \frac{\pi}{3}\frac{s}{\log s}.
\]
\end{proof}
\begin{proof}[Corollary of (\ref{Eqn-lambdaUpperBound})]
By simple change of variables, (\ref{Eqn-lambdaUpperBound}), and choice of data (\ref{DataP1-mass}) and (\ref{DataP1-loglog}),
\begin{equation}\label{Eqn-t1Small}
T_{hyp} 
= \int_{s_0}^{s_1}{\lambda^2(\sigma)\,d\sigma}
\leq \lambda_0\int_{2}^{+\infty}{e^{-\frac{2\pi}{3}\frac{s}{\log s}}\,ds}
< \alpha^*.
\end{equation}
\end{proof}
\begin{proof}[Proof of {\bf I1.3}, 2\textsuperscript{nd} upper bound]
As a direct consequence of (\ref{Eqn-lambdaUpperBound}), again assuming $s_0 > 0$ sufficiently large, 
\begin{equation}\label{Proof-Improv1-loglog-eqn1}
-\log\left(s\lambda(s)\right) \geq \frac{\pi}{3}\frac{s}{\log s} - \log s \geq \frac{s}{\log s}.
\end{equation}
Taking the logarithm and applying equation (\ref{Eqn-bLowerBound}),
\begin{equation}\label{Proof-Improv1-loglog-eqn2}\begin{aligned}
\log\abs{-\log\left(s\lambda(s)\right)} 
	\geq \log\left(\frac{s}{\log s}\right)
	\geq \frac{4}{15}\log s \geq \frac{\pi}{5 b(s)}
&& \Longrightarrow
&& s\lambda(s) \leq e^{-e^{\frac{\pi}{5b}}},
\end{aligned}\end{equation} 
which in particular implies $\lambda \leq e^{-e^{\frac{\pi}{5b}}}$, the second upper bound of {\bf I1.3}
\end{proof}
\begin{proof}[Proof of {\bf I1.4}]
Recall approximate dynamic (\ref{Eqn-lambda-prelimDynamics}), which was due to preliminary estimate (\ref{Eqn-prelimLambda+BEst}) and the hypothesized control on $\epsilon$. As a consequence, for $s\in[s_0,s_1)$,
\begin{equation}\label{Proof-Improv1-enerMoment-eqn1}\begin{aligned}
\frac{d}{ds}\left(\lambda^2e^\frac{5\pi}{b}\right)
	= 2\lambda^2e^\frac{5\pi}{b}\left(\frac{\lambda_s}{\lambda}+b-b-\frac{5\pi b_s}{2b^2}\right)
	\leq& -\lambda^2be^{5\pi}{b} < 0,\\
&\Longrightarrow \lambda^2(t)e^{\frac{5\pi}{b(t)}} \leq \lambda_0^2e^\frac{5\pi}{b_0}.
\end{aligned}\end{equation}
Then, with the estimate on $\Gamma_b$ (\ref{Eqn-GammaBEstimate}), the choice of data (\ref{DataP1-enerMoment}) and the estimate on $\Gamma_b$ again,
\begin{equation}\label{Proof-Improv1-enerMoment-eqn2}
\lambda^2(t)\abs{E_0} < \Gamma_{b(t)}^4\,e^\frac{5\pi}{b_0}\lambda_0^2\abs{E_0}
	<\Gamma_{b(t)}^4\,e^\frac{5\pi}{b_0}\Gamma_{b_0}^{10} \ll \Gamma_{b(t)}^4,
\end{equation}
which proves the energy-normalization part of {\bf I1.4}. Regarding the localized momentum, calculate directly from equation (\ref{Eqn-NLS}) that,
\begin{equation}\begin{aligned}\label{Proof-Improv1-enerMoment-MorawetzType}
\frac{d}{dt}Im\left(\int{\grad\psi^{(x)}\cdot\grad u\overline{u}}\right)
	= &
	Re\left(\int{\partial_{x_j}\partial_{x_k}\psi^{(x)}
		\partial_{x_k}u\partial_{x_j}\overline{u}}\right)\\
	&-\frac{1}{2}\int{\laplacian\psi^{(x)}\abs{u}^4}
	-\frac{1}{4}\int{\laplacian^2\psi^{(x)}\abs{u}^2}.
\end{aligned}\end{equation}
This is a special case of the general Morawetz calculaton - eg. \cite[equation (3.36)]{Tao06}. 
Recall from definition (\ref{DefnEqn-psi-x}) that the support of $\psi^{(x)}$ is well away from $r=0$. Apply the two-dimensional Sobolev embedding $H^\frac{1}{2} \hookrightarrow L^4$ to estimate,
\begin{equation}\label{Proof-Improv1-enerMoment-eqn4}
\abs{\frac{d}{dt}Im\left(\int{\grad\psi^{(x)}\cdot\grad u\overline{u}}\right)}
	\leq C(\psi^{(x)})\norm{u(t)}_{H^1}^2 \lesssim \frac{1}{\lambda^2},
\end{equation}
where the final inequality is due to hypothesized control on $\epsilon$ and the small excess mass {\bf H1.2}.  Note that, $\int_0^t{\frac{d\tau}{\lambda^2(\tau)}} = \int_{s_0}^s{d\sigma} \leq s$, and so we have proven,
\begin{equation}\label{Proof-Improv1-enerMoment-eqn5}
\lambda(t)\abs{Im\left(\grad\psi^{(x)}\cdot\grad u(t)\overline{u}(t)\right)}
	\leq \lambda(t)\abs{Im\left(\grad\psi^{(x)}\cdot\grad u_0\overline{u}_0\right)}
	+ C\lambda(t)s(t).
\end{equation}
Due to the estimate on $\Gamma_b$ (\ref{Eqn-GammaBEstimate}) and equation (\ref{Proof-Improv1-loglog-eqn2}) from the previous proof, $C\lambda(t)s(t) \leq C\Gamma_{b(t)}^{10} \ll \Gamma_b^4$. 
Then by virtually the same calculation as equations (\ref{Proof-Improv1-enerMoment-eqn1}) and (\ref{Proof-Improv1-enerMoment-eqn2}), for $s\in[s_0,s_1)$,
\begin{equation}\label{Proof-Improv1-enerMoment-eqn6}
\frac{d}{ds}\left(\lambda e^\frac{6\pi}{b}\right) \leq -\frac{1}{2}\lambda be^\frac{6\pi}{b} < 0
\Longrightarrow \lambda(t) e^\frac{6\pi}{b(t)} \leq \lambda_0e^\frac{6\pi}{b_0},
\end{equation}
and so by the estimate on $\Gamma_b$ and choice of data (\ref{DataP1-enerMoment}),
\[
\lambda(t)\abs{Im\left(\grad\psi^{(x)}\cdot\grad u_0\overline{u}_0\right)}
\leq \Gamma_{b(t)}^5\,e^\frac{6\pi}{b_0}\Gamma_{b_0}^{10} \ll \Gamma_{b(t)}^4.
\]
This proves the localized-momentum part of {\bf I1.4}.
\end{proof}
\begin{proof}[Proof of {\bf I1.5}]
We follow the argument found in the proof of \cite[Lemma 7]{R-StabilityOfLogLog-05}. 
Fix some $s_2 \leq s_3 \in [s_0,s_1)$.
Substitute the local virial identity (\ref{Eqn-LocalVirial}) into the preliminary estimate (\ref{Eqn-prelimLambda+BEst}) to control the norm of $\epsilon$. With a crude bound for $\Gamma_b$,
\begin{equation}\label{Proof-Improv1-almostMono-eqn1}
\abs{\frac{\lambda_s}{\lambda} + b} \leq C \left( b_s + b^2 \right),
\end{equation}
From hypothesis {\bf H1.2}, $0 < b^2 < \delta(\alpha^*)b$
where $\delta(\alpha^*) \rightarrow 0$ as $\alpha^* \to 0$. Then,
\begin{equation}\label{Proof-Improv1-almostMono-eqn2}
-\log\left(\frac{\lambda(s_2)}{\lambda(s_3)}\right) 
= \int_{s_2}^{s_3}{\left(\frac{\lambda_s}{\lambda}+b\right)} - \int_{s_2}^{s_3}{b}
\leq \delta(\alpha^*) - \frac{1}{2}\int_{s_2}^{s_3}{b}
\leq \delta(\alpha^*).
\end{equation}
In particular, we may assume $\alpha^*$ is such that $\delta(\alpha^*) < \log 2$, which proves {\bf H1.5}.
\end{proof}
\begin{proof}[Proof of {\bf H1.1}]
Preliminary estimate (\ref{Eqn-prelimGamma+REst}) can be crudely simplified to,
\begin{equation}\label{Proof-Improv1-rz-eqn1}
\abs{\frac{r_s}{\lambda}}+\abs{\frac{z_s}{\lambda}} \leq 1.
\end{equation}
Then we have for all $s\in[s_0,s_1)$,
\begin{equation}\label{Proof-Improv1-rz-eqn2}
\abs{r(s)-r_0}+\abs{z(s)-z_0} 
	\leq \int_{s_0}^s{\abs{r_s}+\abs{z_s}} 
	\leq \int_{s_0}^s{\lambda(\sigma)\,d\sigma}
	\leq \sqrt{\lambda_0}\int_2^{+\infty}{e^{-\frac{\pi}{3}\frac{\sigma}{\log\sigma}}\,d\sigma}
<\alpha^*,
\end{equation}
where we applied (\ref{Eqn-lambdaUpperBound}), the choice of data (\ref{DataP1-loglog}) and the smallness of $b_0$ (\ref{DataP1-mass}). With our choice of $r_0$,$z_0$ (\ref{DataP1-rz}), this proves {\bf H1.1}.
\end{proof}

\subsection{Lyapounov Functional}
\label{Subsec-Lyapounov}
To begin this section, we repeat the calculation of the local virial identity, this time including the linear radiation $\widetilde{\zeta}_b$ as part of the central profile. That is we write,
\begin{equation}\label{DefnEqn-epsTilde}\begin{aligned}
\tilde{\epsilon} = \epsilon - \widetilde{\zeta}_b
&& \Rightarrow
&& u(t,x) = \frac{1}{\lambda(t)}\left(\widetilde{Q}_{b(t)}+\widetilde{\zeta}_{b(t)}+\tilde{\epsilon}(t)\right)
	\left(\frac{(r,z)-(r(t),z(t))}{\lambda}\right)e^{-i\gamma(t)},
\end{aligned}\end{equation}
where the parameters of the geometric decomposition are unchanged.  The equation for $\tilde{\epsilon}$ may then be written analogous to (\ref{DefnEqn-epsEqn}), with a new linearized evolution operator analogous to $M$, (\ref{DefnEqn-M}).



\begin{lemma}[Radiative Virial Identity]\label{Lemma-RefinedVirial}
For all $s\in[s_0,s_1)$,
\begin{equation}\label{Eqn-RadiativeVirial}\begin{aligned}
\partial_s f_1 \geq &\delta_2\left(\epsTildeNorm\right)\\
	&+\Gamma_b - \frac{1}{\delta_2}\int_{A\leq\abs{y}\leq 2A}{\abs{\epsilon}^2\,dy},
\end{aligned}\end{equation}
where $\delta_2, c > 0$ are universal constants and,
\begin{equation}\label{DefnEqn-f1}
f_1(s) = 
	\frac{b}{4}\abs{y\widetilde{Q}_b}_{L^2}^2 
	+ \frac{1}{2}Im\left(\int{y\cdot\grad\widetilde{\zeta}_b\overline{\widetilde{\zeta}}_b}\right)
	+Im\left(\epsilon,\Lambda\widetilde{\zeta}_b\right).
\end{equation}
\end{lemma}
Compared with the local virial identity, 
the radiative virial identity is useless to control $\epsilon$ in $\dot{H}^1$ due to the presence of mass term $\int_{A\leq\abs{y}\leq 2A}{\abs{\epsilon}^2}$.
See equation (\ref{Eqn-L2ByGrad}) for further discouragement. Nevertheless, we will link this  term to the ejection of mass from the singularity, through the radiation, into the dispersive regime - Lemma \ref{Lemma-MassDisperse}.  Then, we will show this mass ejection is more or less uninterrupted by demonstrating the Lyapounov fuctional - Lemma \ref{Lemma-LyapounovFunc}. Finally, through the conservation of energy we will prove precise bounds on the Lyapounov functional in terms of the excess mass at the singularity and $\abs{\epsilon}_{\dot{H}^1}$ - Lemma \ref{Lemma-LyapounovEstimates}. These bounds will allow us to bridge between times where $b_s \leq 0$ (times where the local virial identity is useful) to give a control for $\epsilon$ pointwise in time - Lemma \ref{Lemma-Lowerbound}.

Let $\phi_\infty$ be a smooth radial cutoff function on $\real^2$,
\begin{equation}\label{DefnEqn-phiInfty}\begin{aligned}
\phi_{\infty}(y) = \left\{\begin{aligned}
	0 && \text{ for }\abs{y} \leq \frac{1}{2}\\
	1 && \text{ for }\abs{y} \geq 3,
	\end{aligned}\right.
&& \text{ and }
&& \begin{aligned}
		\frac{1}{4}\leq\phi_{\infty}'\leq \frac{1}{2} &&& \text{ for }1\leq \abs{y}\leq 2,\\
		0\leq \phi_{\infty}' &&& \text{ for all }y.
	\end{aligned}
\end{aligned}\end{equation}
\begin{lemma}[Mass-ejection from Singular and Radiative Regimes] 
\label{Lemma-MassDisperse}
\begin{equation}\label{Eqn-MassDisperse}
\partial_s\left(\frac{1}{r(t)}\int{\phi_{\infty}\left(\frac{y}{A}\right)\abs{\epsilon}^2\mu(y)\,dy}\right)
	\geq \frac{b}{400}\int_{A\leq \abs{y} \leq 2A}{\abs{\epsilon}^2\,dy}
	-\Gamma_b^\frac{a}{2}\int{\abs{\grad_y\epsilon}^2\mu(y)\,dy}
	-\Gamma_b^2.
\end{equation}
\end{lemma}
\begin{remark}[Interpretation of Lemma \ref{Lemma-MassDisperse}]\label{Remark-MeaningOfMassDisperse}
Assume for the sake of heuristics that $\epsilon \approx \zeta_b$ on the region, $\abs{y} \sim A$. Then with the definition of $\Gamma_b$ (\ref{DefnEqn-GammaB}) and the assumed control on $\epsilon$ {\bf H1.2}, equation (\ref{Eqn-MassDisperse}) suggests continuous ejection of mass from the region $\abs{y} < \frac{A}{2}$, regardless of whether that region is growing or contracting. 
\end{remark}
\begin{lemma}[Lyapounov Functional, \cite{MR-SharpLowerL2Critical-06} ]\label{Lemma-LyapounovFunc}
For all $s\in[s_0,s_1)$,
\begin{equation}\label{Eqn-Lyapounov}
\partial_s{\mathcal J} \leq -Cb\left(
	\Gamma_b + \epsTildeNorm + \int_{A\leq\abs{y}\leq 2A}{\abs{\epsilon}^2}
\right),
\end{equation}
where $C>0$ is a universal constant,
\begin{equation}\label{DefnEqn-LyapounovFunctional}\begin{aligned}
{\mathcal J}(s) = 
	&\int{\abs{\widetilde{Q}_b}^2} - \int{\abs{Q}^2} 
		+ 2Re\left(\epsilon,\widetilde{Q}_b\right)\\
	&+\frac{1}{r(s)}\int{\left(1-\phi_{\infty}\left(\frac{y}{A}\right)\right)\abs{\epsilon}^2\mu(y)\,dy}\\
	&-\frac{\delta_2}{800}\left(
		b\tilde{f}_1(b) - \int_0^b{\tilde{f}_1(v)\,dv} 
		+b\,Im\left(\epsilon,\Lambda\widetilde{\zeta}_b\right) \right),
\end{aligned}\end{equation}
where $\tilde{f}_1$ is the principal part of $f_1$, equation (\ref{DefnEqn-f1}),
\begin{equation}\label{DefnEqn-f1Tilde}
\tilde{f}_1(b) = 
	\frac{b}{4}\abs{y\widetilde{Q}_b}_{L^2}^2 
	+ \frac{1}{2}Im\left(\int{y\cdot\grad\widetilde{\zeta}_b\overline{\widetilde{\zeta}}_b}\right).
\end{equation}
\end{lemma}
Lemma \ref{Lemma-LyapounovFunc} is proven from the radiative virial estimate (\ref{Eqn-RadiativeVirial}), mass dispersion estimate (\ref{Eqn-MassDisperse}), and the conservation of mass.  The proof is provided at the end of the section.
Now let us discuss what ${\mathcal J}$ is.
\begin{lemma}[Estimates on Lyapounov Functional]\label{Lemma-LyapounovEstimates}
For all $s\in[s_0,s_1)$ we have the crude estimate,
\begin{equation}\label{Eqn-CrudeLyapounovEst}
\abs{ {\mathcal J} - d_0b^2 } < \delta_3b^2,
\end{equation}
where $0 < \delta_3 \ll 1$ is a universal constant, and $d_0b^2$ is the approximate excess mass of profile $\widetilde{Q}_b$ - see (\ref{Eqn-Qb_by_bSquared}). There also holds a more refined estimate,
\begin{equation}\label{Eqn-RefinedLyapounovEst}
{\mathcal J}(s) - f_2(b(s)) \left\{\begin{aligned}
	\leq
		\Gamma_b^{1-Ca} + CA^2&\left(\epsNorm\right)\\
	\geq
		-\Gamma_b^{1-Ca} + \frac{1}{C}&\left(\epsNorm\right),
\end{aligned}\right.
\end{equation}
where $f_2$ is the principal part of ${\mathcal J}$ concerned with mass of the profile,
\begin{equation}\label{DefnEqn-f2}
f_2(b) = 
	\int{\abs{\widetilde{Q}_b}^2} - \int{\abs{Q}^2} 
	-\frac{\delta_2}{800}\left(
		b\tilde{f}_1(b) - \int_0^b{\tilde{f}_1(v)\,dv} \right).
\end{equation}
\end{lemma}

\begin{proof}
To prove (\ref{Eqn-CrudeLyapounovEst}) we will approximate each term of (\ref{DefnEqn-LyapounovFunctional}). To estimate the term in $\abs{\epsilon}^2$, recall the support of $\phi_\infty$ (\ref{DefnEqn-phiInfty}) and the consequence for $\mu(y)$, such as equation (\ref{Hypo1-consequenceForMu}). Then,
\begin{equation}\label{Proof-CrudeLyapounov-eqn1}\begin{aligned}
\int{\left(1-\phi_\infty\left(\frac{y}{A}\right)\right)\abs{\epsilon}^2\mu(y)\,dy}
	&\lesssim \int_{\abs{y}\leq 3A}{\abs{\epsilon}^2}\\
	&\lesssim A^2\log A\left(\epsNorm\right) \leq \Gamma_b^\frac{1}{2},
\end{aligned}\end{equation}
where the second inequality is due to Lemma \ref{Lemma-L2ByGrad} 
and the final inequality is from the definition of $A$ (\ref{DefnEqn-A}) and the hypothesized control of $\epsilon$.  Estimate $(\epsilon,\widetilde{Q}_b)$ by the same control, and the terms in $\widetilde{\zeta}_b$ by (\ref{Eqn-zb_smallH1}). Equation (\ref{Eqn-CrudeLyapounovEst}) then follows from (\ref{Eqn-Qb_by_bSquared}) by noting that the constant $\delta_2$ due to the radiative virial identity (\ref{Eqn-RadiativeVirial}) can be assumed small with respect to universal constant $d_0$, so that $0 < \left.\frac{\partial f_2}{\partial b^2}\right|_{b^2=0} < \infty$.
Next we prove the refined estimate. Note that,
\begin{equation}\label{Proof-RefinedLyapounov-eqn1}
{\mathcal J}(s) - f_2(b(s)) = 
	2Re\left(\epsilon,\widetilde{Q}_b\right) 
	+ \frac{1}{r(t)}\int{\left(1-\phi_\infty\right)\abs{\epsilon}^2\mu(y)}
	- \frac{\delta_2}{800}b\,Im\left(\epsilon,\Lambda\widetilde{\zeta}_b\right).
\end{equation}
By the bounds for $\widetilde{\zeta}_b$, 
Lemma \ref{Lemma-L2ByGrad}, 
and the choice of $A$, 
\begin{equation}\label{Proof-RefinedLyapounov-eqn2}\begin{aligned}
\abs{Im\left(\epsilon,\Lambda\widetilde{\zeta}_b\right)}
&\leq \Gamma_b^{\frac{1}{2}-C\eta}
	\left(\int_{\abs{y} \leq A}\abs{\epsilon}^2\right)^\frac{1}{2}\\
&\lesssim \Gamma_b^{\frac{1}{2}-C\eta}A\left(\log A\right)^\frac{1}{2}
	\left( \epsNorm \right)^\frac{1}{2}\\
&\lesssim \Gamma_b^{1-Ca} + \left(\epsNorm\right).
\end{aligned}\end{equation}
Since $b$ is small, the contribution of (\ref{Proof-RefinedLyapounov-eqn2}) is a factor of $\alpha^*$ smaller than the desired bound. Similar terms will be omitted for the remainder of the proof.
Regarding the two other terms of (\ref{Proof-RefinedLyapounov-eqn1}), the term linear in $\epsilon$ we recognize from the conservation of energy (\ref{Eqn-prelimEnerEst}). Indeed, the upper bound for (\ref{Proof-RefinedLyapounov-eqn1}) follows from (\ref{Eqn-prelimEnerEst}) with (\ref{Eqn-ExpDecayByGrad}) and,
\begin{equation}\label{Proof-RefinedLyapounov-eqn3}
\int{\left(1-\phi_\infty\right)\abs{\epsilon}^2\mu(y)\,dy} \lesssim A^2\log A\left(\epsNorm\right)
\end{equation}
which is due to (\ref{Eqn-L2ByGrad}).

To establish a lower bound for (\ref{Proof-RefinedLyapounov-eqn1}) we will need the following Lemma - the proof is based on a spectral result due to \cite{MartelMerle01}, with additional properties proven \cite{Maris02} and \cite{McLeod93}. See \cite[Lemma 8]{MR-SharpLowerL2Critical-06} for that spectral property, and \cite[Appendix D]{MR-SharpLowerL2Critical-06} for proof of the Lemma.
\begin{lemma}[Elliptic estimate for $L$.]
\label{Lemma-EllipticEstimate}
Recall the linearized Schr\"odinger operator $L$ from (\ref{DefnEqn-L}).
There exists a universal constant $\delta_4 > 0$ such that $\forall v \in H^1(\real^2)$,
\begin{equation}\label{EllipticEstimate}\begin{aligned}
Re\left(L(v),v\right) - \int{\phi_\infty\abs{v}^2}
	\geq &\delta_4\left(\int{\abs{\grad v}^2} + \int{\abs{v}^2e^{-\abs{y}}}\right)\\
	&-\frac{1}{\delta_4}\left(
		Re\left(v,Q\right) + Re\left(v,\abs{y}^2Q\right) 
		+ Re\left(v,yQ\right) + Im\left(v,\Lambda^2Q\right)
	\right)^2.
\end{aligned}\end{equation}
\end{lemma}
Introduce a new radially symmetric cutoff function, analogous to $\phi_A$ (\ref{DefnEqn-phiA}) but with larger support such that $(1-\phi_B(y))(1-\phi_\infty(\frac{y}{A})) = 0$.
\begin{equation}\begin{aligned}
\label{DefnEqn-phiB}
\phi_B(y) = \left\{\begin{aligned}
	1 && \text{ for }\abs{y} \leq 3A\\
	0 && \text{ for }\abs{y} \geq 4A,
	\end{aligned}\right.
\end{aligned}\end{equation}
From equation (\ref{Hypo1-consequenceForMu}), we may rewrite the principal part of the conservation of energy estimate (\ref{Eqn-prelimEnerEst}) as,
\begin{equation}\label{Proof-RefinedLyapounov-eqn4}\begin{aligned}
2Re\left(\epsilon,\widetilde{Q}_b\right) \approx& 
	\int{\left(1-\phi_B^2\right)\abs{\grad\epsilon}^2\mu(y)\,dy}\\
	&+\int{\phi_B^2\abs{\grad\epsilon}^2\,dy}
	-3\int{Q^2(\phi_B\epsRe)^2} - \int{Q^2(\phi_B\epsIm)^2},
\end{aligned}\end{equation}
where we used the exponential spatial decay of $Q$ and the lower bound for $\Gamma_b$ (\ref{DefnEqn-GammaB}) to control the excess in $Q^2\epsilon^2$ on $\abs{y} > \frac{10}{b}$. 
With integration by parts,
\begin{equation}\label{Proof-RefinedLyapounov-eqn5.5}\begin{aligned}
\int{\phi_B^2\abs{\grad\epsilon}^2\,dy} 
=	\int{\abs{\grad(\phi_B\epsilon)}^2\,dy}
	+\int{\laplacian\phi_B\,\phi_B\abs{\epsilon}^2\,dy}.
\end{aligned}\end{equation}
The principal part of (\ref{Proof-RefinedLyapounov-eqn1}) is then,
\begin{equation}\label{Proof-RefinedLyapounov-eqn7}\begin{aligned}
2Re\left(\epsilon,\widetilde{Q}_b\right) 
	&+ \frac{1}{r(t)}\int{(1-\phi_\infty)\abs{\epsilon}^2\mu(y)\,dy}\\
	\approx& 
	\int{\left(1-\phi_B^2\right)\abs{\grad\epsilon}^2\mu(y)\,dy}\\
	&+\left( Re\left(L(\phi_B\epsilon),\phi_B\epsilon\right) 
		- \int{\phi_\infty\abs{\phi_B\epsilon}^2} \right) 
		+ \int{\laplacian\phi_B\,\phi_B\abs{\epsilon}^2}\\
	&+\int{ \left(1-\phi_\infty\right)\left(\frac{\mu}{r(t)}-\phi_B^2\right) 
		\abs{\epsilon}^2}.
\end{aligned}\end{equation}
The final term can be neglected, as $\left(1-\phi_\infty\right)\left(\frac{\mu}{r(t)}-\phi_B^2\right)$ is of the order $\lambda y_1$, and supported on $\abs{y} < 4A$. 
The lower bound for (\ref{Proof-RefinedLyapounov-eqn1}) then follows from Lemma \ref{Lemma-EllipticEstimate}, an integration by parts, and the straightforward comparison,
\begin{equation}\label{Proof-RefinedLyapounov-eqn8}
\int{\phi_B^2\abs{\grad\epsilon}^2} + \int{\abs{\phi_B\epsilon}^2e^{-\abs{y}}}
\gtrsim \int{\phi_B^2\abs{\grad\epsilon}^2\mu(y)} + \int_{\abs{y}\leq\frac{10}{b}}{\abs{\epsilon}^2e^{-\abs{y}}},
\end{equation}
again due to the support of $\phi_B$ and the bound on $\lambda$.
This completes the proof of (\ref{Eqn-RefinedLyapounovEst}).
\end{proof}

\begin{lemma}[Lower Bound on Blowup Rate]\label{Lemma-Lowerbound}
For all $s\in[s_0,s_1)$,
\begin{equation}\label{Eqn-bUpperBound}
b(s) \leq \frac{4\pi}{3\log s},
\end{equation}
\begin{equation}\label{Eqn-epsIntegral}
\int_{s_0}^s{\left(
\Gamma_{b(\sigma)} + \epsNorm
\right)\,d\sigma}
\leq C \alpha^*, 
\end{equation}
where $C > 0$ is a universal constant and,
\begin{equation}\label{Eqn-epsImproved}
\epsNorm \leq \Gamma_b^\frac{4}{5},
\end{equation}
which is equation (\ref{Improv1-epsilon}), the remaining part of {\bf I1.2}.
\end{lemma}
Note that (\ref{Eqn-bUpperBound}) is the 1\textsuperscript{st} upper bound of {\bf I1.3}. The only estimate remaining to establish Proposition \ref{Prop-Improv} follows as a corollary,
\begin{proof}[Proof of {\bf I1.3}, 2\textsuperscript{nd} lower bound]
Recall the approximate dynamics of $\lambda$, equation 
(\ref{Eqn-lambda-prelimDynamics}).  Since $b>0$ is small, 
$-\frac{\lambda_s}{\lambda} \leq 3b$, which we integrate with (\ref{Eqn-bUpperBound}),
\begin{equation}\label{Proof-Improv1-loglog-eqn3}
-\log \lambda(s) 
	\leq -\log\lambda_0 + 4\pi\int_{s_0}^s{\frac{1}{\log\sigma}\,d\sigma} 
	\leq -\log\lambda_0 + 4\pi (s-s_0).
\end{equation}
Use (\ref{Eqn-bUpperBound}) again, and recall the definition of $s_0$ (\ref{DefnEqn-s}) and choice of data (\ref{DataP1-loglog}), 
\begin{equation}\label{Eqn-lambdaLowerBound}
\lambda(s) 
	\geq \lambda_0e^{4\pi s_0}\,e^{-4\pi e^\frac{4\pi}{3b(s)}} 
	> e^{-e^\frac{5\pi}{b(s)}}.
\end{equation}
\end{proof}

\begin{proof}[Proof of Lemma \ref{Lemma-Lowerbound}]
To begin, note from the crude estimate (\ref{Eqn-CrudeLyapounovEst}) that we may divide the Lyapounov inequality (\ref{Eqn-Lyapounov}) by $\sqrt{\mathcal J}$ and integrate in time leaving,
\begin{equation}\label{Proof-LowerBound-eqn1}
\int_{s_0}^s{\left(
\Gamma_{b(\sigma)} + \epsNorm
\right)\,d\sigma}
\leq C\left(\sqrt{{\mathcal J}(s_0)} - \sqrt{{\mathcal J}(s)}\right)
\leq Cb_0.
\end{equation}
The choice of data (\ref{DataP1-mass}) then proves (\ref{Eqn-epsIntegral}). Alternately, we may view the crude estimate (\ref{Eqn-CrudeLyapounovEst}) and the Lyapounov inequality (\ref{Eqn-Lyapounov}) as giving a differential inequality for ${\mathcal J}$,
\begin{equation}\label{Proof-LowerBound-eqn2}\begin{aligned}
\partial_se^{+\frac{5\pi}{4}\sqrt{\frac{d_0}{{\mathcal J}}}}
	\gtrsim \frac{b}{{\mathcal J}}\,\Gamma_be^{\frac{5\pi}{4}\sqrt{\frac{d_0}{{\mathcal J}}}} 
	\geq 1
&& \Longrightarrow 
&& e^{+\frac{5\pi}{4}\sqrt{\frac{d_0}{{\mathcal J}(s)}}} \geq 
		e^{+\frac{5\pi}{4}\sqrt{\frac{d_0}{{\mathcal J}(s_0)}}}
		+s -s_0.
\end{aligned}\end{equation}
Note that here we applied the bound on $\Gamma_b$ (\ref{Eqn-GammaBEstimate}), for which it was essential $\frac{5}{4} > 1+C\eta$ - see Remark \ref{Remark-NormalBrezisFails}.
By crude estimate (\ref{Eqn-CrudeLyapounovEst}) and the definition of $s_0$ (\ref{DefnEqn-s}),
\begin{equation}\label{Proof-LowerBound-eqn2.5}
e^{+\frac{5\pi}{4}\sqrt{\frac{d_0}{{\mathcal J}(s_0)}}}
	> e^\frac{\pi}{b_0} > s_0,
\end{equation}
which, again with estimate (\ref{Eqn-CrudeLyapounovEst}), proves (\ref{Eqn-bUpperBound}) from (\ref{Proof-LowerBound-eqn2}).

It remains to establish the pointwise control of $\epsilon$. Fix $s \in[s_0,s_1)$.
\begin{enumerate}
\item If $\partial_sb(s) \leq 0$, then (\ref{Eqn-epsImproved}) follows from the local virial identity, Lemma \ref{Lemma-LocalVirial}. 
\item If $\partial_sb(s) > 0$, then there exists a largest interval $(s_+,s)$, with $s_0\leq s_+$, on which $\partial_sb >0$.
\[
\begin{aligned}
\text{This implies, }
&& b(s_+) < b(s)
&& \text{ and either, }
&& 	\left.\begin{aligned}
		\left(a\right) &&&s_+ = s_0, \\
		\;\;\text{ or,}\\
		\left(b\right) &&&\partial_sb(s_+) = 0.
		\end{aligned}\right. 
\end{aligned}\]
In case (a) or (b), by the choice of small $\epsilon_0$ 
 or the local virial identity, respectively, 
\[
\int{\abs{\grad_y\epsilon(s_+,y)}^2\mu(y)\,dy}
	+\int_{\abs{y}\leq\frac{10}{b(s_+)}}{\abs{\epsilon(s_+,y)}^2e^{-\abs{y}}\,dy}
	\leq \Gamma_{b(s_+)}^\frac{6}{7}.
\]
From the upper bound of refined estimate (\ref{Eqn-RefinedLyapounovEst}), and assuming $a>0$ is sufficiently small,
\begin{equation}\label{Proof-LowerBound-eqn5}
{\mathcal J}(s_+) - f_2(b(s_+)) \leq \Gamma_{b(s_+)}^\frac{5}{6} < \Gamma_{b(s)}^\frac{5}{6}.
\end{equation}
Since ${\mathcal J}$ is non-increasing, 
and from the lower bound of refined estimate (\ref{Eqn-RefinedLyapounovEst}),
\begin{equation}\label{Proof-LowerBound-eqn6}\begin{aligned}
\Gamma_{b(s)}^\frac{5}{6} 
\geq & 
{\mathcal J}(s) - f_2(b(s_+))\\
\gtrsim & 
\left(\int{\abs{\grad_y\epsilon(s,y)}^2\mu(y)\,dy}
	+\int_{\abs{y}\leq\frac{10}{b(s)}}{\abs{\epsilon(s,y)}^2e^{-\abs{y}}\,dy}
	\right)\\
&\qquad-\Gamma_{b(s)}^{1-Ca} + \left(f_2(b(s)) - f_2(b(s_+))\right).
\end{aligned}\end{equation}
As noted in the proof of crude estimate (\ref{Eqn-CrudeLyapounovEst}), we may assume the constant $\delta_2$ of equation (\ref{Eqn-RadiativeVirial}) is sufficiently small relative to $d_0$, such that $0 < \left.\frac{\partial f_2}{\partial_{b^2}}\right|_{b^2=0} < \infty$, and proving that $\left(f_2(b(s)) - f_2(b(s_+))\right) > 0$. Assuming $a>0$ is sufficiently small, this proves (\ref{Eqn-epsImproved}).
\end{enumerate}
\end{proof}

\begin{proof}[Proof of Lemma \ref{Lemma-MassDisperse}, \cite{MR-SharpLowerL2Critical-06}]
Directly from equation (\ref{Eqn-NLS}),
\begin{multline}\label{Proof-MassDisperse-eqn1}
\frac{1}{2}\partial_s\left(\int{
	\phi_{\infty}\left(\frac{(r,z) - (r(t),z(t))}{\lambda A}\right)
	\abs{u}^2\,dx}\right)\\
\begin{aligned}=&
\frac{1}{\lambda A}Im\left(\int{\grad_x\phi_\infty\left(\frac{y}{A}\right)
\cdot\grad_x u\,\overline{u}\,dx}\right)\\
&-\frac{1}{2\lambda^2 A}\int{\left( 
		\left(\frac{\lambda_s}{\lambda} + \frac{A_s}{A}\right)
		y
		+\frac{\partial_s(r,z)}{\lambda}\right)
	\cdot\grad_x\phi_\infty\left(\frac{y}{A}\right)
	\abs{u}^2\,dx}.
\end{aligned}
\end{multline}
From the choice of $A$ (\ref{DefnEqn-A}) and $\phi_\infty$ (\ref{DefnEqn-phiInfty}), the support of $\widetilde{Q}_b$ and $\phi_\infty\left(\frac{y}{A}\right)$ are disjoint. With the geometric decomposition 
and change of variables we may rewrite (\ref{Proof-MassDisperse-eqn1}) in terms of $\abs{\epsilon}^2$,
\begin{multline}\label{Proof-MassDisperse-eqn2}
\frac{1}{2}\frac{d}{ds}\int{\phi_\infty\left(\frac{y}{A}\right)\abs{\epsilon}^2\mu(y)\,dy}\\
\begin{aligned}
=&
\frac{1}{A}Im\left(\int{\grad_x\phi_\infty\left(\frac{y}{A}\right)\cdot\grad_y\epsilon\,\overline{\epsilon}\mu(y)\,dy}\right)
+ 
\frac{b}{2}\int{\frac{y}{A}\cdot\grad_x\phi_\infty\left(\frac{y}{A}\right)\abs{\epsilon}^2\mu(y)\,dy}\\
&-\frac{1}{2A}\int{\left(\left(\frac{\lambda_s}{\lambda}+b+\frac{A_s}{A}\right)y+\frac{\partial_s(r,z)}{\lambda}\right)
	\cdot\grad_x\phi_\infty\left(\frac{y}{A}\right)\abs{\epsilon}^2\mu(y)\,dy}.
\end{aligned}
\end{multline}
By Cauchy-Schwarz, the definition of $A$ (\ref{DefnEqn-A}) and the lower bound on $\Gamma_b$ (\ref{Eqn-GammaBEstimate}),
\begin{multline}\label{Proof-MassDisperse-eqn3}
\abs{\frac{1}{A}Im\left(\int{\grad_x\phi_\infty\left(\frac{y}{A}\right)\cdot\grad_y\epsilon\,\overline{\epsilon}\mu(y)\,dy}\right)}\\
\begin{aligned}
	\leq&\frac{1}{A}\left(\int{\abs{\grad\epsilon}^2\mu(y)\,dy}\right)^\frac{1}{2}
		\left(\int{\abs{\grad_x\phi_\infty\left(\frac{y}{A}\right)}\abs{\epsilon}^2\mu(y)\,dy}\right)^\frac{1}{2}\\
	\leq&\frac{1}{2}\Gamma_b^\frac{a}{2}\int{\abs{\grad\epsilon}^2\mu(y)\,dy}
		+\frac{b}{40}\int{\abs{\grad_x\phi_\infty\left(\frac{y}{A}\right)}\abs{\epsilon}^2\mu(y)\,dy}.
\end{aligned}
\end{multline}
The factor $\frac{b}{40}$ is arbitrary by assuming $b$ is sufficiently small.
The following term is the principal part of (\ref{Proof-MassDisperse-eqn2}): from the support of ${\phi_\infty}'$, and that ${\phi_\infty}'\geq 0$, equation (\ref{DefnEqn-phiInfty}),
\begin{equation}\label{Proof-MassDisperse-eqn4}
\frac{b}{2}\int{\frac{y}{A}\cdot\grad_x\phi_\infty\left(\frac{y}{A}\right)\abs{\epsilon}^2\mu(y)\,dy}
\geq \frac{b}{5}\int{\abs{\grad_x\phi_\infty\left(\frac{y}{A}\right)}\abs{\epsilon}^2\mu(y)\,dy}.
\end{equation}
Regarding the last line of (\ref{Proof-MassDisperse-eqn2}), apply preliminary estimates (\ref{Eqn-prelimLambda+BEst}) and (\ref{Eqn-prelimGamma+REst}), the support of ${\phi_\infty}'$, 
and the definition of $A$ 
to estimate,
\begin{equation}\label{Proof-MassDisperse-eqn5}
\abs{\frac{1}{2A}\left(\left(\frac{\lambda_s}{\lambda}+b+\frac{A_s}{A}\right)y+\frac{\partial_s(r,z)}{\lambda}\right)}
\leq \frac{b}{40}.
\end{equation}
Due to the bounds for ${\phi_\infty}'\left(\frac{y}{A}\right)$ on $A\leq\abs{y}\leq 2A$, and lower bounds for $\mu$ similar to equation (\ref{Hypo1-consequenceForMu}),
\begin{equation}\label{Proof-MassDisperse-eqn6}
\int{\abs{\grad_x\phi_\infty\left(\frac{y}{A}\right)}\abs{\epsilon}^2\mu(y)\,dy}
\geq \frac{1}{6}\int_{A \leq \abs{y} \leq 2A}{\abs{\epsilon}^2\,dy}.
\end{equation}
From (\ref{Proof-MassDisperse-eqn2}) we have proven,
\begin{equation}\label{Proof-MassDisperse-eqn7}
\frac{d}{ds}\int{\phi_\infty\left(\frac{y}{A}\right)\abs{\epsilon}^2\mu(y)\,dy}
\geq
\frac{b}{20}\int_{A \leq \abs{y} \leq 2A}{\abs{\epsilon}^2\,dy} -
\Gamma_b^\frac{a}{2}\int{\abs{\grad\epsilon}^2\mu(y)\,dy}.
\end{equation}
Finally note that by  preliminary estimate (\ref{Eqn-prelimGamma+REst}), $r(t) \sim 1$ from {\bf H1.1}, change of variables, and the log-log relationship (\ref{Hypo1-consequenceForLambdaGamma}), we have the easy estimate,
\begin{equation}
\abs{\frac{r_s}{r^{2}(t)}\int{\phi_\infty\left(\frac{y}{A}\right)\abs{\epsilon}^2\mu(y)\,dy}}
\ll \lambda \int{\abs{\tilde{u}}^2} \ll \Gamma_b^2.
\end{equation}
This completes the proof of Lemma \ref{Lemma-MassDisperse}. 
\end{proof}

\begin{proof}[Proof of Lemma \ref{Lemma-LyapounovFunc}]
Multiply the radiative virial identity (\ref{Eqn-RadiativeVirial}) by $\frac{\delta_2 b}{800}$, and sum with the mass ejection equation (\ref{Eqn-MassDisperse})
to cancel the bad sign of $\int_{A\leq\abs{y}\leq 2A}{\abs{\epsilon}^2}$,
\begin{multline}\label{Proof-LyapounovFunc-eqn1}
\partial_s\left(\frac{1}{r^k(t)}\int{\phi_{\infty}\left(\frac{y}{A}\right)\abs{\epsilon}^2\mu(y)\,dy}\right)
+ \frac{\delta_2b}{800}\partial_s f_1 \\
\begin{aligned}
	\geq& \frac{\delta_2^2b}{800}\left(\epsTildeNorm\right)\\
		&+\frac{b}{800}\int_{A\leq \abs{y} \leq 2A}{\abs{\epsilon}^2\,dy}
		+\frac{\delta_2b}{1000}\Gamma_b
		-\Gamma_b^\frac{a}{2}\int{\abs{\grad_y\epsilon}^2\mu(y)\,dy}.
\end{aligned}
\end{multline}
The final term of (\ref{Proof-LyapounovFunc-eqn1}) has the bad sign.
Recall $\epsilon = \widetilde{\epsilon} + \widetilde{\zeta}_b$, equation (\ref{DefnEqn-epsTilde}), and $\widetilde{\zeta}_b$ is small in $\dot{H}^1$, equation (\ref{Eqn-zb_smallH1}), with support on which we may estimate $\mu$ so that,
\begin{equation}\label{Proof-LyapounovFunc-eqn2}\begin{aligned}
\Gamma_b^\frac{a}{2}\int{\abs{\grad_y\epsilon}^2\mu(y)\,dy}
	&\lesssim \Gamma_b^\frac{a}{2}\left(
		\Gamma_b^{1-C\eta} + \int{\abs{\grad\widetilde{\epsilon}}^2\mu(y)\,dy}
	\right)\\
	&\leq \Gamma_b^{1+\frac{a}{4}} 
		+ \Gamma_b^\frac{a}{2}\int{\abs{\grad\widetilde{\epsilon}}^2\mu(y)\,dy},
\end{aligned}\end{equation}
where for the second inequality we require $a > 4C\eta$ 
	- see Remark \ref{Remark-ChoiceOf-a}.
To rewrite $\frac{\delta_2b}{800}\partial_s f_1$, note that,
\begin{equation}\label{Proof-LyapounovFunc-eqn3}\begin{aligned}
b\partial_s f_1 = \partial_s\left(
	b\tilde{f}_1(b) - \int_0^b{\tilde{f}_1(v)\,dv} 
		+b Im\left(\epsilon,\Lambda\widetilde{\zeta}_b\right) 
	\right)
	- \partial_s b \,Im\left(\epsilon,\Lambda\widetilde{\zeta}_b\right),
\end{aligned}\end{equation}
where $\tilde{f}_1$ is the principal part of $f_1$, given by equations (\ref{DefnEqn-f1Tilde}) and (\ref{DefnEqn-f1}), respectively. 
Estimate the final term of (\ref{Proof-LyapounovFunc-eqn3}) with a combination of preliminary estimate (\ref{Eqn-prelimLambda+BEst}), H\"older, Lemma \ref{Lemma-L2ByGrad}, and the hypothesis $\epsilon$ {\bf H1.2}.
Equation (\ref{Proof-LyapounovFunc-eqn1}) has transformed into,
\begin{multline}\label{Proof-LyapounovFunc-eqn4}
\partial_s\left(
\frac{1}{r(t)}\int{\phi_{\infty}\left(\frac{y}{A}\right)\abs{\epsilon}^2\mu(y)\,dy}
+
\frac{\delta_2}{800}\left(
	b\tilde{f}_1(b) - \int_0^b{\tilde{f}_1(v)\,dv} 
		+b Im\left(\epsilon,\Lambda\widetilde{\zeta}_b\right) 
	\right)
\right)\\
\geq 
\frac{\delta_2^2b}{800}\left(
	\epsTildeNorm + \int_{A\leq \abs{y} \leq 2A}{\abs{\epsilon}^2\,dy}
	\right)
+\frac{\delta_2b}{2000}\Gamma_b.
\end{multline}
To identity the LHS of (\ref{Proof-LyapounovFunc-eqn4}) with $-\partial_s{\mathcal J}$, inject the conservation of mass, $\int_{\real^{3}}{\abs{u(t)}^2} = \int{\abs{u_0}^2}$. 
As we did before, equation (\ref{Proof-prelimMass-eqn1}), rewrite $u(t)$ with the geometric decomposition, expand the product, change variables, expand the measure $\mu$,
divide by $r(t)$, and take the derivative $\partial_s$,
\begin{equation}\label{Proof-LyapounovFunc-eqn6}\begin{aligned}
\partial_s\left(\frac{1}{r(t)}\int{\phi_{\infty}\left(\frac{y}{A}\right)\abs{\epsilon}^2\mu(y)\,dy}\right)
=&
-\partial_s\left(\int{\abs{\widetilde{Q}_b}^2} - \int{\abs{Q}^2} 
	+ 2Re\left(\epsilon,\widetilde{Q}_b\right)\right)\\
& -\partial_s\left(\frac{1}{r(t)}
		\int{\lambda y_1P(\lambda y_1,r(t))\left(\abs{\widetilde{Q}_b}^2 
			+ 2Re\left(\epsilon\overline{\widetilde{Q}_b}\right)\right)}
		\right)\\
	&-\frac{\partial_sr}{r^{2}(t)}\int{\abs{u_0}^2}.
\end{aligned}\end{equation}
Through a combination of preliminary estimates (\ref{Eqn-prelimLambda+BEst}) and (\ref{Eqn-prelimGamma+REst}), the $\epsilon$-equation (\ref{DefnEqn-epsEqn}), and the log-log rate (\ref{Hypo1-consequenceForLambdaGamma}),
\[
\abs{-\partial_s\left(\frac{1}{r(t)}
		\int{\lambda y_1P(\lambda y_1,r(t))\left(\abs{\widetilde{Q}_b}^2 
			+ 2Re\left(\epsilon\overline{\widetilde{Q}_b}\right)\right)}
		\right)}
\lesssim \lambda < \Gamma_b^2.
\]
Likewise, $\abs{\frac{\partial_sr}{r^{2}(t)}}\int{\abs{u_0}^2} \lesssim \lambda\int{\abs{u_0}^2} < \Gamma_b^2$. 
Inserting equation (\ref{Proof-LyapounovFunc-eqn6}) into (\ref{Proof-LyapounovFunc-eqn4}) completes the proof of Lemma \ref{Lemma-LyapounovFunc}.
\end{proof}

\section{Proof of Global Behaviour}
\label{Section-BootAtInfty}

In this chapter we prove that properties {\bf I2.1} through {\bf I2.3} are a consequence of hypotheses {\bf H1.1} through {\bf H2.3}.  The following properties of the singular dynamic proven in Chapter \ref{Section-BootLoglog} will be used:
the specific log-log rate, the geometric decomposition and resulting control on $b_s$,  
and the integrability of $\norm{\tilde{u}}_{L^2_tH^1_x}$.

\subsection{Growth of \texorpdfstring{$\norm{u}_{H^{3}}$ }{u in H\textthreesuperior}} 
It is left until Chapter \ref{Section-FinalProof} to show that $\lambda^{-1}$ follows the log-log rate (\ref{Thm-MainResult-LogLog}).  Here, we use the log-log rate in the form {\bf H1.3}, and the control of $b_s$, to prove directly that $\lambda^{-1}(t)$ has the same integrability in time as $\sqrt{\frac{\log\abs{\log(T-t)}}{T-t}}$.
The following Lemma was previously noted, \cite[equation (51)]{RaphaelSzeftel-StandingRingNDimQuintic-08}.
\begin{lemma}[Integrability due to log-log rate]\label{Lemma-lambdaIntegralCrude}
Let $0 \leq \mu <2$ and $\sigma_1 \in \real$. Then,
\begin{equation}\label{Eqn-crudeLambdaIntegral}
\int_0^t{\frac{e^{\frac{\sigma_1}{b(\tau)}}}{\lambda^\mu(\tau)}\,d\tau} \lesssim C(\mu,\sigma_1,\alpha^*),
\end{equation}
where for fixed $\mu$ and $\sigma_1$, $C(\mu,\sigma_1,\alpha^*)$ decays much faster than $e^{-\frac{1}{\alpha^*}}$ as $\alpha^* \rightarrow 0$.
\end{lemma}
\begin{proof}
From the log-log rate, {\bf H1.3}, 
\[\begin{aligned}
e^\frac{\pi}{10 b} < \abs{\log\lambda}
&&\text{ and, } &&
\frac{\pi}{10b} > \frac{1}{100}\log s && \Rightarrow && \frac{1}{\lambda} > e^{+e^{\frac{\pi}{10b}}} > e^{\left(s^\frac{1}{100}\right)}.
\end{aligned}\]
Then by change of variables 
and almost-monotony of $\lambda$, {\bf H1.5}, 
\[
\int_0^t{\frac{e^{\frac{\sigma_1}{b(\tau)}}}{\lambda^\mu(\tau)}\,d\tau} < \int_{s_0}^s{\frac{\abs{\log\lambda}^\frac{10\sigma_1}{\pi}}{\lambda^{\mu-2}(\tau')}\,d\tau'}
\lesssim \frac{\abs{\log\lambda(t)}^\frac{10\sigma_1}{\pi}}{\lambda^{\mu-2}(t)}\left(s(t)-s_0\right)
\lesssim e^{(\mu-2)s^\frac{1}{100}(t)}.
\]
Finally, to prove the behaviour of $C(\mu,\sigma_1,\alpha^*)$, recall $s(t) \geq s_0 = e^\frac{3\pi}{4b_0}$, and $b_0 < \alpha^*$.
\end{proof}
\begin{remark}[Lemma \ref{Lemma-lambdaIntegralCrude} for $\mu \geq 2$]
From the log-log rate, {\bf H1.3}, 
$s(t)-s_0 \lesssim e^{\frac{10}{\pi}\frac{1}{b(t)}}$, so by the same proof,
\begin{equation}\label{Eqn-outdatedLambdaIntegral}
\int_0^t{\frac{1}{\lambda^\mu}} \lesssim \frac{e^{\frac{10}{\pi}\frac{1}{b(t)}}}{\lambda^{\mu-2}}.
\end{equation}
This is the primary integrability tool of \cite{RaphaelSzeftel-StandingRingNDimQuintic-08}. The following improvement will be crucial,
\end{remark}
\begin{lemma}[Refined Integrability due to control of $b_s$]\label{Lemma-lambdaIntegral}
Let $\mu > 2$, $\sigma^*$ arbitrary and assume $\alpha^*>0$ is sufficiently small. Then for any $\sigma_2 > 0$ and all $t\in[0,T_{hyp})$,
\begin{equation}\label{Eqn-lambdaIntegral}
\int_0^t{\frac{e^{-\frac{\sigma^*}{b(\tau)}}}{\lambda^\mu(\tau)}\,d\tau} 
	\leq 
	C(\mu,\sigma_2,\alpha^*)\,\frac{
		e^{-\frac{\sigma^*}{b(t)}}e^{+\frac{\sigma_2}{b(t)}}
	}{\lambda^{\mu-2}(t)},
\end{equation}
where, for fixed $\mu$ and $\sigma_2$, $C(\mu,\sigma_2,\alpha^*) \rightarrow 0$ as $\alpha^* \rightarrow 0$.
\end{lemma}
\begin{proof}
To begin, we prove the case $\sigma^* = 0$. By direct calculation,
\[
\frac{d}{ds}\left(\frac{1}{b}\frac{1}{\lambda^{\mu-2}}\right)
=
\frac{1}{\lambda^{\mu-2}}\left((\mu-2) - \frac{b_s}{b^2} - (\mu-2)\frac{\frac{\lambda_s}{\lambda}+b}{b}\right).
\]
For $\alpha^*$ sufficiently small relative to $\mu$, 
from {\bf H1.2} 
and the control of $b_s$, (\ref{Eqn-prelimLambda+BEst}), 
\begin{equation}\label{Proof-lambdaIntegral-eqn1}
\frac{1}{\lambda^\mu} 
\leq
C(\mu)\frac{1}{\lambda^2}\,\frac{d}{ds}\left(\frac{1}{b}\frac{1}{\lambda^{\mu-2}}\right)
=
C(\mu)\frac{d}{dt}\left(\frac{1}{b}\frac{1}{\lambda^{\mu-2}}\right).
\end{equation}
After integration, we estimate $C(\mu)\frac{1}{b}\frac{1}{\lambda^{\mu-2}} \leq C(\mu,\sigma_2,\alpha^*)\frac{e^{+\frac{\sigma_2}{b}}}{\lambda^{\mu-2}}$.
For those cases where $\sigma^* \neq 0$, integrate by parts,
\begin{equation}\label{Proof-lambdaIntegral-eqn2}
\int_0^t{\frac{ e^{-\frac{\sigma^*}{b(\tau)}} }{\lambda^\mu(\tau)}d\,\tau}
=\left.
	e^{-\frac{\sigma^*}{b(\tau)}}\int_0^\tau{\frac{1}{\lambda^\mu(\tau')}d\,\tau'}
\right|_0^t
-\int_0^t{\sigma^*\left(\frac{b_\tau}{b^2}e^{-\frac{\sigma^*}{b(\tau)}}\int_0^\tau{\frac{1}{\lambda^\mu(\tau')}d\,\tau'}\right)d\,\tau}.
\end{equation}
Apply the previous case to the first RH term.  For the second term, make the change of variable $b_\tau = \frac{b_{s(\tau)}}{\lambda^2(\tau)}$ and apply the previous case for some $\sigma_2 << \frac{1}{2}$.  Use (\ref{Eqn-prelimLambda+BEst}) to approximate $b_s$, and we have bound the second term by a small multiple of the LHS.
\end{proof}
Lemma \ref{Lemma-lambdaIntegral} is simply not true for $\mu = 2$. As a substitute, we prove a corollary of the integrated Lyapounov inequality, equation (\ref{Eqn-epsIntegral}).
\begin{corollary}\label{Corollary-lambdaIntegral2version2}
Let $\sigma_3 \geq 0$.  Then for all $t\in[0,T_{hyp})$,
\begin{equation}\label{Eqn-lambdaIntegral2version2}
\int_0^t{e^{\frac{\sigma_3}{b(\tau)}}
	\left(\norm{\tilde{u}(\tau)}_{H^1}^2 + \frac{\Gamma_{b(\tau)}}{\lambda^2(\tau)}\right)
	\,d\tau}
	\lesssim C(\alpha^*)e^{\frac{\sigma_3}{b(t)}}.
\end{equation}
\end{corollary}
\begin{proof}
By change of variables and integration by parts,
\[\begin{aligned}
\int_0^t{e^{\frac{\sigma_3}{b(\tau)}}
	\left(\norm{\tilde{u}(\tau)}_{H^1}^2 + \frac{\Gamma_{b(\tau)}}{\lambda^2(\tau)}\right)
	\,d\tau}
	=&\left.e^{\frac{\sigma_3}{b(\sigma)}}
		\int_0^\sigma{\lambda^2(\sigma ')\norm{\widetilde{u}(\sigma ')}_{H^1_x}^2 
			+ \Gamma_{b(\sigma ')}\,d\sigma '}
		\right|_{s_0}^{s(t)}\\
	&+\sigma_3\int_{s_0}^{s(t)}{\frac{b_s}{b^2}e^{\frac{\sigma_3}{b(\sigma)}}
		\left(\int_0^\sigma{\lambda^2\norm{\widetilde{u}}_{H^1_x}^2+\Gamma_b}\right)\,d\sigma}.
\end{aligned}\]
Then observe the control on $b_s$ (\ref{Eqn-prelimLambda+BEst}) and estimate (\ref{Eqn-epsIntegral}).
\end{proof}
\begin{remark}[Optimality of (\ref{Eqn-lambdaIntegral2version2})]
Corollary \ref{Corollary-lambdaIntegral2version2} is the best possible integrability of $\frac{e^\frac{\delta}{b}}{\lambda^2}$ for constant $\delta$. 
As a heuristic, assume that $\lambda \sim \sqrt{T-t}$ and $e^{\frac{1}{b}} \sim \abs{\log \lambda} \sim \abs{\log(T-t)}$, motivated by the log-log rate {\bf H1.3}.
The integral, $\int^T{\frac{\abs{\log(T-t)}^{\delta}}{T-t}\,dt}$, is only finite for values of $\delta$ sufficiently negative. In our case, the maximum threshold for $\delta$ is given dynamically by (\ref{Eqn-epsIntegral}). 
\end{remark}

Next, we translate hypotheses {\bf H2.1} through {\bf H2.3} into a gain of derivative during particular three-dimensional Sobolev embeddings. Consider a smooth cutoff function with support on $\chi^{-1}\{1\}$,
\begin{equation}\label{DefnEqn-TildeChi}
\widetilde{\chi}(r,z,\theta) = \left\{\begin{aligned}
		1 && \text{ for } \abs{(r,z) - (1,0)} \geq \frac{3}{4}\\
		0 && \text{ for } \abs{(r,z) - (1,0)} \leq \frac{2}{3}.
	\end{aligned}\right.
\end{equation}
\begin{lemma}[Consequences of Bootstrap Hypotheses]\label{Lemma-NKSobolev}
Let $v = \widetilde{\chi} u$, and suppose that
$2 \geq l_1 \geq l_2 \geq l_3 \geq 0$ with $l_1+l_2+l_3 = 3$.
Then,
\begin{equation}\label{Eqn-NKSobolev-simple}\begin{aligned}
\int{\abs{\grad^{2}v}^2\abs{v}^2} \leq C\left(\widetilde{\chi},\alpha^*\right) \norm{v}_{H^{3}}^2,
\end{aligned}\end{equation}
where $C\left(\widetilde{\chi},\alpha^*\right) \rightarrow 0$ as $\alpha^*\rightarrow 0$, and,
\begin{equation}\label{Eqn-NKSobolev-full}
\begin{aligned}
	&\int{\abs{\grad^{3}v}
		\abs{\grad^{l_1}v}\abs{\grad^{l_2}v}\abs{\grad^{l_3}v}}\\
	&+\int{\abs{\grad v}\abs{v}^2
		\abs{\grad^{l_1}v}\abs{\grad^{l_2}v}\abs{\grad^{l_3}v}}\\
	&+\int{\abs{\grad^{2}v}^2\left(\abs{\grad v}^2+\abs{v}^4\right)}
\end{aligned}
\leq C(\widetilde{\chi}) \frac{1}{\lambda^{7}}.
\end{equation}
\end{lemma}
\begin{proof}
To prove (\ref{Eqn-NKSobolev-simple}), apply the three-dimensional Sobolev embeddings $H^{1} \hookrightarrow L^{6}$ and $H^{\frac{1}{2}}\hookrightarrow L^{3}$,
\[
\int{\abs{\grad^{2}v}^2\abs{v}^2} \leq 
	\norm{\grad^{2}v}_{L^{6}}^2\norm{v}_{L^{3}}^2.
	\lesssim \norm{v}_{H^{3}}^2\norm{v}_{H^\frac{1}{2}}^2,
\] 
then recall hypothesis {\bf H2.3}. 
Now, consider the three LH terms of (\ref{Eqn-NKSobolev-full}) in turn, applying H\"older and three-dimensional Sobolev embeddings in each case.
\begin{enumerate}
\item
\begin{equation}\label{Proof-NKSobolev-eqn1}\begin{aligned}
\int{\abs{\grad^{3}v}\abs{\grad^{l_1}v}\abs{\grad^{l_2}v}\abs{\grad^{l_3}v}}
	&\lesssim \norm{v}_{\dot{H}^{3}}\\
	\prod_{j=1,2,3}\norm{\grad^{l_j}v}_{L^\frac{6}{l_j}}
	&\lesssim \norm{v}_{\dot{H}^{3}}
	\prod_{j=1,2,3}\norm{v}_{H^{\frac{3+l_j}{2}+\delta}},
\end{aligned}\end{equation}
where $\frac{1}{2} \gg \delta > 0$ is only necessary if $l_3 = 0$.  Apply hypotheses {\bf H2.1} and {\bf H2.2},
interpolating if $\delta \neq 0$.  The resulting bound is of the order $\frac{1}{\lambda^6}$.

\item 
\begin{equation}\label{Proof-NKSobolev-eqn2}\begin{aligned}
\int{\abs{\grad v}\abs{v}^2
		\abs{\grad^{l_1}v}\abs{\grad^{l_2}v}\abs{\grad^{l_3}v}}
	&\lesssim \norm{\grad v}_{L^3}\norm{v}_{L^{18}}^2
		\prod_{j=1,2,3}\norm{\grad^{l_j}v}_{L^{\frac{27}{5}\frac{1}{l_j}}}\\
	&\lesssim \norm{v}_{H^\frac{3}{2}}\norm{v}_{H^\frac{4}{3}}^2
		\prod_{j=1,2,3}
		\norm{v}_{H^{\frac{3}{2}+\frac{4}{9}l_j+\delta}}
\end{aligned}\end{equation}
where, again, $\frac{1}{2} \gg \delta > 0$ is only necessary if $l_3 = 0$.
Apply hypotheses {\bf H2.1} through {\bf H2.3}, interpolating where necessary.
The resulting bound is of the order $\frac{1}{\lambda^\frac{8}{3}}$.

\item\begin{equation}\label{Proof-NKSobolev-eqn4}\begin{aligned}
\int{\abs{\grad^{2}v}^2\left(\abs{\grad v}^2+\abs{v}^4\right)}
	&\lesssim \norm{\grad^{2}v}_{L^{6}}^2
		\left(\norm{\grad v}_{L^{3}}^2 
			+ \norm{v}_{L^{6}}^4\right)\\
	&\lesssim \norm{v}_{H^{3}}^2
		\left(\norm{v}_{H^{\frac{3}{2}}}^2 
			+ \norm{v}_{H^1}^4\right).
\end{aligned}\end{equation}
Apply hypotheses {\bf H2.1} and {\bf H2.3}. The resulting bound is of the order $\frac{1}{\lambda^6}$. 
\end{enumerate}

Finally, use hypothesis {\bf H1.3} to estimate the neglected factors of $e^\frac{1}{b}$ by a single factor of $\frac{1}{\lambda}$.
\end{proof}

Near the singular ring, and in particular on the support of $\grad\chi$, we do not have the luxury of bootstrap hypotheses.  However, in this region two-dimensional type Sobolev embeddings are applicable. Coupled to the geometric decomposition, we achieve precisely the weakest usable bounds. 
\begin{lemma}[Two-dimensional version of Lemma \ref{Lemma-NKSobolev}]
\label{Lemma-NSobolev}
Let $v=(1-\widetilde{\chi})u$, and suppose that $2 \geq l_1 \geq l_2 \geq l_3 \geq 0$ with $l_1+l_2+l_3 = 3$.  There exists $\sigma_4 > 0$ so that,
\begin{equation}\label{Eqn-NSobolev-simple}
\int{\abs{\grad^{2}v}^2\abs{v}^2} \leq C\left(\widetilde{\chi},\widetilde{Q}_b\right)
	\left(\frac{1}{\lambda^{6}} + 
		e^{-\frac{\sigma_4}{b}}\norm{u}_{H^{3}}^{2}\right),
\end{equation}
and,
\begin{equation}\label{Eqn-NSobolev-full}
\begin{aligned}
	&\int{\abs{\grad^{3}v}\abs{\grad^{l_1}v}\abs{\grad^{l_2}v}\abs{\grad^{l_3}v}}\\
	&+\int{\abs{\grad v}\abs{v}^2
		\abs{\grad^{l_1}v}\abs{\grad^{l_2}v}\abs{\grad^{l_3}v}}\\
	&+\int{\abs{\grad^{2}v}^2\left(\abs{\grad v}^2+\abs{v}^4\right)}\end{aligned}
	 \leq C\left(\widetilde{\chi},\widetilde{Q}_b\right)
		\left( \frac{1}{\lambda^{8}} + 
		e^{-\frac{\sigma_4}{b}}\frac{1}{\lambda^2}\norm{u}_{H^{3}}^2\right).
\end{equation}
Moreover, the value of $\sigma_4>0$ is uniform over all $m>0$ sufficiently small.
\end{lemma}
\begin{proof}
Due to the concentrated support of $\Qb$ - see (\ref{Hypo1-consequenceForMu}),
\begin{equation}\label{Eqn-NSobolev-decomp}
\left(1-\widetilde{\chi}\right)u(r,z,\theta) = 
	\frac{1}{\lambda}
		\widetilde{Q}_b\left(\frac{(r,z)-(r_0,z_0)}{\lambda}\right)e^{-i\gamma}
 + \left(1-\widetilde{\chi}\right)\widetilde{u}(r,z),
\end{equation}
which we denote by $W+w$. Due to Lemma \ref{Lemma-UnprovenProperty}, the various norms of $W$ are explicit.  For example, $\norm{\grad^3 W}_{L^\infty} \leq C(\widetilde{Q}_b) \frac{1}{\lambda^{4}}$, where the constant is uniform over all $b$ sufficiently small. To prove (\ref{Eqn-NSobolev-simple}) and (\ref{Eqn-NSobolev-full}), we substitute $v = W + w$ and consider two cases: all factors are $W$, or, at least one factor is $w$.  The first case is explicit and trivial. In the second case we will extract a factor that is a power of $\norm{w}_{H^1}$. Assuming $m>0$ is sufficiently small, {\bf H1.2} 
will then yield the factor of $e^{-\frac{\sigma_4}{b}}$.
Throughout this proof, we preserve the correct multiplicity of $\frac{1}{\lambda}$ and $\norm{u}_{H^{3}}$ by avoiding the Sobolev embedding into $L^\infty$.

Make the substitution $v = W + w$. To prove (\ref{Eqn-NSobolev-simple}) we need to show the same bound for,
\begin{equation}\label{Proof-NSobolev-eqn1}
\int{\abs{\grad^{2}w}\abs{\grad^{2}v}\abs{v}^2}
+\int{\abs{\grad^{2}v}^2\abs{v}\abs{w}}.
\end{equation}
For the first term, apply H\"older, the two-dimensional embedding $H^\frac{1}{2} \hookrightarrow L^4$,
and interpolate,
\begin{equation}\label{Proof-NSobolev-eqn2}\begin{aligned}
\int{\abs{\grad^{2}w}\abs{\grad^{2}v}\abs{v}^2}
	&\leq \norm{\grad^{2}w}_{L^4} 
		\norm{\grad^{2}v}_{L^4}
		\norm{v}_{L^{4}}^2\\
	&\lesssim \norm{w}_{H^{3-\frac{1}{2}}}
		\norm{v}_{H^{3-\frac{1}{2}}}\norm{v}_{H^{\frac{1}{2}}}^2
	&\lesssim 
		\left(\norm{w}_{H^{3}}^{\frac{3}{4}}
				\norm{w}_{H^1}^{\frac{1}{4}}\right)
		\norm{v}_{H^{3-\frac{1}{2}}}\norm{v}_{H^{\frac{1}{2}}}^2.
\end{aligned}
\end{equation}
Interpolate the norms in $v$ between $\norm{u}_{L^2}$ and $\norm{u}_{H^{3}}$. The factor of $\norm{w}_{H^1}$ provides a factor of $e^{-\frac{\sigma_4}{b}}$ for some $\sigma_4 > 0$. 
For the second term of (\ref{Proof-NSobolev-eqn1}) follow the same strategy, except use the interpolation $\norm{w}_{H^\frac{1}{2}} \lesssim \norm{w}_{H^1}^\frac{1}{2}\norm{w}_{L^2}^{\frac{1}{2}}$. This completes the proof of (\ref{Eqn-NSobolev-simple}).

Now consider the three LH terms of (\ref{Eqn-NSobolev-full}) in turn. In each case make the substitution $v= W+ w$ and assume at least one factor is $w$.
\begin{enumerate}
\item We need to show the same bound for,
\begin{equation}\label{Proof-NSobolev-eqn3}
\int{\abs{\grad^{3}w}\abs{\grad^{l_a}W}\abs{\grad^{l_b}W}\abs{\grad^{l_c}W}}
	+\int{\abs{\grad^{3}v}\abs{\grad^{l_a}v}\abs{\grad^{l_b}v}\abs{\grad^{l_c}w}},
\end{equation}
where $2\geq l_a, l_b, l_c \geq 0$, $l_a+l_b+l_c = 3$, is some permutation of $l_1, l_2, l_3$. Integrate the first term of (\ref{Proof-NSobolev-eqn3}) by parts, use H\"older and interpolate,
\[\begin{aligned}
\norm{\grad^{2}w}_{L^2}&\norm{\grad^{l_a}W\grad^{l_b}W\grad^{l_c}W}_{H^1}\\
&\lesssim
	\norm{w}_{H^{3}}^\frac{1}{2}
	\norm{w}_{H^1}^\frac{1}{2}
	\norm{\grad^{l_a}W\grad^{l_b}W\grad^{l_c}W}_{H^1}.
\end{aligned}\]
The norms of $W$ have explicit scaling-consistent bounds of the order $\left(\frac{1}{\lambda}\right)^{(2+l_a+l_b+l_c)}$. Again, the factor $\norm{w}_{H^1}$ provides a factor of $e^{-\frac{\sigma_4}{b}}$ for some $\sigma_4 > 0$ and the resulting bound is much better than $\frac{1}{\lambda^7}$.

The remaining term of (\ref{Proof-NSobolev-eqn3}) is more difficult. Choose $p_a, p_b, p_c > 0$ such that,
\begin{equation}\label{Proof-NSobolev-eqn4}\begin{aligned}
\sum_{j=a,b,c}\frac{1}{p_j} = \frac{1}{2}
&& \text{ and } 
&& \begin{aligned}
		&\frac{1}{p_j} < \frac{l_j}{2} && \text{ if } l_j \neq 0,\\
		&\frac{1}{p_j} < \delta_5 && \text{ if } l_j = 0,
	\end{aligned}
\end{aligned}\end{equation} 
where $0 <\delta_5 \ll 1$ is an arbitrary universal constant. Apply H\"older and two-dimensional Sobolev embeddings,
\begin{equation}\label{Proof-NSobolev-eqn5}\begin{aligned}
\int{\abs{\grad^{3}v}\abs{\grad^{l_a}v}\abs{\grad^{l_b}v}\abs{\grad^{l_c}w}}
	&\leq
	\norm{v}_{H^{3}}\norm{\grad^{l_a}v}_{L^{p_a}}\norm{\grad^{l_b}v}_{L^{p_b}}
		\norm{\grad^{l_c}w}_{L^{p_c}}\\
	&\lesssim
	\norm{v}_{H^{3}}
	\prod_{j=a,b}\norm{v}_{H^{2\left(\frac{1}{2}-\frac{1}{p_j}\right)+l_j}}\,
	\norm{w}_{H^{2\left(\frac{1}{2}-\frac{1}{p_c}\right)+l_c}}.\\
\end{aligned}\end{equation}
Due to choice (\ref{Proof-NSobolev-eqn4}), the final three norms of (\ref{Proof-NSobolev-eqn5}) may be interpolated strictly between $H^{3}$ and $H^1$, or strictly between $H^1$ and $L^2$, if $l_j = 0$.
We are guaranteed a factor in $\norm{w}_{H^1}$,
\begin{equation}\label{Proof-NSobolev-Interpolate}
(\ref{Proof-NSobolev-eqn5}) \lesssim \left\{
\begin{aligned}
&\norm{u}_{H^{3}}^2\norm{u}_{H^1}^{2-C(l_c)}\norm{w}_{H^1}^{C(l_c)}
	&& \text{ if } l_j \text{ all non-zero},\\
&\norm{u}_{H^{3}}^{2+C(\delta_5)}\norm{u}_{H^1}^{2-3C(\delta_5)-C(l_c)}\norm{w}_{H^1}^{C(l_c)}
	&& \text{ if } l_j \text{ zero for some }j.
\end{aligned}
\right.
\end{equation}
For $m>0$ sufficiently small (relative to $\delta_5$), there is a spare factor of $e^{-\frac{\sigma_4}{b}}$, for some $\sigma_4>0$. This proves the bound for the first LH term of (\ref{Eqn-NSobolev-full}).

\item Apply H\"older and two-dimensional Sobolev embeddings, using the same values $p_j$,
\begin{equation}\begin{aligned}\label{Proof-NSobolev-eqn7}
\int{\abs{\grad v}\abs{v}^2
		\abs{\grad^{l_1}v}\abs{\grad^{l_2}v}\abs{\grad^{l_3}v}}
\lesssim \norm{\grad v}_{L^4}\norm{v}_{L^8}^2
		\prod_{j=1,2,3}\norm{\grad^{l_j}v}_{L^{q_j}}.
\end{aligned}\end{equation}
Recall, at least one factor of $v$ in (\ref{Proof-NSobolev-eqn7}) is infact $w$. Apply Sobolev embeddings and interpolation exactly as we did to equation (\ref{Proof-NSobolev-eqn5}).
This proves the bound for the second LH term of (\ref{Eqn-NSobolev-full}).

\item Apply H\"older and two-dimensional Sobolev,
\begin{equation}\label{Proof-NSobolev-eqn8}\begin{aligned}
\int{\abs{\grad^{2}w}\abs{\grad^{2}v}\abs{\grad v}^2}
&\leq \norm{\grad^{2}w}_{L^4} 
		\norm{\grad^{2}v}_{L^4}
		\norm{\grad v}_{L^4}^2\\
	&\lesssim \norm{w}_{H^{\frac{5}{2}}}
		\norm{v}_{H^{\frac{5}{2}}}\norm{v}_{H^{\frac{3}{2}}}^2,
\end{aligned}\end{equation}
and,
\begin{equation}\label{Proof-NSobolev-eqn9}\begin{aligned}
\int{\abs{\grad^{2}w}\abs{\grad^{2}v}\abs{v}^4}
&\leq \norm{\grad^{2}w}_{L^4} 
		\norm{\grad^{2}v}_{L^4}
		\norm{v}_{L^8}^4\\
	&\lesssim \norm{w}_{H^{\frac{5}{2}}}
		\norm{v}_{H^{\frac{5}{2}}}\norm{v}_{H^{\frac{3}{4}}}^4.
\end{aligned}\end{equation}
The bound for the third LH term of (\ref{Eqn-NSobolev-full}) follows from interpolation.
\end{enumerate}
\end{proof}

\begin{lemma}[$H^{3}$ Energy Identity]\label{Lemma-HnkEnergy}
Denote the third-order energy by,
\begin{equation}\label{DefnEqn-Enk}
E_{3}(u) =
	\int{\abs{\grad^{3}u}^2}
	-\left(2\int{\abs{\grad^{2}u}^2\abs{u}^2}
		+Re\int{(\grad^{2}\overline{u})^2u^2}\right).
\end{equation}
Then,
\begin{equation}\label{Eqn-HnkEnergy}\begin{aligned}
\frac{1}{C}\abs{\frac{d}{dt}E_{3}(u)}  
\leq &
\int{\abs{\grad^{3}u}
		\abs{\grad^{l_1}u}\abs{\grad^{l_2}u}\abs{\grad^{l_3}u}}\\
&+\int{\left(\abs{\grad u}\abs{u}^2\right)
		\abs{\grad^{l_1}u}\abs{\grad^{l_2}u}\abs{\grad^{l_3}u}}\\
&+ \int{\abs{\grad^{2}u}^2\left(\abs{\grad u}^2+\abs{u}^4\right)},
\end{aligned}
\end{equation}
where the right side is implicitly summed over $2 \geq l_1 \geq l_2 \geq l_3 \geq 0$ with $l_1+l_2+l_3 = 3$.
\end{lemma}

\begin{proof}
We refer to the RHS of (\ref{Eqn-HnkEnergy}) as error terms of type I, II, and III respectively.
By direct calculation,
\begin{equation}\label{Proof-HnkEnergy-eqn1}
\begin{aligned}
	\frac{1}{2}\frac{d}{dt}\left(\int{\abs{\grad^{3}u}^2}\right)
	& = -Im
			\int{\grad^{3}\left(\laplacian u 
			+u\abs{u}^2\right)\,\grad^{3}\overline{u}}\\
	& = -2\,Im\int{\grad\left(\grad^{2}u\,\abs{u}^2\right)\,\grad^{3}\overline{u}}\\
	&\mspace{36.0mu}-Im\int{\grad\left(\grad^{2}\overline{u}u^2\right)\,\grad^{3}\overline{u}}\\
	&\mspace{54.0mu}+ \text{terms of the form }
			\int{\grad \left(\grad u \grad u\,u\right)\grad^{3}\overline{u}}.
\end{aligned}\end{equation} 
The final terms of (\ref{Proof-HnkEnergy-eqn1}) are error of type I. Regarding the first RH term of (\ref{Proof-HnkEnergy-eqn1}),
\begin{equation}\label{Proof-HnkEnergy-eqn2}\begin{aligned}
-2\,Im\int{\grad\left(\grad^{2}u\,\abs{u}^2\right)\,\grad^{3}\overline{u}}
	\;=\; & 2\,Im\int{\grad^{2}u\,\abs{u}^2\grad^{2}\laplacian \overline{u}}\\
	=\; & \int{\frac{d}{dt}\left(\abs{\grad^{2}u}^2\right)\,\abs{u}^2}
		-2\,Im\int{\grad^{2}u\,\abs{u}^2\grad^{2}\left(\overline{u}\abs{u}^2\right)}.
\end{aligned}
\end{equation}
Recognize the last term of (\ref{Proof-HnkEnergy-eqn2}) as error of type II and III. Regarding the other term,
\begin{equation}\label{Proof-HnkEnergy-eqn3}\begin{aligned}
\int{\frac{d}{dt}\left(\abs{\grad^{2}u}^2\right)\,\abs{u}^2}
=&\frac{d}{dt}\left(\int{\abs{\grad^{2}u}^2\,\abs{u}^2}\right)
+2\,Im\int{\abs{\grad^{2}u}^2\left(\laplacian u + u\abs{u}^2\right) \overline{u}}.
\end{aligned}\end{equation}
After integration by parts, we recognize the final term of (\ref{Proof-HnkEnergy-eqn3}) as error of type I and III. We have shown that,
\[
-2\,Im\int{\grad\left(\grad^{2}u\,\abs{u}^2\right)\,\grad^{3}\overline{u}}
 = \frac{1}{2}\frac{d}{dt}\left(\int{\abs{\grad^{2}u}^2\,\abs{u}^2}\right),
\]
up to error terms. It is virtually the same calculation to show that,
\[
-Im\int{\grad\left(\grad^{2}\overline{u}u^2\right)\,\grad^{3}\overline{u}}
 = \frac{1}{2}\frac{d}{dt}\left( Re\int{(\grad^{2}\overline{u})^2u^2} \right),
\]
also up to error terms of type I, II, and III. This completes the proof of (\ref{Eqn-HnkEnergy}).
\end{proof}

Now we simply combine the previous three Lemmas.
Equations (\ref{Eqn-NKSobolev-simple}) and (\ref{Eqn-NSobolev-simple}) prove that $E_3 \approx \norm{u}_{H^3}$. Equations (\ref{Eqn-NKSobolev-full}) and (\ref{Eqn-NSobolev-full}) control $\frac{d}{dt}E_3$. Integrate the bound on $\frac{d}{dt}E_3$ using Lemma \ref{Lemma-lambdaIntegral}, with $\sigma_2 < \min(\sigma_4, 2m)$. Choose $m'>0$ to be any value, $m-\frac{\sigma_4}{2}<m'<m$.  Assuming $\alpha^*$ is sufficiently small (depending on the choice of $m'$), we have proven,
\begin{lemma}[Controlled Growth of $H^{3}$]\label{Lemma-ControlledHnk}
For all $t\in[0,T_{hyp})$,
\begin{equation}\label{Eqn-Hnk}
\norm{u(t)}_{H^{3}(\real^{3})} 
	< \frac{e^{\frac{m'}{b(t)}}}{\lambda^{3}(t)}.
\end{equation}
That is, statement {{\bf I2.1}}. 
\end{lemma}

\begin{remark}[Higher Dimensions]\label{Remark-HigherDim-NthEnergy}
For the higher dimensional case, define,
\[
E_N(u) = \int{\abs{\grad^Nu}^2} - \left(2\int{\abs{\grad^{N-1}u}^2\abs{u}^2} + Re\int{\left(\grad^{N-1}u\right)^2\left(\overline{u}\right)^2}\right).
\]
Then $\abs{\frac{d}{dt}E_N(u)}$ may be bounded in the same way by generalizing error of type II to include,
\[\begin{aligned}
\int{\left(\abs{\grad^{k_1}u}\abs{\grad^{k_2}u}\abs{\grad^{k_3}u}\right)
		\abs{\grad^{l_1}u}\abs{\grad^{l_2}u}\abs{\grad^{l_3}}}
&&\text{ with }&&k_1+k_2+k_3 = N-2.
\end{aligned}\]
For the calculation with quintic nonlinearity, see \cite[Lemma 1]{RaphaelSzeftel-StandingRingNDimQuintic-08}.
\end{remark}

\subsection{Behaviour away from both Infinity and the Singularity}
	\label{Subsection-Annular}
This section we concentrate on the interface between the singular ring and the truly three-dimensional region that contains the origin.  On this interface, away from $r=0$, the dynamics remains essentially two-dimensional and $L^2$-critical
\begin{lemma}[Two-dimensional endpoint Sobolev control away from the singularity]
\label{Lemma-NdimEndpointSobolev}
For $\sigma_5 > 0$, and a smooth cutoff function $\varphi$ compactly supported away from both the singular ring and the origin,
\begin{equation}\label{Eqn-NdimEndpointSobolev}
\norm{\varphi u(t)}_{L^\infty(\real^3)} 
	\leq C(\sigma_5, \varphi) 
		e^{+\frac{\sigma_5}{b(t)}}
		\left(\norm{\widetilde{u}(t)}_{H^1(\real^3)} 
		+ \frac{\Gamma^\frac{1}{2}_{b(t)}}{\lambda(t)}\right).
\end{equation}
This is a two-dimensional type of estimate due to the support of $\varphi$.
\end{lemma}
The key feature of Lemma \ref{Lemma-NdimEndpointSobolev} is that we may avoid the Sobolev embedding $H^{1+\epsilon}(\real^2) \hookrightarrow L^\infty(\real^2)$. At the order of the blowup parameter $\lambda$, equation (\ref{Eqn-NdimEndpointSobolev}) is consistent with scaling. In the case of radial symmetry, such as \cite{R06,RaphaelSzeftel-StandingRingNDimQuintic-08}, Strauss' radial embedding is used instead.

\begin{proof}[Proof of Lemma \ref{Lemma-NdimEndpointSobolev}]
We adapt an argument of Brezis \& Gallou\"et.
Our estimate is for a fixed time $t\in[0,T_{hyp})$. 
Choose $R = \norm{\widetilde{u}(t)}_{H^1} + \frac{\Gamma^\frac{1}{2}_{b(t)}}{\lambda(t)} \gg 0$. Denote $v = \varphi u$ and partition phase space,
\[
\abs{v} \leq \norm{\hat{v}}_{L^1} = 
	\int_{\abs{\xi} \leq R}{\abs{\hat{v}(\xi)}\,d\xi}
	+\int_{\abs{\xi} > R}{\abs{\hat{v}(\xi)}\,d\xi}
\]
Rewrite the low frequencies and apply H\"older,
\begin{equation}\label{Eqn-NdimEndpoint-eqn1}\begin{aligned}
\int_{\abs{\xi} \leq R}{\abs{\hat{v}}\,d\xi} 
	&= \int_{\abs{\xi} \leq R}{
		\left({\langle\xi\rangle}^{\frac{1}{2}}
			\abs{\hat{v}}^{\frac{1}{2}}\right)
		\left(\abs{\hat{v}}^{\frac{1}{2}}\right)
		\left(\frac{1}{{\langle\xi\rangle}^{\frac{1}{2}}}\right)\,d\xi
	}\\
	&\leq \norm{v}_{H^{1}}^{\frac{1}{2}}
			\norm{v}_{L^2}^{\frac{1}{2}}
			\left(\int_{\abs{\xi} \leq R}{\frac{1}{{\langle\xi\rangle}}\,d\xi}
			\right)^\frac{1}{2},
\end{aligned}\end{equation}
where ${\langle\xi\rangle}$ denotes $\sqrt{1 + \abs{\xi}^2}$. Note the final integral of (\ref{Eqn-NdimEndpoint-eqn1}) is,
$
\int_{\abs{\xi} \leq R}{\frac{1}{{\langle\xi\rangle}}\,d\xi}
	\leq \int_0^R{\frac{1}{\rho}\rho\,d\rho}
	= R.
$ 
Apply a similar argument to high frequencies, with parameter $\nu(\sigma_5,m)>1$ to be determined,
\begin{equation}\label{Eqn-NdimEndpoint-eqn3}\begin{aligned}
\int_{\abs{\xi} > R}{\abs{\hat{v}}\,d\xi}
	&= \int_{\abs{\xi} > R}{\left({\langle\xi\rangle}^\nu\abs{\hat{v}}\right)
			\frac{1}{{\langle\xi\rangle}^\nu}\,d\xi}\\
	&\lesssim \norm{v}_{H^\nu}\,
			\left(\int_R^{+\infty}{\frac{1}{{\langle \rho\rangle}^{2\nu}}\rho\,d\rho}\right)^\frac{1}{2}\\
	&\leq \frac{1}{2(\nu-1)}\norm{v}_{H^\nu}\,\frac{1}{R^{\nu-1}}\\
	&\lesssim \frac{1}{2(\nu-1)}
		\left(\norm{v}_{H^1}^{2-\nu}R^{\nu - 1}\right)
		\left(\frac{\norm{v}_{H^2}^{\nu-1}}{R^{2(\nu-1)}}\right).
\end{aligned}\end{equation}
Due to hypothesis {\bf H2.1} 
and $\Gamma_b$-estimate (\ref{Eqn-GammaBEstimate}), the final term of (\ref{Eqn-NdimEndpoint-eqn3}) is bounded by $e^{+\frac{\sigma_5}{b(t)}}$ for any choice of $\nu > 1$ sufficiently small.
\end{proof}

\begin{definition}[Cutoffs to cover $Supp\left(\grad\chi\right)$]\label{Defn-AnnularCutoffs}
Fix seven smooth cylindrically symmetric cutoff functions, $\psi^{(0)}$, $\psi^{(\frac{1}{2})}$, $\psi^{(1)}$, $\varphi^{(\frac{5}{2})}$, $\varphi^{(2)}$, $\varphi^{(\frac{3}{2})}$, $\varphi^{(1)}$, such that,
\begin{enumerate}
\item {\it They cover the support of $\grad\chi$.}

Each function is $1$ on, $\left\{ \frac{1}{3} < \abs{(r,z)-(1,0)} < \frac{2}{3}\right\}$.

\item {\it Tails do not overlap.}

The support of each cutoff is contained where the previous cutoff is $1$.

\item {\it Supported away from both the singularity and the origin.}

The largest support, that of $\psi^{(0)}$, is contained in, $\left\{ \frac{1}{7} < \abs{(r,z) - (1,0)} < \frac{6}{7}\right\}$.
\end{enumerate}
\end{definition}

\begin{lemma}[Annular $H^\frac{1}{2}$ control - the crucial first step]
\label{Lemma-Annular12} 
For all $t\in[0,T_{hyp})$,
\begin{equation}\label{Eqn-Annular12}
\norm{\psi^{(\frac{1}{2})}u}_{H^\frac{1}{2}} \lesssim \frac{1}{\lambda^{C(\alpha^*)}(t)},
\end{equation}
where $C(\alpha^*) \rightarrow 0$ as $\alpha^* \rightarrow 0$.
\end{lemma}
This is the first proof that any behaviour better than scaling extends beyond the support of hypotheses {\bf H2.2} and {\bf H2.3}.
\begin{remark}[Analogue in \cite{R06,RaphaelSzeftel-StandingRingNDimQuintic-08}]
In radial cases, one proves Lemma \ref{Lemma-Annular12} for $H^\nu$, $\nu < \frac{1}{2}$. The subsequent $H^\frac{1}{2}$ bound, for example \cite[Lemma 10]{R06}, should be seen as comparable to forthcoming Lemma \ref{Lemma-AnnularN2}.
\end{remark}

\begin{proof}[Proof of Lemma \ref{Lemma-Annular12}]
By direct calculation,
\begin{equation}\label{Proof-Annular12-Gronwall}
\frac{1}{2}\frac{d}{dt}\norm{\psi^{(\frac{1}{2})}u}_{\dot{H}^\frac{1}{2}}^2 =
	Im\left(\int{D^\frac{1}{2}\left(u\laplacian\psi^{(\frac{1}{2})} 
		+ 2\grad\psi^{(\frac{1}{2})}\cdot\grad u 
		- \psi^{(\frac{1}{2})}u\,\abs{u}^2\right)
		\,D^\frac{1}{2}\left(\psi^{(\frac{1}{2})}\overline{u}\right)}\right).
\end{equation}
Estimate the first and second RH terms of (\ref{Proof-Annular12-Gronwall}), 
\begin{align}
\notag \norm{D^\frac{1}{2}\left(u\laplacian\psi^{(\frac{1}{2})}\right)}_{L^2}
		\norm{\psi^{(\frac{1}{2})}u}_{H^\frac{1}{2}}
		&\leq C\left(\psi^{(\frac{1}{2})}\right)\norm{\psi^{(0)}u}_{H^\frac{1}{2}}
		\norm{\psi^{(\frac{1}{2})}u}_{H^\frac{1}{2}},\\
\notag \norm{\grad\psi^{(\frac{1}{2})}\cdot\grad u}_{L^2}
		\norm{\psi^{(\frac{1}{2})}u}_{H^1}
		&\leq C\left(\psi^{(\frac{1}{2})}\right)\norm{\psi^{(0)}u}_{H^{1}}^2.
\end{align}
The nonlinear term of (\ref{Proof-Annular12-Gronwall}) does not enjoy any real-valued cancellations, as the operator $D$ does not have an exact Leibniz property.
Apply standard commutation estimates,
\begin{equation}\label{Proof-Annular12-eqnNeedBrezis}\begin{aligned}
\norm{D^\frac{1}{2}\left(\psi^{(\frac{1}{2})}u\,\abs{u}^2\right)}_{L^2} 
\lesssim 
	&
	\norm{\psi^{(\frac{1}{2})}u}_{H^\frac{1}{2}}\norm{\psi^{(0)}u}_{L^\infty(\real^2)}^2
	+\norm{\psi^{(\frac{1}{2})}u}_{L^4}
	\norm{\psi^{(0)}{u}}_{W^{\frac{1}{2},4}}\norm{\psi^{(0)}u}_{L^\infty(\real^2)}\\
\lesssim &
	\norm{\psi^{(\frac{1}{2})}u}_{H^{\frac{1}{2}}}\left(
		\norm{\psi^{(0)}u}_{L^\infty(\real^2)}^2 
		+ \norm{\psi^{(0)}u}_{H^1}\norm{\psi^{(0)}u}_{L^\infty(\real^2)}
	\right).
\end{aligned}\end{equation}
From support away from the singularity, $\psi^{(0)}u = \psi^{(0)}\widetilde{u}$, and we may apply the endpoint estimate of Lemma \ref{Lemma-NdimEndpointSobolev}.
Denote $\norm{\psi^{(\frac{1}{2})}u(t)}_{\dot{H}^\frac{1}{2}}$ by $f$. We have the simple ODE,
\begin{equation}\label{Proof-Annular12-simpleODE}\begin{aligned}
\frac{1}{2}\frac{d}{dt}\left(f^2\right) \leq 
	&C\left(\psi^{(\frac{1}{2})}\right)\left(
		f\,\norm{\widetilde{u}(t)}_{H^1}^\frac{1}{2}
		\norm{\widetilde{u}(t)}_{L^2}^\frac{1}{2}
		+ \norm{\widetilde{u}(t)}_{H^1}^2
	\right)\\
	&+C\left(\sigma_5,\psi^{(0)}\right)f^2\,
		e^{+\frac{\sigma_5}{b(t)}}
			\left(\norm{\widetilde{u}(t)}_{H^1}^2 
			+ \frac{\Gamma_{b(t)}}{\lambda^2(t)}\right).
\end{aligned}\end{equation}
The final term is dominant. After integration by Corollary \ref{Corollary-lambdaIntegral2version2},
\[
\norm{\psi^{(\frac{1}{2})}u(t)}_{{H}^\frac{1}{2}}
\lesssim e^{\left\lbrack C\left(\alpha^*\right)\,C\left(\sigma_5,\psi^{(0)}\right)e^{+\frac{\sigma_5}{b(t)}}\right\rbrack}.
\]
To complete the proof, choose $\sigma_5 = \frac{\pi}{10}$ and recall the log-log rate {\bf H1.3}.
\end{proof}

\begin{remark}[Justification for Lemma \ref{Lemma-NdimEndpointSobolev}]
\label{Remark-NormalBrezisFails}
The open nature of hypothesis {\bf H1.3} is an essential feature of any modulation argument. It is for this reason that we must be free to choose $\sigma_5$.  The standard Brezis-Gallou\"et estimate, 
$
\norm{v}_{L^\infty(\real^2)} \lesssim \norm{v}_{H^1}\sqrt{\log\left(\norm{v}_{H^2}\right)},
$ would not suffice to prove Lemma \ref{Lemma-Annular12}.
\end{remark}

We now reformulate the calculation of equation (\ref{Proof-Annular12-Gronwall}) for repeated application.
\begin{lemma}[Standard Gronwall Argument]\label{Lemma-GronwallArg}
Let $\psi^A$ be supported where $\psi^B\equiv 1$, let $I$ be any subinterval of $[0,T_{hyp})$, and let $\nu\geq 0$. Then,
\begin{equation}\label{Eqn-GronwallTool}
\norm{\psi^Au}_{L^\infty_IH^\nu} \leq C(\psi^A)\left(
	\norm{\psi^Bu_0}_{H^\nu} + \abs{I} +
	\norm{\psi^Bu}_{L^2_IH^{\nu+\frac{1}{2}}} +
	\norm{\psi^Au\,\abs{u}^2}_{L^1_IH^{\nu}}\right). 
\end{equation}
\end{lemma}

\begin{lemma}[Annular $H^1$ control - propagation of Lemma \ref{Lemma-Annular12}]
\label{Lemma-AnnularN2}
There exists $\sigma_6 > 0$, universal for all $m>0$ sufficiently small, such that for all $t\in[0,T_{hyp})$,
\begin{equation}\label{Eqn-AnnularN2}
\norm{\psi^{(1)}u(t)}_{H^1} < C(\alpha^*)\frac{e^{-\frac{\sigma_6}{b(t)}}}{\lambda^\frac{1}{2}(t)},
\end{equation}
where $C(\alpha^*) \rightarrow 0$ as $\alpha^* \rightarrow 0$.
\end{lemma}
\begin{proof}
Apply the standard Gronwall argument, equation (\ref{Eqn-GronwallTool}), for $\nu=1$, $I = [0,t<T_{hyp}]$, $\psi^A = \psi^{(1)}$, and $\psi^B = \psi^{(\frac{1}{2})}$. Note that $\psi^{(1)}u = \psi^{(1)}\tilde{u}$. Through interpolation and hypotheses {\bf H1.2} and {\bf H2.1},
\begin{equation}\label{Proof-AnnularN2-eqn1}\begin{aligned}
\norm{\psi^{(1)}u}_{L^2_IH^{1+\frac{1}{2}}}
\lesssim& \left(\int{
	\norm{\tilde{u}}_{H^1}^{2-\frac{1}{2}}
	\norm{\tilde{u}}_{H^{3}}^{\frac{1}{2}}
	}\right)^\frac{1}{2}
&\lesssim& \left(\int{
	e^{-\left(\frac{1}{4}-\frac{m}{2}\right)\frac{1}{b}}\frac{1}{\lambda^3}
	}\right)^\frac{1}{2}.
\end{aligned}\end{equation}
Assuming $m>0$ is sufficiently small, apply integrability Lemma \ref{Lemma-lambdaIntegral} for $\sigma_2 > 0$, also sufficiently small. 
Regarding the final term of (\ref{Eqn-GronwallTool}), apply H\"older, two-dimensional Sobolev embedding, and interpolate,
\begin{equation}\label{Proof-AnnularN2-eqn2}\begin{aligned}
\norm{\psi^{(1)}u\,\abs{u}^2}_{H^1} \lesssim &
\norm{\grad\left(\psi^{(\frac{1}{2})}u\right)\,\left(\psi^{(\frac{1}{2})}u\right)^2}_{L^2}\\
\lesssim &
\norm{\grad\left(\psi^{(\frac{1}{2})}u\right)}_{L^4}\norm{\psi^{(\frac{1}{2})}u}_{L^8}^2\\
\lesssim &
\norm{\psi^{(\frac{1}{2})}u}_{H^\frac{3}{2}}\norm{\psi^{(\frac{1}{2})}u}_{H^\frac{3}{4}}^2\\
\lesssim &
\norm{u}_{H^{3}}^{\frac{3}{5}}
\norm{\psi^{(\frac{1}{2})}u}_{H^\frac{1}{2}}^{3-\frac{3}{5}}
	& \lesssim\frac{1}{\lambda^{\frac{9}{5}+C(\alpha^*)}(t)},
\end{aligned}\end{equation}
where the final inequality is due to hypothesis {\bf H2.1} and Lemma \ref{Lemma-Annular12}. Apply Lemma \ref{Lemma-lambdaIntegralCrude}.
\end{proof}

\begin{remark}[Scheme for the remainder of Chapter \ref{Section-BootAtInfty}]
The proof of Lemma \ref{Lemma-AnnularN2} may be repeated, with a shrunken cutoff and $H^\frac{3}{2}$ in place of $H^1$.  However, due to the new version of equation (\ref{Proof-AnnularN2-eqn1}), iteration to higher norms will not yield more than the same $\frac{1}{2}$-derivative improvement over scaling.

Instead, we switch direction. Starting with {\bf I2.1}, at each stage the previous iterate will give progressively better control on the equivalent of (\ref{Proof-AnnularN2-eqn1}).  Lemma \ref{Lemma-AnnularN2} will be used to help control the equivalent of equation (\ref{Proof-AnnularN2-eqn2}).
\end{remark}

\begin{lemma}[Moser-type Product Estimate]\label{Lemma-MoserType}
Let $v \in H^{\nu+\frac{1}{2}}(\real^d)$ for some $\nu \geq \frac{d-1}{2}$, not necessarily an integer. Then,
\begin{equation}\label{Eqn-MoserEst}
\norm{v^3}_{H^\nu} \lesssim 
	\norm{v}_{H^{\nu+\frac{1}{2}}}
	\norm{v}_{H^{\frac{d}{2}}}^2.
\end{equation}
\end{lemma}

\begin{lemma}[{\bf I2.2} and {\bf I2.3} on the support of $\grad \chi$]
\label{Lemma-AnnularHlower}
For all $t \in [0,T_{hyp})$:
\begin{equation}\label{Eqn-AnnularHlower}
\norm{\varphi^{(3-\kappa)}u}_{H^{3-\kappa}} < C(\alpha^*) \frac{e^{(1+\kappa)\frac{m'}{b(t)}}}{\lambda^{3-2\kappa}},
\end{equation}
for each half integer $\frac{1}{2} \leq \kappa < \frac{3}{2}$,
\begin{equation}\label{Eqn-AnnularHlowerN2}
\norm{\varphi^{(\frac{3}{2})} u(t)}_{H^\frac{3}{2}} < C(\alpha^*) e^{+\frac{2m' + \pi}{b(t)}},
\end{equation}
and,
\begin{equation}\label{Eqn-AnnularHCrit}
\norm{\varphi^{(1)}u}_{H^{1}} < C(\alpha^*) \left(\alpha^*\right)^{\frac{1}{5}},
\end{equation}
where in each case $C(\alpha^*) \rightarrow 0$ as $\alpha^* \rightarrow 0$.
\end{lemma}
\begin{proof}
We prove (\ref{Eqn-AnnularHlower}) by induction in $\kappa$. The base case $\kappa = 0$ is Lemma \ref{Lemma-ControlledHnk}. 
Hypothesize that (\ref{Eqn-AnnularHlower}) holds for $\kappa - \frac{1}{2}$, some $\kappa \geq \frac{1}{2}$. Denote $\nu = 3-\kappa$, and apply the standard Gronwall argument for $I=[0,t<T_{hyp}]$, $\psi^A =\phi^{(\nu)}$ and $\psi^B = \phi^{(\nu+\frac{1}{2})}$,
\begin{equation}\label{Proof-AnnularLower-eqn1}
\norm{\varphi^{(\nu)}u}_{H^\nu} \lesssim
	\norm{\chi_0u_0}_{H^\nu} 
	+ \norm{\varphi^{(\nu+\frac{1}{2})}u}_{L^2_tH^{\nu+\frac{1}{2}}}
	+ \norm{\varphi^{(\nu)}u\,\abs{u}^2}_{L^1_tH^{\nu}}.
\end{equation}
Apply our induction hypothesis to the second RH term of (\ref{Proof-AnnularLower-eqn1}),
\begin{equation}\label{Proof-AnnularLower-eqn1.2}
\norm{\varphi^{(\nu+\frac{1}{2})}u}_{L^2_tH^{\nu+\frac{1}{2}}}
	\lesssim \left(\int_I{\left(\frac{e^{\left(1+(\kappa-\frac{1}{2})\right)\frac{m'}{b(\tau)}}}{\lambda^{3-2(\kappa-\frac{1}{2})}(\tau)}\right)^2d\,\tau}\right)^\frac{1}{2}
	\lesssim
	\left(\frac{e^{(1+\kappa)\frac{m'}{b(t)}}}{\lambda^{3-2\kappa}(t)}\right)
	\left(e^{\frac{\sigma_2-m'}{b(t)}}\right)^\frac{1}{2}.
\end{equation}
where, since $\kappa < \frac{3}{2}$, we applied Lemma \ref{Lemma-lambdaIntegral} for some $\sigma_2 < m'$.
Examine the final term of (\ref{Proof-AnnularLower-eqn1}). 
Note that, $\varphi^{(\nu)}u = \varphi^{(\nu)}\left( \varphi^{(\nu+\frac{1}{2})}u \right)$.
Apply the Moser-type estimate of Lemma \ref{Lemma-MoserType} 
and inject both the $H^1$ control of Lemma \ref{Lemma-AnnularN2} and the induction hypothesis,
\begin{equation}\label{Proof-AnnularLower-useMoser}\begin{aligned}
\norm{\varphi^{(\nu)}u\,\abs{u}^2}_{H^\nu} \lesssim &
\norm{\varphi^{(\nu+\frac{1}{2})}u}_{H^{\nu+\frac{1}{2}}}
	\norm{\varphi^{(\nu+\frac{1}{2})}u}_{H^1}^2\\
\lesssim &
	\frac{e^{+\frac{(1+(\kappa-\frac{1}{2}))m'}{b}}}{\lambda^{3-2(\kappa-\frac{1}{2})}}
	\frac{e^{-\frac{2\sigma_6}{b}}}{\lambda}
& =  \frac{e^{\frac{(\text{neg})}{b}}}{\lambda^2}\frac{1}{\lambda^{3-2\kappa}},
\end{aligned}\end{equation}
where we made the assumption $m>0$ is sufficiently small relative to $\sigma_6$.
Finally, apply Lemma \ref{Lemma-lambdaIntegral} for some $\sigma_2$ less than the negative exponent. This completes the proof of (\ref{Eqn-AnnularHlower}).

To prove (\ref{Eqn-AnnularHlowerN2}) let $\kappa = \frac{3}{2}$. We proceed exactly as above, using (\ref{Eqn-AnnularHlower}) in place of the induction hypothesis, and applying Corollary \ref{Corollary-lambdaIntegral2version2} in place of Lemma \ref{Lemma-lambdaIntegral}.

To prove (\ref{Eqn-AnnularHCrit}), let $\kappa = 2$. We proceed exactly as above using (\ref{Eqn-AnnularHlowerN2}) in place of the induction hypothesis, and applying Lemma \ref{Lemma-lambdaIntegralCrude} in place of Lemma \ref{Lemma-lambdaIntegral}.
\end{proof}

\subsection{Improved Behaviour at Infinity}
	\label{Subsection-Infty}

With Lemma \ref{Lemma-AnnularHlower} covering the support of $\grad\chi$, we prove the corresponding result for $\chi$ by similar methods. Note the argument is now in three-dimensions.
\begin{proof}[Proof of {\bf I2.2} and {\bf I2.3}]
We revisit the proof of the standard Gronwall argument. 
Let $I = [0,t<T_{hyp}]$, $\nu \geq 0$, and denote $v = \chi u$. With equation (\ref{Eqn-NLS}),
\begin{equation}\label{Proof-OriginHLower-eqn1}
iv_t+\laplacian v +v\abs{v}^2= u\laplacian\chi + 2\grad\chi\cdot\grad u + \left(\chi^2 - 1\right)\chi u\abs{u}^2.
\end{equation}
Note that the terms on the RHS of (\ref{Proof-OriginHLower-eqn1}) are localized to the support of $\grad\chi$, a region of two-dimensional character where $\varphi^{(1)} \equiv 1$.
By direct calculation,
\begin{equation}\label{Proof-OriginHLower-miniGron}\begin{aligned}
\frac{1}{2}\norm{\chi u}_{L^\infty_IH^\nu} \leq 
	&\norm{\chi u_0}_{H^\nu} + \norm{\chi u\abs{\chi u}^2}_{L^1_IH^\nu}\\
&+
C(\chi)\left(
	\norm{\varphi^{(1)}u_0}_{H^\nu} + \abs{I} 
	+\norm{\varphi^{(1)}u}_{L^2_IH^{\nu+\frac{1}{2}}} 
	+\norm{\varphi^{(1)}u\abs{\varphi^{(1)}u}^2}_{L^1_IH^\nu}
\right).
\end{aligned}\end{equation}
Consider $\nu = 3-\kappa$ for some $\frac{1}{2}\leq\kappa\leq 2$.
Due to Definition \ref{Defn-AnnularCutoffs}, all the conclusions of Lemma \ref{Lemma-AnnularHlower} apply to $\varphi^{(1)}u$, which we use in place of an induction hypothesis to control the second line of (\ref{Proof-OriginHLower-miniGron}) exactly as we did equation (\ref{Proof-AnnularLower-eqn1}). These terms will give the largest contribution.

Finally, examine the term nonlinear in $\chi u$.
Apply the Moser-type estimate of Lemma \ref{Lemma-MoserType}, 
interpolate, and inject hypotheses {\bf H2.2},
\begin{equation}\label{Proof-OriginHLower-useMoser}\begin{aligned}
\norm{\chi u\,\abs{\chi u}^2}_{L^1_IH^\nu} 
	\lesssim &
		\norm{ \norm{\chi u}_{H^{\nu+\frac{1}{2}}}
			\norm{\chi u}_{H^\frac{3}{2}}^2 }_{L^1_I}\\
	\lesssim & \left\{\begin{aligned}
			&\int_I{
				\frac{e^{\left(1+(\kappa-\frac{1}{2})\right)\frac{m}{b(\tau)}}}
				{\lambda^{3-2(\kappa-\frac{1}{2})}(\tau)}
				e^{2\frac{2m+\pi}{b(\tau)}}\,d\,\tau}
				&&\text{ for } \kappa < 2,\\
			&\int_I{
				e^{3\frac{2m+\pi}{b(\tau)}}\,d\,\tau}
				&&\text{ for } \kappa = 2.
			\end{aligned}\right.
\end{aligned}\end{equation}
Apply Lemma \ref{Lemma-lambdaIntegralCrude} for $\kappa \geq \frac{3}{2}$, Corollary \ref{Corollary-lambdaIntegral2version2} for $\kappa=1$, and Lemma \ref{Lemma-lambdaIntegral} for $\kappa = \frac{1}{2}$. Note that the result of equation (\ref{Proof-OriginHLower-useMoser}) is an entire order better in $\frac{1}{\lambda}$ than necessary.  
\end{proof}

This completes the proof of Proposition \ref{Prop-Improv}.

\section{Proof of Theorem \ref{Thm-MainResult}}
\label{Section-FinalProof}

\begin{proof}[Proof of norm growth (\ref{Thm-MainResult-LogLog}), (\ref{Thm-MainResult-LogLogHigher})]
From Proposition \ref{Prop-Improv} we have that $T_{hyp} = T_{max}$, and from (\ref{Eqn-t1Small}) we have blowup in finite time. By the failure of local wellposedness we have that $\lambda(t) \to 0$ as $t \rightarrow T_{max}$.  Recall the approximate dynamics of $\lambda$, equation (\ref{Eqn-lambda-prelimDynamics}), which with the control on $b$ implies in particular that $\abs{\frac{\lambda_s}{\lambda}} < 1$ on $[s_0,s_{max})$, which easily integrates to,
\begin{equation}\label{Eqn-sIsInfty}\begin{aligned}
\abs{\log \lambda(s)} \lesssim 1+s
&& \Longrightarrow 
&& s_{max} = +\infty.
\end{aligned}\end{equation}
By direct calculation and a change of variable,
\[\begin{aligned}
-\partial_t\left(\lambda^2\log\abs{\log\lambda}\right) =
	&-\frac{\lambda_s}{\lambda}\log\abs{\log\lambda}
		\left(2 + \frac{1}{\abs{\log\lambda}\log\abs{\log\lambda}}\right).
\end{aligned}\]
From the approximate dynamics (\ref{Eqn-lambda-prelimDynamics}),
$
\frac{b}{2} \leq -\frac{\lambda_s}{\lambda} \leq 2b,
$
and so with the log-log rate {\bf H1.3} we have proven that, for some universal constant $C>0$ and all $t \in [0,T_{max})$,
\begin{equation}\label{Proof-MainResult-eqn1}
\frac{1}{C} \leq -\partial_t\left(\lambda^2\log\abs{\log\lambda}\right) \leq C.
\end{equation}
For all $t\in[0,T_{max})$, integrate equation (\ref{Proof-MainResult-eqn1}). Since $\lambda$ is very small we may estimate,
\begin{equation}\label{Proof-MainResult-BoundForLambda}\begin{aligned}
\frac{1}{C}\left(\frac{T_{max}-t}{\log\abs{\log(T_{max}-t)}}\right)^\frac{1}{2}
\leq \lambda(t)
\leq C \left(\frac{T_{max}-t}{\log\abs{\log(T_{max}-t)}}\right)^\frac{1}{2}.
\end{aligned}\end{equation}
We do not prove the exact value of the constant in equation (\ref{Thm-MainResult-LogLog}) - see \cite[Proposition 6]{MR-SharpLowerL2Critical-06}. 
Finally, we conclude that equation (\ref{Thm-MainResult-LogLogHigher}) follows from the log-log relationship {\bf H1.3}, higher-order norm control {\bf H2.1}, and from $m > 0$ small.
As an aside, recall that $\frac{ds}{dt} = \frac{1}{\lambda^2}$, so that with (\ref{Proof-MainResult-BoundForLambda}) one would conclude,
\begin{equation}\label{Proof-MainResult-BoundForS}
\frac{1}{C}\abs{\log(T_{max}-t)} \leq s(t) \leq C\abs{\log(T_{max}-t)}.
\end{equation}
Then from the explicit lower and upper bounds for $b$, equations (\ref{Eqn-bLowerBound}) and (\ref{Eqn-bUpperBound}),
\begin{equation}\label{Proof-MainResult-BoundForB}
\frac{1}{C \log\abs{\log(T_{max}-t)}}
	\leq b(t) 
	\leq \frac{C}{\log\abs{\log(T_{max}-t)}}.
\end{equation}
\end{proof}
\begin{proof}[Proof of stable locus of concentration, (\ref{Thm-MainResult-RZ})]
The preliminary estimate (\ref{Eqn-prelimGamma+REst}) implies in particular that $\abs{\frac{\partial_s(r,z)}{\lambda}} < 1$ on $[s_0,s_1)$. Then by change of variable, equation (\ref{Proof-MainResult-BoundForLambda}) and the bound on $T_{max}$, equation (\ref{Eqn-t1Small}),
\begin{equation}\label{Proof-MainResult-RZ}
\int_0^{T_{max}}{\abs{\partial_t(r,z)}\,dt} < \int_0^{T_{max}}{\frac{1}{\lambda(t)}\,dt} < \delta(\alpha^*).
\end{equation}
Equation (\ref{Thm-MainResult-RZ}) follows from choice of initial data {\bf C1.1}.
\end{proof}
\begin{proof}[Proof of regularity away from singular ring, (\ref{Thm-MainResult-SingOnRing})]
Given $R>0$, define $\chi_R$ to be a suitable modification of $\chi$ (\ref{DefnEqn-Chi}), equal to one for $\abs{(r,z)-(r_{max},z_{max})} > R$. 
Choose some $t(R)\in[0,T_{max})$ such that,
\begin{equation}\label{Proof-MainResult-BootstrapInitial}\begin{aligned}
A(t)\lambda(t) + \abs{(r(t),z(t))-(r_{max},z_{max})} \ll R
&& \text{ for all } t\in[t(R),T_{max}),
\end{aligned} \end{equation}
and hence $\chi_R u = \chi_R\widetilde{u}$ for all $t\in[t(R),T_{max})$.
Hypothesize $t_3\in(t(R),T_{max}]$ to be the largest value such that, 
\begin{equation}\label{Proof-MainResult-Hcrit}\begin{aligned}
\norm{\chi_R u(t)}_{H^1} < 2\norm{\chi_R u(t(R))}_{H^1}
&& \text{ for all } t\in[t(R),t_3).
\end{aligned}\end{equation}
This choice of $t_3>t(R)$ is possible since $u(t)$ is strongly continuous in $H^1$ at time $t(R) < T_{max}$.
With interpolation, (\ref{Proof-MainResult-Hcrit})
replaces bootstrap hypotheses {\bf H2.2} and {\bf H2.3}. Repeating the arguments of Chapter \ref{Section-BootAtInfty} proves $t_3 = T_{max}$ and,
\begin{equation}\label{Proof-MainResult-SingOnRing}\begin{aligned}
\norm{\widetilde{u}(t)}_{H^{1}\left(\abs{(r,z)-(r_{max},z_{max})}>R\right)} < C(R)
&& \text{ for all } t \in[0,T_{max}).
\end{aligned}\end{equation}
\end{proof}
\begin{proof}[Proof of mass concentration, (\ref{Thm-MainResult-L2})]
Let $R> 0$.  To begin we will prove there exists a residual profile in $L^2$ away from the singular ring,
\begin{equation}\label{Proof-MainResult-eqn3}\begin{aligned}
\tilde{u}(t) \to u^*
&& \text{ in } && L^2_x\left(\abs{(r,z)-(r_{max},z_{max})} \geq R\right)
&& \text{ as } && t \rightarrow T_{max}.
\end{aligned}\end{equation}
Then to establish equation (\ref{Thm-MainResult-L2}) we will prove $u* \in L^2(\real^{3})$,
\begin{equation}\label{Proof-MainResult-eqn3.2}\begin{aligned}
u^*\in L^2(\real^{3})
&& \text{ and }
&&\int{\abs{u^*}^2} = \lim_{t\rightarrow T_{max}}\int{\abs{\tilde{u}(t)}^2}.
\end{aligned}\end{equation} 
Let $\epsilon_0 > 0$ be arbitrary. Due to equation (\ref{Eqn-epsIntegral}), we may choose $t(R)<T_{max}$ such that both,
\begin{equation}\label{Proof-MainResult-eqn4}\begin{aligned}
T_{max}-t(R) < \frac{\epsilon_0}{1+C\left(\frac{R}{4}\right)} 
&&\text{ and }
&&\int_{t(R)}^{T_{max}}{\int{\abs{\grad\widetilde{u}}^2\,dx}\,dt} < \epsilon_0,
\end{aligned}\end{equation}
where $C(\frac{R}{4})$ is the constant from equation (\ref{Proof-MainResult-SingOnRing}). We may assume that, for $t\in[t(R),T_{max})$, $u(t) = \tilde{u}$ on $\left\{\abs{(r,z)-(r_{max},z_{max})} > \frac{R}{4}\right\}$.
For parameter $\tau > 0$, to be fixed later, we denote, 
\begin{equation}\label{Proof-MainResult-DefnEqn-vTau}
v^\tau(t,x) = u(t+\tau,x) - u(t,x).
\end{equation}
Since $t(R) < T_{max}$, $u(t)$ is strongly continuous in $L^2$ at time $t(R)$. Thus, there exists $\tau_0$ such that,
\begin{equation}\label{Proof-MainResult-eqn5}\begin{aligned}
\int{\abs{v^\tau(t(R))}^2\,dx} < \epsilon_0
&& \text{ for all } \tau \in [0,\tau_0].
\end{aligned}\end{equation}
Denote $\phi_R$ a smooth cutoff function analogous to $\phi_\infty$ of equation (\ref{DefnEqn-phiInfty}),
\begin{equation}\label{Proof-MainResult-DefnEqn-phiR}
\phi_R(r,z) = \phi_\infty^4\left(\frac{(r,z) - (r_{max},z_{max})}{R}\right).
\end{equation}
By direct calculation,
\begin{equation}\label{Proof-MainResult-eqn7}\begin{aligned}
\frac{1}{2}\partial_t\left(\int{\phi_R\abs{v^\tau}^2}\right)
	=& Im\left(\int{\grad\phi_R\cdot\grad v^\tau \overline{v^\tau}\,dx}\right)\\
	&+ Im\left(\int{\phi_R v^\tau\left(\overline{u\abs{u}^2(t+\tau) - u\abs{u}^2(t)}\right)\,dx}\right).
\end{aligned}\end{equation}
Regarding the first RH term of (\ref{Proof-MainResult-eqn7}), from H\"older and our choice of $t(R)$ we have that,
\begin{equation}\label{Proof-MainResult-eqn8}
\int_{t(R)}^{T_{max}}\abs{
	Im\left(\int{\grad\phi_R\cdot\grad v^\tau \overline{v^\tau}\,dx}\right)
	\,dt}
\leq C \left(\int_{t(R)}^{T_{max}}{1^2\,dt}\right)^\frac{1}{2}\epsilon_0^\frac{1}{2}
< C\epsilon_0.
\end{equation}
Regarding the second RHS term of (\ref{Proof-MainResult-eqn7}), by homogeneity,
\begin{equation}\label{Proof-MainResult-eqn9}
\abs{\phi_R v^\tau\left(\overline{u\abs{u}^2(t+\tau) - u\abs{u}^2(t)}\right)}
\leq C\left(\abs{\phi_R^\frac{1}{4}u(t+\tau)}^4 + \abs{\phi_R^\frac{1}{4}u(t)}^4\right).
\end{equation}
Then, as we did in proving estimate (\ref{Eqn-4thOrderEst}), apply the Sobolev embedding $H^{\frac{3}{4}} \hookrightarrow L^4(\real^{3})$ and interpolate, 
$\int{\abs{\phi_R^\frac{1}{4}u}^4} \leq C\norm{\phi_R^\frac{1}{4}u}_{H^\frac{1}{2}}^2
		\norm{\phi_R^\frac{1}{4}u}_{H^1}^2$.  
By the uniform control of $H^\frac{1}{2}$, equation (\ref{Proof-MainResult-SingOnRing}), and our choice of $t(R)$,
\begin{equation}\label{Proof-MainResult-eqn10}
\int_{t(R)}^{T_{max}}{\abs{\phi_R v^\tau\left(\overline{u\abs{u}^2(t+\tau) - u\abs{u}^2(t)}\right)}\,dt}
\leq C\epsilon_0.
\end{equation}
Through the integration of equation (\ref{Proof-MainResult-eqn7}) we have proven,
\begin{equation}\label{Proof-MainResult-eqn6}\begin{aligned}
\int{\phi_R\abs{v^\tau(t)}^2\,dx} < C \epsilon_0
&& \text{ for all } \tau \in [0,\tau_0]
\text{ and } t\in[t(R),T_{max}-\tau).
\end{aligned}\end{equation}
This shows that $\widetilde{u}$ is Cauchy, which proves (\ref{Proof-MainResult-eqn3}). We now turn our attention to (\ref{Proof-MainResult-eqn3.2}). Denote the thickness of the toroidal support of the singular profile and radiation by,
\begin{equation}\label{Proof-MainResult-DefnEqn-Rt}
R(t) = A(t)\lambda(t).
\end{equation}
Recall the definition of $A(t)$, equation (\ref{DefnEqn-A}). By the log-log rate {\bf H1.3}, $\lambda(t) \to 0$ implies that $A(t) \to 0$ and in particular,
\begin{equation}\label{Proof-MainResult-BoundForA}
A(t) \leq \frac{1}{\abs{\log(T_{max}-t)}^C}.
\end{equation}
Consider now $\phi_{R(t),\tau} = \phi_{\infty}^4\left(\frac{(r,z)-(r(\tau),z(\tau))}{R(t)}\right)$, a family of time-variable cutoffs similar to $\phi_{R(t)}$. For fixed time $t<T_{max}$ we calculate directly that,
\begin{equation}\label{Proof-MainResult-eqn11}
\begin{aligned}
\frac{1}{2}\partial_\tau\left(\int{\phi_{R(t),\tau}\abs{u(\tau)}^2\,dx}\right)
	=&
\frac{1}{R(t)}Im\left(\int{\grad_x\phi_{R(t),\tau}\cdot\grad_xu(\tau)\overline{u(\tau)}\,dx}\right)\\
		&-\frac{1}{2R(t)}\int{\partial_\tau(r(\tau),z(\tau))
			\cdot\grad_x\phi_{R(t),\tau}\abs{u(\tau)}^2\,dx},
\end{aligned}\end{equation}
where we use $\grad_x\phi_{R(t),\tau}$ to denote $\left.\grad_y\phi_\infty^4(y)\right|_{y=\frac{(r,z)-(r(\tau),z(\tau))}{R(t)}}$. 
Regarding the first RH line of (\ref{Proof-MainResult-eqn11}),
\[
\abs{\frac{1}{R(t)}
		Im\left(\int{\grad_x\phi_{R(t),\tau}\cdot\grad_xu(\tau)\overline{u(\tau)}\,dx}\right)}
\lesssim \frac{1}{R(t)}\norm{u(\tau)}_{H^1} \lesssim \frac{1}{A(t)\lambda(t)\lambda(\tau)}.
\]
Regarding the second RH line of (\ref{Proof-MainResult-eqn11}),
apply the preliminary estimate (\ref{Eqn-prelimGamma+REst}),
\[
\abs{\frac{1}{2R(t)}\int{\partial_\tau(r(\tau),z(\tau)) 
		\cdot\grad_x\phi_{R(t),\tau}\abs{u(\tau)}^2\,dx}}
	\lesssim \frac{1}{A(t)\lambda(t)\lambda(\tau)}\int{\abs{u_0}^2}.
\]
Integrate (\ref{Proof-MainResult-eqn11}) in $\tau$, and apply the bounds for $A$ and $\lambda$ from equations (\ref{Proof-MainResult-BoundForA}) and (\ref{Proof-MainResult-BoundForLambda}),
\begin{multline}\label{Proof-MainResult-eqn12}
\abs{\int{\phi_{R(t),T_{max}}\abs{u^*}^2\,dx} - \int{\phi_{R(t),t}\abs{u(t)}^2\,dx}}\\
	\begin{aligned}
	&\leq C \frac{1}{A(t)\lambda(t)}\int_{t}^{T_{max}}{\frac{1}{\lambda(\tau)}\,d\tau}\\
	&\leq \frac{C}{\abs{\log(T_{max}-t)}^C}\left(\frac{\log\abs{\log(T_{max}-t)}}{T_{max}-t}\right)^\frac{1}{2}
		\int_t^{T_{max}}{\left(\frac{\log\abs{\log(T_{max}-\tau)}}{T_{max}-\tau}\right)^\frac{1}{2}\,d\tau}\\
	&\leq \frac{1}{\abs{\log(T_{max}-t)}^\frac{C}{2}}.
	\end{aligned}
\end{multline}
The final inequality relied upon $T_{max}-t < T_{max} < \alpha^*$, equation (\ref{Eqn-t1Small}), both to approximate the integral and then to approximate $C\log\abs{\log(T_{max}-t)} < \abs{\log(T_{max}-t)}^\frac{C}{2}$.
Taking the limit $t \rightarrow T_{max}$ we see that,
\begin{equation}\label{Proof-MainResult-eqn13}
\int{\abs{u^*}^2} = \lim_{t\to T_{max}}\int{\phi_{R(t),t}\abs{u(t)}^2}.  
\end{equation}
From the definition of $R(t)$ and the geometric decomposition,
\[
\lim_{t\to T_{max}}\int{\phi_{R(t),t}\abs{u(t)}^2}
=\lim_{t\to T_{max}}\int{\phi_{R(t),t}\abs{\widetilde{u}(t)}^2}
=\lim_{t\to T_{max}}\int{\abs{\widetilde{u}(t)}^2},
\]
which proves that the limit in (\ref{Proof-MainResult-eqn13}) exists and establishes equation (\ref{Proof-MainResult-eqn3.2}). This completes the proof of equation (\ref{Thm-MainResult-L2}).
\end{proof}
\begin{remark}[Consistency with $u^*\notin H^1$]
\label{Remark-ProfileNonH1}
By repeating the proof of {\bf I2.3}, we expect that following the proof of (\ref{Proof-MainResult-eqn3}) it could be shown that $\tilde{u}(t) \to u^*$ in 
$H^1\left(\abs{(r,z)-(r_{max},z_{max})}\geq R\right)$.
Nevertheless, an attempt to prove a version of (\ref{Proof-MainResult-eqn3.2}) in $H^1$ will fail. Indeed, the second RH line of equation (\ref{Proof-MainResult-eqn11}) would require a bound for $\abs{\grad u(\tau)}$ on the support of $\grad\phi$, with
nothing to take the role mass conservation. 
\end{remark}

\appendix

\section{Appendix}

\label{Appendix}

\begin{proof}[Proof ${\mathcal P}$ is Nonempty]
Choose $r_0 = 1$, $z_0 = 0$, $b_0 > 0$ small enough to satisfy (\ref{DataP1-mass}), and $\lambda_0$ in the range of {\bf C1.3}. 
Fix some smooth $f(y)$, radially symmetric with support $\abs{y} \leq 2$, $\norm{f}_{H^3(\real^2)} \leq 1$ and $(f,Q) = 1$, such that
$\epsilon_0(y) = \nu f(y)$ satisfies orthogonality conditions (\ref{DataP1-orthog})
 for any $\nu\in\real$ to be determined.
One may explicitly calculate such an $f$ from $Q$.
With $\gamma_0 = 0$, we now find $\nu=\nu(b_0)$ to satisfy {\bf C1.4} 
 and the small-mass requirement of {\bf C1.2}. 

By the choice of $\lambda_0$, $\abs{(r,z)-(1,0)} < \frac{1}{3}$ on the support of $\widetilde{Q}_{b_0}(y)$, which includes the support of $f(y)$. After change of variables, we will expand $\mu_{\lambda_0,1}(y)$ as $1 + \lambda_0y_1$ so that,
\begin{equation}\label{Proof-DataP1-nonempty-ener}\begin{aligned}
\lambda_0^2\abs{E_0}
&=\abs{\frac{1}{2}\int{\abs{\grad_y\left(\widetilde{Q}_{b_0}+\nu f\right)}^2\mu_{\lambda_0,1}(y)\,dy}
	-\frac{1}{4}\int{\abs{\widetilde{Q}_{b_0}+\nu f}^4\mu_{\lambda_0,1}(y)\,dy}}\\
&\lesssim \abs{\frac{1}{2}\int{\abs{\grad_y\left(\widetilde{Q}_{b_0}+\nu f\right)}^2\,dy}
	-\frac{1}{4}\int{\abs{\widetilde{Q}_{b_0}+\nu f}^4\,dy}} + \lambda_0,
\end{aligned}
\end{equation}
a small correction from the two-dimensional energy. By direct calculation with equation (\ref{DefnEqn-Q}),
\begin{equation}\label{Proof-DataP1-nonempty-derivEner}
\left.
\frac{d}{d\nu}\left(
	\frac{1}{2}\int{\abs{\grad_y\left(Q+\nu f\right)}^2\,dy}
	-\frac{1}{4}\int{\abs{Q+\nu f}^4\,dy}
\right) 
\right|_{\nu=0}
= -\left(f,Q\right) = -1.
\end{equation}
Thus by the degenerate energy of $\widetilde{Q}_{b_0}$ 
there exists $\nu = \nu(b_0)$ of the order $\abs{\nu} \leq \Gamma_{b_0}^{1-C\eta}$ such that,
\begin{equation}\label{Proof-DataP1-nonempty-findZeroEner}
\frac{1}{2}\int{\abs{\grad_y\left(\widetilde{Q}_{b_0}+\nu f\right)}^2\,dy}
	-\frac{1}{4}\int{\abs{\widetilde{Q}_{b_0}+\nu f}^4\,dy} = 0.
\end{equation}
Note the choice $\nu = 0$ is impossible as the energy of $\widetilde{Q}_{b_0}$ alone is too large to satisfy {\bf C1.4}. 
Note with this choice of $\nu$ that {\bf C1.4} 
is satisfied. Indeed,
\begin{equation}\label{Proof-DataP1-nonempty-massOK}
\norm{\tilde{u}_0}_{L^2(\real^{3})}
=\abs{\nu}\left(\int{\abs{f(y)}^2\mu_{\lambda_0,1}(y)\,dy}\right)^\frac{1}{2} < \alpha^*,
\end{equation}
due to the norm and support of $f$, our choice of $\lambda_0$, and that $b_0$ is small. Next we show the momentum requirement of {\bf C1.4} 
is satisfied. Again from the choice of $\lambda_0$, the support of $\widetilde{Q}_{b_0} + \nu f$ lies well within $\abs{(r,z)-(1,0)}\leq\frac{1}{2}$, a region where $\grad_x\psi^{(x)}$ is constant - see definition (\ref{DefnEqn-psi-x}).
Furthermore, $\widetilde{Q}_b$ and $f$ are symmetric, so that we have,
\begin{multline}\label{Proof-DataP1-nonempty-moment}
\lambda_0\abs{Im\left(
	\int{\grad_x\psi^{(x)}\cdot\grad_xu_0\overline{u}_0}\right)}\\
\begin{aligned}
&=\abs{(1,1)\cdot Im\left(
	\int{
		\grad_y\left(\widetilde{Q}_{b_0} + \nu f\right)
		\left(\overline{\widetilde{Q}_{b_0}+\nu f}\right)
		\mu_{\lambda_0,1}(y)\,dy}\right)}\\
&=
	C\abs{Im\left(
	\int{\grad_y\left(\widetilde{Q}_{b_0} + \nu f\right)
		\left(\overline{\widetilde{Q}_{b_0}+\nu f}\right)
		\lambda_0y_1\,dy}\right)}
&\lesssim \lambda_0.\end{aligned}
\end{multline}
Finally, note that requirements {\bf C2.2}
 and {\bf C2.3} 
are automatic from the support of $f$.  Constant $C$ of {\bf C2.1} 
is due to Lemma \ref{Lemma-UnprovenProperty} and the choice of $\nu$.
\end{proof}

\begin{remark}[Relationship with the Classic Virial Argument]
\label{Remark-NegEnergyData} 
For data $u_0\in H^1$ with finite variance, due to the classic virial identity, a sufficient condition for blowup is,
\begin{equation}\label{Proof-RemarkNegEnergy-eqn1}
\left\lbrack Im\left(\int{x\overline{u_0}\grad u_0}\right)\right\rbrack^2
	> 2 \left\lbrack\norm{x u_0}_{L^2}^2\right\rbrack E(u_0).
\end{equation}
We remark that there exists $u_0\in{\mathcal P}$ for which condition (\ref{Proof-RemarkNegEnergy-eqn1}) fails.  From the construction above,
\[\begin{aligned}
Im\left(\int{x\overline{u_0}\grad u_0}\right)
	 = &\, Im\left(\int{y
		\left(\overline{\Qb}+\nu f\right)
		\grad_y\left(\Qb+\nu f\right)
			\mu_{\lambda_0,1}\,dy}\right)\\
	&\, +	\frac{(1,0)}{\lambda_0}\cdot Im\left(\int{
		\left(\overline{\Qb}+\nu f\right)
		\grad_y\left(\Qb+\nu f\right)
			\mu_{\lambda_0,1}\,dy}\right).
\end{aligned}\]
Then, from equation (\ref{Proof-DataP1-nonempty-moment}), we observe that the LHS of (\ref{Proof-RemarkNegEnergy-eqn1}) is bounded by a universal constant, for all $b_0$ and $\lambda_0$ sufficiently small. 
Regarding the RHS of (\ref{Proof-DataP1-nonempty-moment}), note that instead of making the choice of equation (\ref{Proof-DataP1-nonempty-findZeroEner}), we may adjust $\nu$ to the order of $\lambda_0$ to ensure the energy is positive, and of the order $\frac{1}{\lambda_0}$. \end{remark}

\bibliographystyle{plain}
\bibliography{../../Bibliography/Bibliography} 

\end{document}